\documentclass[a4paper,10pt]{amsart}

\usepackage[utf8]{inputenc}
\usepackage{color}
\usepackage{amssymb,esint,graphicx}
\usepackage{amsmath}
\usepackage{verbatim}
\usepackage{hyperref}
\usepackage{enumerate}
\usepackage{pdfpages}

\newtheorem{theorem}{Theorem}[section]

\newtheorem{lemma}[theorem]{Lemma}
\newtheorem{proposition}[theorem]{Proposition}
\newtheorem{corollary}[theorem]{Corollary}
\newtheorem{definition}[theorem]{Definition}

\theoremstyle{remark}
\newtheorem{remark}[theorem]{Remark}
\theoremstyle{definition}
\newtheorem{example}{Example}[section]

\numberwithin{equation}{section}

\newcommand{\R}{\ensuremath{\mathbb{R}}}
\newcommand{\N}{\ensuremath{\mathbb{N}}}

\newcommand{\Levy}{\ensuremath{\mathcal{L}}}
\newcommand{\Levyd}{\ensuremath{\overline{\mathcal{L}}}}

\newcommand{\ra}{\rightarrow}
\newcommand{\alp}{\alpha}
\newcommand{\veps}{\varepsilon}

\newcommand{\uu}{\hat{u}}
\newcommand{\vv}{\hat{v}}
\newcommand{\B}{B_\varepsilon^\mu}
\newcommand{\Bn}{B_{\varepsilon_n}^\mu}

\newcommand{\As}{{A}}
\newcommand{\Bs}{{B}}
\newcommand{\Cs}{{C}}
\newcommand{\Ss}{{S}}
\newcommand{\Gs}{{G}}

\newcommand{\dd}{\,\mathrm{d}}

\newcommand{\dell}{\partial}
\newcommand{\indikator}{\mathbf{1}_{|z|\leq 1}}

\newcommand{\lef}{\Big(}
\newcommand{\rig}{\Big)}

\DeclareMathOperator{\sgn}{\textup{sign}}

\DeclareMathOperator{\supp}{supp}

\begin{document}

\title[Nonlocal porous medium equations]{Uniqueness and
  properties of distributional solutions of nonlocal  
equations of porous medium type
}


\author[F.~del~Teso]{F\'elix del Teso}
\address[F. del Teso]{Basque Center for Applied Mathematics (BCAM)\\
Bilbao, Spain} 
\email[]{felix.delteso\@@{}uam.es}
\urladdr{http://www.bcamath.org/es/people/fdelteso}

\author[J.~Endal]{J\o rgen Endal}
\address[J. Endal]{Department of Mathematical Sciences\\
Norwegian University of Science and Technology (NTNU)\\
N-7491 Trondheim, Norway} 
\email[]{jorgen.endal\@@{}math.ntnu.no}
\urladdr{http://www.math.ntnu.no/\~{}jorgeen/}

\author[E. ~R.~Jakobsen]{Espen R. Jakobsen}
\address[E. R. Jakobsen]{Department of Mathematical Sciences\\
Norwegian University of Science and Technology (NTNU)\\
N-7491 Trondheim, Norway} 
\email[]{erj\@@{}math.ntnu.no}
\urladdr{http://www.math.ntnu.no/\~{}erj/}

\subjclass[2010]{35A02, 35B30, 35B35, 35B53, 35D30, 35J15, 35K59, 35K65, 35L65, 35R09, 35R11}

\keywords{uniqueness, distributional solutions, nonlinear degenerate
  diffusion, porous medium equation, Stefan problem, fractional
  Laplacian, nonlocal operators, existence, stability, local limits, continuous
  dependence, numerical approximation, convergence}


\begin{abstract}  
  We study the uniqueness,
  existence, and properties of bounded distributional solutions of the
  initial value problem for the anomalous diffusion equation
  $\partial_tu-\mathcal{L}^\mu [\varphi (u)]=0$. Here $\mathcal{L}^\mu$ can be
  any nonlocal symmetric degenerate elliptic operator including the
  fractional Laplacian and numerical discretizations of this
  operator. The function $\varphi:\mathbb{R} \to \mathbb{R}$ is only assumed to be
  continuous and 
  nondecreasing. The class of equations include nonlocal (generalized) porous
  medium equations, fast diffusion equations, and Stefan problems.
In addition to very general uniqueness and existence results, we
obtain stability, $L^1$-contraction, and a priori estimates. We also study local
limits, continuous dependence, and properties and convergence of a
numerical approximation of our equations.
\end{abstract}

\let\thefootnote\relax\footnote{\textcopyright 2016. This manuscript version is made available under the CC-BY-NC-ND 4.0 license http://creativecommons.org/licenses/by-nc-nd/4.0/.}

\maketitle


\section{Introduction}

In this paper, we obtain uniqueness, existence, and various other
properties for bounded distributional solutions of a class of
possibly degenerate nonlinear anomalous diffusion equations of the form:
\begin{align}\label{E}
\dell_t u -\Levy^\mu[\varphi(u)]&=0 &&\text{in}\quad Q_T:=\R^N\times(0,T)\\
u(x,0)&=u_0(x) &&\text{on}\quad \R^N \label{IC}
\end{align}
where $u=u(x,t)$ is the solution and $T>0$. The nonlinearity
$\varphi$ is  an arbitrary continuous nondecreasing
function, while the anomalous or nonlocal diffusion operator
$\Levy^\mu$ is defined for any $\psi \in  
C_\textup{c}^\infty(\R^N)$ as 
\begin{equation}\label{deflevy}
\Levy ^\mu [\psi](x)=\int_{\R^N\setminus\{0\} } \lef\psi(x+z)-\psi(x)-z\cdot D\psi(x)\indikator \rig\dd\mu(z),
\end{equation}
where $D$ is the gradient,  $\indikator$ a characteristic
function, and $\mu$ a nonnegative symmertic possibly singular measure
satisfying the L\'evy condition $\int |z|^2\wedge 1 \dd \mu(z)<\infty$. For the precise
assumptions, we refer to Section \ref{sec:mainresults}.

The class of nonlocal diffusion operators we consider coincide with
the generators of the symmetric pure-jump L\'evy processes
\cite{Ber96,App09,Sch03} 
like e.g. compound Poisson processes, CGMY processes in Finance, and symmetric
$s$-stable processes. Included are the
well-known fractional Laplacians $-(-\Delta)^{\frac{s}{2}}$ for
$s\in(0,2)$ (where $\dd\mu(z)=c_{N,s}\frac{\dd z}{|z|^{N+s}}$ for some
$c_{N,s}>0$ \cite{DrIm06, App09}), 
along with degenerate operators, and surprisingly,
numerical discretizations of these operators!  

In the language of \cite{Vaz07}, equation \eqref{E} is a generalized
porous medium equation. {On one hand,} since $\varphi$ is
only assumed to be continuous, the full range of porous medium and fast diffusion nonlinearities are included:
$\varphi(r)=r|r|^{m-1}$ for $m>0$. This is somehow optimal for power
nonlinearities since if $m<0$ (ultra fast diffusion), then not only
uniqueness, but also existence may fail \cite{BoSeVa16}.
 On the other hand, since $\varphi$ is only
assumed to be nondecreasing, it can be constant on sets of positive
measure and then equation \eqref{E} is strongly degenerate. This case
include Stefan type of problems, like e.g. when $c_1,c_2,T> 0$ and
$$\varphi(r)=\begin{cases}c_2 r,& r<0,\\ c_1(r-T)^+,& r\geq 0.\end{cases}$$
Many physical problems can be modelled by equations like
\eqref{E}. We mention flow in a porous medium of e.g.~oil, gas, and
groundwater, nonlinear heat transfer, and population dynamics. For more
information and examples, we refer to 
Chapter 2 and 21 in \cite{Vaz07} for local problems, and to
\cite{Woy01,MeKl04,App09,Vaz12,Vaz14} for nonlocal problems.

A key result in this paper is the  uniqueness result for bounded
distributional solutions of \eqref{E} and \eqref{IC}. 
Almost half of the paper is devoted to the proof
of this result. Once we have it, we prove a general
stability result, and then we obtain other properties like existence,
$L^1$-contraction, and many a priori estimates from more regular problems via
approximation and compactness arguments. As straightforward
applications of all of these estimates, we then obtain the following results:
(i) Convergence as $s\to2^-$ of distributional solutions of 
\begin{align}
\label{FPME}
\dell_t u+(-\Delta)^{\frac{s}{2}}\varphi(u)=0 \qquad \text{in}\qquad
Q_T, 
\end{align}
to distributional solutions of the local equation
\begin{align}
\label{GPME}
\dell_t u-\Delta\varphi(u)=0 \qquad \text{in}\qquad
Q_T; 
\end{align}
(ii) continuous dependence in $(m,s)\in(0,\infty)\times(0,2]$ for the
porous medium equation of \cite{PaQuRoVa12},
\begin{equation}
\label{FPME2}
\dell_t u+(-\Delta)^{\frac{s}{2}}u|u|^{m-1}=0 \qquad\text{in}\qquad Q_T,
\end{equation}
including for the first time also the fast diffusion range; and
(iii) convergence of  semi-discrete numerical
approximations of a class of equations including \eqref{E}
(cf. \eqref{limiteq} and \eqref{approx} in Section \ref{sec:num}).

The uniqueness result is hard to prove because of our very general
assumptions on the initial value problem combined with a very weak
solution concept -- merely bounded distributional solutions.
This combination means that many classical techniques do not
work:  Fourier techniques are 
hard to apply because the problem is nonlinear and the Fourier symbol of
$\Levy^\mu$ could be merely a bounded function, energy estimates do
not imply uniqueness because 
$\varphi$ is not strictly increasing, and $L^1$-contraction
arguments do not apply since we do not assume additional entropy
conditions (cf. e.g. \cite{AnMa10} for the local case), or
equivalently, additional regularity in time as in \cite{PaQuRoVa12}
(see the uniqueness result for  
so-called strong solutions). The (weighted)
  $L^1$-contraction argument 
for ordered solutions given in \cite{BoVa14} avoids these additional
conditions, but it cannot be adapted here since it strongly depends
on the equation being like \eqref{FPME2} with $0<m<1$ and $s\in(0,2)$. Finally, since our solutions are not assumed to have 
finite energy, the classical uniqueness argument of Oleinik
\cite{OK58} cannot be adapted either.
We refer to \cite{OK58,Vaz07} for the local case, and
the uniqueness argument for so-called weak solutions in \cite{PaQuRoVa12} for 
results in the nonlocal case.

For the local equation \eqref{GPME}, uniqueness for bounded
distributional solution was proven by Brezis and Crandall in
\cite{BrCr79} under similar assumptions on $\varphi$ and $u_0$. Their
argument is quite indirect and rely on a clever idea using resolvents
and their integral representations (fundamental solutions). In this
paper, we adapt such an approach to our nonlocal setting.
 But because
of the generality of 
our diffusion operators, we cannot 
rely on explicit fundamental solutions for our proofs. Instead, we
have to develop 
this part of the theory from scratch, using the equation and the regularity
that comes with our solutions concept to obtain the necessary
estimates. 
To do this, a key tool is to approximate the possibly singular
integral operator $\Levy^\mu$ by a bounded integral operator and then
 carefully pass  to the limit.
This proceedure, and
hence also the proof, is truly
nonlocal -- there is no similar approximation by local operators.
The proof necessarily becomes much more involved than in
 \cite{BrCr79}, and includes a number of approximations, a priori estimates,
 $L^1$-contraction estimates, comparison principles, compactness and
 regularity arguments. It also includes new Stroock-Varoupolous
 inequalities and a new Liouville type of result for nonlocal
 operators. Both our approach and intermediate results
 should be of independent interest. 

 Let us give the main references for the well-posedness of the Cauchy
 problems for \eqref{E} and \eqref{GPME}. We start with the local case
 \eqref{GPME}. In the linear case, when $\varphi(u)=u$, it is the
 classical heat equation, cf. e.g. \cite{Eva10}. When
 $\varphi(u)=u^m$, it is a porous medium equation, and a very complete
 theory can be found in \cite{Vaz07}. In the general case,
 \eqref{GPME} is a generalized porous medium equation (or filtration
 equation). We refer again to \cite{Vaz07}. Uniqueness of
 distributional solutions of this equation was proven in \cite{BrCr79}
 for bounded initial data and continuous, nondecreasing $\varphi$, and
 in \cite{HePi85} for locally integrable initial data,
 $\varphi(r)=r^m$ for $0<m<1$, and with regularity assumptions on
 $\dell_t u$.
 Some nonuniqueness results can be found in
 e.g. \cite{Tyc35,Vaz90}. In the 
 presence of convection, or if general $L^1$-contraction results are
 sought, then so-called entropy solutions are a useful tool to obtain
 well-posedness \cite{Kru70,Car99}. A very general well-posedness
 result which cover the case of merely continuous $\varphi$ can then
 be found in \cite{AnMa10}. 

In the nonlocal case, one linear special case of \eqref{E} is the
fractional heat equation $\dell_t
u+(-\Delta)^{\frac{s}{2}}u=0$ for $s\in(0,2)$.  As in
  the local case, 
the initial value problem has a classical solution $u(x,t)=(K_s(\cdot,t)\ast
u(\cdot,0))(x)$ for
$\mathcal{F}(K_s(\cdot,t))(\xi)=\textup{e}^{-|\xi|^s
  t}$. It is well-posed even for measure data and solutions growing at
infinity \cite{BaPeSoVa14,BoSiVa16}.
 The
fractional porous medium equations \eqref{FPME2} are examples of
nonlinear equations of the form \eqref{E}. In
\cite{PaQuRoVa11,PaQuRoVa12}, existence, uniqueness and a priori
estimates for \eqref{FPME2} are proven for so-called weak
$L^1$-energy solutions -- possibly unbounded solutions with finite
energy. In \cite{BoVa14} there are  existence and uniqueness
  results for minimal distributional solutions of \eqref{FPME2} with
  $0<m<1$ in weighted 
$L^1$-spaces (solutions can grow at infinity).
We also mention that logarithmic diffusion
($\varphi(u)=\log(1+u)$) is considered in \cite{PaQuRoVa14},
singular or ultra fast diffusions in \cite{BoSeVa16}, weighted 
equations with measure data in \cite{GrMuPu15},  and 
problems on bounded domains in \cite{BoSiVa14, BoVa15, BoVa16}. 
Energy solutions
of equations with a larger class of
  nonlinearities $\varphi$ and nonlocal operators $\Levy^\mu$ are
  studied in the recent paper \cite{PaQuRo15}. The
  authors obtain results on well-posedness, 
continuity/regularity, and long time asymptotics.
 The setting, solution concept, and techniques are different from
 ours. Their operators 
 $\Levy^\mu$ can have some $x$-dependence, but the (singular part)
 must be comparable to a fractional Laplacian
 (i.e. be nondegenerate). Initial data in $L^\infty\cap L^1$ is assumed
 for uniquenss. In the $x$-independent case their assumptions are 
 less general than ours, especially those for $\Levy^\mu$ and the
 regularity of the solutions.  Other types of equations of the form
 \eqref{E} can be found in \cite{AMRT10}. These equations involve
bounded diffusion operators that can be represented by nonsingular
integral operators of the form \eqref{deflevy}. Because of this, at
least the well-posedness is easier to handle in this case.

It should be clear from the previous discussion that even if our
uniqueness result is very general, it is usually not strictly
comparable to the other results. E.g. a price  to pay to
work with general $\varphi$ and a very weak solution concept, is that
solutions $u$ have to be bounded. Our method of proof also requires
that $u-u_0\in L^1(Q_T)$. For particular choices of
$\varphi$, these assumptions may not be optimal. E.g. if you change
the solution concept and assume finite energy, then there are
uniqueness results for unbounded solutions of \eqref{FPME2} in $L^1$
in \cite{PaQuRoVa11,PaQuRoVa12}. There are even uniquness results in
weighted $L^1$-spaces, see \cite{BoVa14}. Here the solutions are
allowed to grow at infinity, but the uniqueness result is weaker in
the sense that it only holds for {\em minimal} distributional
solutions.


There are other ways to generalize the porous medium equation to a
nonlocal setting. In \cite{BiKaMo10,CaVa11,StanTesoVazCRAS, BiImKa15, StTeVa15} the authors consider
a so-called porous medium equations with fractional pressure. These
equations are in a divergence form, and no uniqueness is known except
when $N=1$. 
Finally, we mention that in the presence of (nonlinear) convection,
additional entropy conditions are needed to have uniqueness as in the
local case. Nonuniqueness of distributional solutions is proven in
\cite{AlAn10}, and several well-posedness results for entropy
solutions are given in \cite{Ali07, CiJa11, EnJa14}. These latter
results requires $\varphi$ to be linear or locally Lipschitz and hence
do not apply to our case where $\varphi$ is merely continuous.

\subsection*{Outline}
In Section \ref{sec:mainresults} we state the assumptions and present and discuss our main
results. The proof of the uniqueness result is given in Section
\ref{proofofuniqueness}. This proof requires a number of results and
estimates for a resolvent equation -- an auxiliary elliptic equation -- and these are proven in Section \ref{ellipticsection}.
In Section \ref{sec:stabexapriori}, we prove the main stability and 
existence result, along with a number of a priori estimates. We then
apply these results to prove the convergence to the local case,
continuous dependence, and the properties and convergence of the
numerical scheme in Section \ref{sec:appl}. Finally, after Section
\ref{ellipticsection}, there is an
appendix with the proofs of some technical results.

\subsection*{Notation} 
For $x\in\mathbb{R}$,  $x^+:=\, \text{max}\{x,0\}$, $x^-:=(-x)^+$,
and $\sgn^+(x)$ is $+1$ for $x>0$ and $0$ for $x\leq0$. We let
$B(x,r)=\{y\in\mathbb{R}^d : |x-y|< r\}$, $\mathbf{1}_{A}(x)$ be $1$
for $x\in A\subset \R^N$ and $0$ 
otherwise, and $\text{supp}\,\psi$ be the
support of a function $\psi$. Derivatives are denoted by $'$,
$\frac{d}{dt}$, $\partial_t$, $\dell_{x_i}$, and $D\psi$ and $D^2\psi$ denote the $x$-gradient and Hessian matrix of $\psi$. Convolution is defined as
$f\ast g(x)=\left[f\ast g\right](x)=\int_{\R^N}f(x-y)g(y)\dd y$, and
$(f,g)=\int_{\R^N}fg\dd x$ whenever the integral is well-defined. If
$f,g\in L^2(\R^ N)$, we write $(f,g)_{L^2(\R^N)}$. 
The $L^2$-adjoint of an operator $\mathcal{T}$ is denoted by
$\mathcal{T}^*$, and the reader may check that $(\mathcal
L^{\mu})^*=\mathcal L^{\mu^*}$ (see below for the definition of $\mu^*$). 
 A modulus of continuity is a nonnegative function $\lambda(\veps)$ which is
continuous in $\veps$ with $\lambda(0)=0$. 
By a classical solution, we mean a solution such that the equation
holds pointwise everywhere. 

Function spaces: $C_0$, $C_\textup{b}$, $C_\textup{b}^\infty$ and $C_\textup{c}^\infty$ are 
spaces of continuous functions that are vanishing at infinity; bounded; bounded with bounded
derivatives of all orders; and smooth functions with compact
support respectively. $C([0,T]; L_\textup{loc}^1(\mathbb{R}^N))$ is the space of
measurable functions $\psi:\R^N\times[0,T]\to\R$ such that
(i) $\psi(\cdot,t)\in 
L_{\textup{loc}}^1(\mathbb{R}^N)$ for every $t\in[0,T]$; (ii) for all
compact $K\subset\R^N$, 
$\int_{K}|\psi(x,t)-\psi(x,s)|\dd x \to0$ when $t\to s\in[0,T]$; and (iii) $\|\psi\|_{C([0,T];L^1(K))}:={\textup{ess\,sup}}_{t\in[0,T]}\int_{K}|\psi(x,t)|\dd x<\infty$.

Measures: $\delta_a(x)$ denotes the delta measure centered at $a\in \R^N$.
Let $X\subset\R^N$ be open and $\mu$ a Borel measure on $X$. For $x\in X$ and
$\Omega\subset X$ Borel, we denote $\mu_x(\Omega)=\mu(\Omega+x)$ where
$\Omega+{x}=\{y+x: y\in \Omega\}$. Moreover, $\mu^*$ is
defined as $\mu^*(B)=\mu(-B)$ for all Borel sets $B$, and we say that
$\mu$ is symmetric if $\mu^*=\mu$. The support of a Borel measure
$\mu$ on is 
$$\supp{\mu}=\{x\in X: \mu(B(x,r)\cap X)>0\ \text{for all}\
r>0\}.$$
The Lebesgue measure of $\R^N$ is denoted by $\dd w$ if $w$ is a
generic variable on $\R^N$. Moreover, the tensor product $\dd
\mu(z)\dd w$ is a well-defined nonnegative Radon measure since $\mu$
is $\sigma$-finite (for more details, consult \cite[Section 2.1.2]{AlCiJa12}.)

For the rest of the paper, we fix two families of mollifiers
$\omega_\delta$, $\rho_\delta$ defined by
\begin{align}
\omega_\delta(\sigma):=\frac{1}{\delta^N}\omega\left(\frac{\sigma}{\delta}\right)
\label{mollifierspace}
\end{align}
for fixed $0\leq\omega\in C_\textup{c}^\infty(\R^N)$ satisfying $\text{supp}\,
\omega\subseteq \overline{B}(0,1)$, $\omega(\sigma)=\omega(-\sigma),$ 
$\int\omega=1$; and
\begin{align}
&\rho_\delta(\tau):=\frac{1}{\delta}\rho\left(\frac{\tau}{\delta}\right)
\label{mollifiertime}
\end{align} 
for fixed $0\leq\rho\in C_\textup{c}^\infty([0,T])$, $\textup{supp}\,\rho\subseteq
[-1,1]$, $\rho(\tau)=\rho(-\tau)$,
$\int\rho=1$. 
\section{The main results}
\label{sec:mainresults}

In this section, we present the main results: first of all uniqueness, and then stability, existence and a number of estimates for
the solutions of \eqref{E} and \eqref{IC}. As an application of 
our main results, we give compactness and continuous dependence
estimates. We introduce a semi-discrete numerical scheme
for even more general equations and show that convergence and other
properties easily follow from our previous results. Finally, we
establish a new existence result that also cover local diffusion
equations. 

 Throughout the paper we assume that
\begin{align}
&\varphi:\R\to\R\text{ is continuous and nondecreasing};
\tag{$\textup{A}_\varphi$}&
\label{phiassumption}\\
&u_0\in L^{\infty}(\mathbb{R}^N);
\tag{$\textup{A}_{u_0}$}&
\label{u_0assumption}\\
&\label{muassumption}\tag{$\textup{A}_{\mu}$} \mu \text{ is a nonnegative symmetric Radon measure on
}\R^N\setminus\{0\}
\text{ satisfying}
\nonumber\\ 
&\quad\int_{|z|\leq1}|z|^2\dd \mu(z)+\int_{|z|>1}1\dd
\mu(z)<\infty.\nonumber
\end{align}

\begin{remark}
\begin{enumerate}[(a)]
\item Without loss of generality, we can assume $\varphi(0)=0$ (by adding a constant to $\varphi$). 
\item A nonlocal operator defined by \eqref{deflevy} is a nonpositive
  operator (see Lemma \ref{fouriersymbol}).
\end{enumerate}
\end{remark}

We use the following definition of  distributional solutions of  \eqref{E}
and \eqref{IC}.

\begin{definition}\label{distsol} 
Let  $u_0\in L_\textup{loc}^1(\R^N)$
  and $u\in L_\textup{loc}^1(Q_T)$. Then
\begin{enumerate}[(a)]
\item
$u$ is a distributional solution of equation \eqref{E} if
\begin{equation*}
\dell_t u-\Levy^\mu[\varphi(u)]=0 \quad\text{in}\quad \mathcal{D}'(Q_T),
\end{equation*}
\item $u$ is a distributional solution of the initial condition \eqref{IC} if
\begin{equation*}
\underset{t\to0^+}{\textup{ess\,lim}}\int_{\R^N}u(x,t)\psi(x,t)\dd x=\int_{\R^N}u_0(x)\psi(x,0)\dd x\quad\forall\psi\in C_\textup{c}^\infty(\R^N\times[0,T)).
\end{equation*}
\end{enumerate}
\end{definition}
The equation in part (a) is
well-defined when e.g.~\eqref{phiassumption} and 
\eqref{muassumption} hold and $u\in  L^\infty(Q_T)$. Note as well that
the initial condition $u_0$ is assumed in the distributional sense
($u_0$ is a weak initial 
trace). See Lemma \ref{equivdistsol} below for an equivalent definition.

We state the main result
of this paper.

\begin{theorem}\label{uniqueness}
Assume \eqref{phiassumption} and \eqref{muassumption}. Let $u(x,t)$ and $\uu(x,t)$ satisfy 
\begin{equation}\label{cond1}
u,\uu \in L^{\infty}(Q_T),
\end{equation}
\begin{equation}\label{cond3}
u-\uu \in L^1(Q_T),
\end{equation}
\begin{equation}\label{cond2}
\dell_t u-\Levy^{\mu}[\varphi(u)]=\dell_t\uu-\Levy^{\mu}[\varphi(\uu)] \qquad\text{in}\qquad \mathcal{D}'(Q_T)
\end{equation}
\begin{equation}\label{cond4}
\underset{t\to0^+}{\textup{ess\,lim}}\int_{\R^N}(u(x,t)-\uu(x,t))\psi(x,t)\dd x=0 \quad\text{for all}\quad \psi\in C_\textup{c}^\infty(\R^N\times[0,T)).
\end{equation}
Then $u=\uu$ a.e. in $Q_T$.
\end{theorem}
Sections \ref{proofofuniqueness} and \ref{ellipticsection} are devoted
to the (long) proof of this result. 

\begin{corollary}[Uniqueness]\label{uniquenessresult}
Assume \eqref{phiassumption}, \eqref{u_0assumption} and \eqref{muassumption}.  
Then there is at most one distributional solution $u$ of \eqref{E} and
\eqref{IC} such that $u\in L^\infty(Q_T)$ and $u-u_0\in L^1(Q_T)$.
\end{corollary}

\begin{proof} Assume  there are two solutions $u$
  and $\uu$. Then all assumptions of Theorem \ref{uniqueness}
  obviously hold ($\| u-\uu \|_{L^1}\leq\| u-u_0 \|_{L^1}+\| \uu-u_0
  \|_{L^1}<\infty$), and $u=\uu$ a.e.
\end{proof}

\begin{remark}
 Uniqueness holds for $u_0\not\in L^1$, for example
  $u_0(x)=c+\phi(x)$ for $c\in\R$ and $\phi\in L^\infty(\R^N)\cap L^1(\R^N)$. However, periodic
  $u_0$ are not included. In Section \ref{sec:rems} below we discuss some
  extensions of the uniqueness result.
\end{remark}

Next, we study under which assumptions solutions of  
\begin{equation}\label{En}
\dell_t u_n -\Levy^{\mu_n}[\varphi_n(u_n)]=0 \quad\text{in}\quad Q_T,
\end{equation}
converge to solutions of
\begin{equation}\label{limitprob}
\dell_t u -\Levy[\varphi(u)]=0 \quad\text{in}\quad Q_T.
\end{equation}

\begin{theorem}[Stability]\label{stability}
Assume $\Levy:C_\textup{c}^\infty(Q_T)\to L^1(Q_T)$,
$\mu_n$ satisfies \eqref{muassumption},
$\varphi_n$ and $\varphi$ satisfy \eqref{phiassumption}, and $u_n,u\in
L^\infty(Q_T)$ for every $n\in\N$.
Then if $\{u_n\}_{n\in\N}$ is a sequence of 
distributional solutions of
\eqref{En}, $\sup_n\|u_n\|_{L^\infty(Q_T)}<\infty$, and 
\begin{enumerate}[(i)]  
\item $\Levy^{\mu_n}[\psi]\to\Levy[\psi]$ in $L^1(\R^N)$ for all $ \psi \in C_\textup{c}^\infty(\R^N)$;
\smallskip
\item $\varphi_n\to\varphi$ locally uniformly;
\smallskip
\item
$u_n\to u$ pointwise a.e. in $Q_T$;
\end{enumerate}
then $u$ is a distributional solution of \eqref{limitprob}.
\end{theorem}
This result is proven in Section \ref{sec:stabexapriori}.

\begin{remark}
 The limit operator $\Levy$ need not satisfy
  \eqref{muassumption}, we can recover any
  operator of the form
  $\Levy[\psi]=\mathrm{tr}[\sigma\sigma^TD^2\psi]+\Levy^{\mu}[\psi]$:
  the general form of the
  generator of a {\em symmetric L\'evy process} \cite{App09}.
See sections \ref{sec:num} and \ref{sec:pfconvnum} 
  for more details and examples. An extension of this result will be
  discussed in Section \ref{sec:rems} below. 
\end{remark}

The stability result will be used along with approximation and
compactness arguments to obtain the following existence result and a
priori estimates.
\begin{theorem}[Existence and uniqueness]\label{existenceparabolic}
Assume \eqref{phiassumption}, $\eqref{muassumption}$, and $u_0\in
L^\infty(\R^N)\cap L^1(\R^N)$. Then there exists a unique
distributional solution $u$  of \eqref{E} and \eqref{IC} satisfying 
$$u\in L^\infty(Q_T)\cap L^1(Q_T)\cap C([0,T];L_\textup{loc}^1(\R^N)).$$ 
\end{theorem}

\begin{remark}
Existence results for merely bounded (and more general) initial
data can be found in Theorem 3.1 in \cite{BoVa14} in the setting of the
fractional porous medium equation \eqref{FPME2} with $0<m<1$.
\end{remark}

\begin{theorem}[A priori estimates]\label{consequences}
Assume \eqref{phiassumption}, $\eqref{muassumption}$, $u_0, \uu_0\in
L^\infty(\R^N)\cap L^1(\R^N)$. Let $u, \uu$ be the distributional
solutions of \eqref{E} with initial data $u_0,
\uu_0$ in the sense of Definition \ref{distsol} (b), respectively. Then
\begin{enumerate}[(a)]
\item \textup{($L^1$-contraction)} $\int_{\R^N}(u(x,t)-\uu(x,t))^+\dd
  x \leq \int_{\R^N}(u_0(x)-\uu_0(x))^+\dd x$, $t\in[0,T]$;
  
 \smallskip
\item \textup{(Comparison principle)} If $u_0\leq \uu_0$ a.e.~in $\R^N$, then $u\leq \uu$ a.e.~in $Q_T$;

\smallskip
\item \textup{($L^1$-bound)}
  $\|u(\cdot,t)\|_{L^1(\R^N)}\leq\|u_0\|_{L^1(\R^N)}$, $t\in[0,T]$;\medskip
\item \textup{($L^\infty$-bound)}
  $\|u(\cdot,t)\|_{L^\infty(\R^N)}\leq\|u_0\|_{L^\infty(\R^N)}$,
  $t\in[0,T]$;
  
  \smallskip 
\item \textup{(Time regularity)} For every $t, s\in[0,T]$ and compact
  set $K\subset \R^N$, 
$$
\|u(\cdot,t)-u(\cdot,s)\|_{L^1(K)}\leq \lambda_{u_0}\left(|t-s|^{\frac{1}{3}}\right)+C_{K,\varphi,u_0,\mu}\left(|t-s|^{\frac{1}{3}}+|t-s|\right)
$$
where
$\lambda_{u_0}(\delta)=\max_{|\sigma|\leq\delta}\|u_0-u_0(\cdot+\sigma)\|_{L^1(\R^N)},$  
$|K|$ is the Lebesgue measure of $K$, and for some
  constant $C$ independent of  $K$, $\varphi$, $u_0$, and $\mu$,
\begin{equation*}
C_{K,\varphi,u_0,\mu}= C|K|\Big(\sup_{|r|\leq \|u_0\|_{L^\infty}}|\varphi(r)|+1\Big)\int_{|z|>0}\min\{|z|^2,1\}\dd \mu(z).
\end{equation*}

\item \textup{(Mass conservation)}  If, in addition, there exists
  $L, \delta>0$ such that $|\varphi(r)|\leq L |r|$ for  $|r|\leq
  \delta$, then 
$$
\int_{\R^N}u(x,t)\dd x=\int_{\R^N}u_0(x)\dd x,\quad t\in[0,T].
$$
\end{enumerate}
\end{theorem}

These results are proven in Section \ref{sec:stabexapriori}.

\begin{remark}
The condition $|\varphi(r)|\leq L |r|$ in Theorem \ref{consequences}
(f) is sharp in the following sense: If $\varphi(r)=r^{m}$ for any
$m<1$, then there is $\Levy^\mu=-(-\Delta)^\frac{s}{2}$ such
that positive solutions $u$ of \eqref{E} and \eqref{IC} has extinction in
finite time and hence $\int u\neq \int u_0$. Simply take $N\in\N$ and
$s\in(0,2)$ such that $m\leq \frac{(N-s)^+}{N}$: see \cite{PaQuRoVa12}
for the details. 
\end{remark}

We now present several applications of the previous results.

\subsection{Application  1: Compactness, local limits, continuous dependence} 
\label{sec:rescompconv}
We start by a compactness and convergence result for very general
approximations of \eqref{E} and \eqref{IC}. 
\begin{theorem}[Compactness and convergence]
\label{compactness}
Assume $\Levy:C_\textup{c}^\infty(Q_T)\to L^1(Q_T)$,
$\mu_n$ satisfies \eqref{muassumption},
$\varphi_n$ and $\varphi$ satisfy \eqref{phiassumption}, and
$u_{0,n}\in L^\infty(\R^N)\cap L^1(\R^N)$ for every $n\in\N$.
Then if $\{u_n\}_{n\in\N}$ is a sequence of 
distributional solutions of
\eqref{En} with initial data $\{u_{0,n}\}_{n\in\N}$ in the sense of
Definition \ref{distsol} (b), and
\begin{enumerate}[(i)]
\item $\sup_n\int_{|z|>0}\min\{|z|^2,1\}\dd \mu_n(z)<\infty$;
\smallskip
\item $\sup_n\|u_{0,n}\|_{L^\infty(\R^N)}<\infty$;
\smallskip
\item $\Levy^{\mu_n}[\psi]\to\Levy[\psi]$ in $L^1(\R^N)$ for
  all $ \psi \in C_\textup{c}^\infty(\R^N)$;
\smallskip
\item $\varphi_n\to\varphi$ locally uniformly;
\smallskip
\item $u_{0,n}\to u_0$ in $L_\textup{loc}^1(\R^N)$. 
\end{enumerate}
Then
\begin{enumerate}[(a)]
\item there exist a subsequence $\{u_{n_j}\}_{j\in \N}$ and a $u\in C([0,T];L_\textup{loc}^1(\R^N))$ such that
$$u_{n_j}\to u \qquad\text{in}\qquad C([0,T];L_\textup{loc}^1(\R^N))\qquad\text{as}\qquad j\to\infty;$$
\item the limit $u$ from part (a) is a distributional solution of
  \eqref{limitprob} and \eqref{IC}.
\end{enumerate}
\end{theorem}
The proof can be found in Section \ref{sec:compconv}.
Using this result, we study the case
$\Levy^\mu=-(-\Delta)^{\frac s2}$, $s\in(0,2)$. As expected, we find
that solutions of the fractional equation \eqref{FPME} converge as
$s\to2^-$ to the solution of the local equation \eqref{GPME}.
Then we obtain a new result about continuous dependence on $(m,s)$ for the
porous medium equation of \cite{PaQuRoVa12}, that is, equation \eqref{FPME2}.

\begin{corollary}\label{convtothelocalprob}
Assume \eqref{phiassumption} and $u_0\in L^\infty(\R^N)\cap L^1(\R^N)$.
\begin{enumerate}[(a)]
\smallskip
\item The distributional solution $u_s$ of \eqref{FPME} and
  \eqref{IC}, converges in $C([0,T];
  L_\textup{loc}^1(\R^N))$ as $s\to2^-$ to a function $u$, and $u$ is a
  distributional solution of \eqref{GPME} and \eqref{IC}.\!\!\!\!\!\!
\medskip
\item Let $u_{n}$ and $\bar u$ be distributional solutions of \eqref{FPME2} and
  \eqref{IC} with $(m,s)=(m_n,s_n)$ and $(m,s)=(\bar m,\bar s)$
  respectively. If 
$$(0,\infty)\times(0,2)\ni(m_n,s_n)\ \longrightarrow\ 
  (\bar m,\bar s)\in(0,\infty)\times(0,2],$$ 
then $u_{n}\to \bar u$ in $C([0,T];L_\textup{loc}^1(\R^N))$.
\end{enumerate}
\end{corollary}
The proof of this result can also be found in section \ref{sec:compconv}.

\begin{remark}
When $u_0\in L^1(\R^N)$, the authors of \cite{PaQuRoVa12} show continuous
  dependence in $C([0,T];L^1(\R^N))$ for \eqref{FPME2} and \eqref{IC}
  for  $(m,s)\in
  \Big(\frac{(N-s)^+}N,\infty\Big)\times (0,2]$.  When $m\leq\frac{(N-s)^+}N$, we
  are in the fast 
  diffusion range and
  Corollary
  \ref{convtothelocalprob} (b) provides the first continuous dependence
  result for this case.
\end{remark} 

\subsection{Application 2:  
Numerical approximation, convergence, existence}
\label{sec:num}
Surprisingly, our class of operators $\Levy^\mu$ is so wide that it
contains a lot of its own numerical 
discretizations! It even contains common discretizations of local
operators as well. We illustrate this by giving one such discretization,
 a basic and very natural one, and then analyzing the resulting
 semidiscrete numerical method for \eqref{E}, or rather
 \eqref{limiteq}. We prove that it satisfies many 
 properties including convergence, and conclude a second and more general
 existence result. Consider 
\begin{equation}\label{limiteq}
\dell_t
u-\left(L^\sigma+\Levy^\mu\right)[\varphi(u)]=0\qquad\text{in}\qquad Q_T,
\end{equation}
where $\Levy^\mu$ is defined as before and $L^\sigma$ is a possibly degenerate local operator
\begin{equation*}
L^\sigma[\psi](x):=\text{tr}\big[\sigma\sigma^TD^2\psi(x)\big]
\end{equation*}
where $\sigma=(\sigma_1,....,\sigma_P)\in\R^{N\times P}$, $P\in \N$,
and $\sigma_i\in \R^N$. Note that $L^\sigma+\Levy^\mu$ is the generator of a {\em symmetric} L\'evy process, and conversely, any {\em symmetric} L\'evy processes has a generator like $L^\sigma+\Levy^\mu$ (cf. \cite{App09}). Moreover, equation \eqref{E} and \eqref{GPME} are special
cases of \eqref{limiteq} since $\sigma$ and $\mu$ may be degenerate or even zero. 

For any $h>0$, we approximate \eqref{limiteq} in the following way,
\begin{equation}\label{approx}
\dell_tu_h-\left(L_h^\sigma + \Levy_h^\mu\right)[\varphi(u_h)]=0 \qquad\text{in}\qquad Q_T.
\end{equation}
where 
\begin{align}\label{AproxLoc}
L^\sigma_h[\psi](x)&:=\sum_{i=1}^P \frac{\psi(x+\sigma_i h) + \psi(x-\sigma_i h)-2\psi(x)}{h^2},\\
\label{AproxOp}
\Levy^\mu_h[\psi](x)&:=
\sum_{\alpha\not=0}\left(\psi(x+z_\alpha)-\psi(x)\right)\mu\left(z_\alpha+R_h\right),
\end{align}
and $z_\alpha=h\alpha$, $\alpha=(\alpha_1,...,\alpha_N)\in \mathbb{Z}^N$, 
$R_h=\frac{h}{2}[-1,1)^N$. This is a finite difference approximation of $L^\sigma$
and quadrature approximation of $\Levy^\mu$.

\begin{remark}
\begin{enumerate}[(a)]
\item When $\sigma=e_i$, a standard basis vector of $\R^N$, then
  $L^{e_i}=\frac{\partial_i^2}{\partial x_i^2}$ and
  $L^{e_i}_h\psi(x)=\frac{\psi(x+he_i)-2\psi(x)+\psi(x-he_i)}{h^2}$:
  a classical finite difference approximation.\!\! 
\smallskip
\item Both $L_h^\sigma$ and $\Levy_h^\mu$ are in form \eqref{deflevy}
  and satisfy \eqref{muassumption}:
  cf. Lemma \ref{discretelocal} and \ref{discretenonlocal}.
\smallskip 
\item $L^\sigma \psi(x)=\sum_{i=1}^P \sigma_i^T D^2\psi(x) \sigma_i
  =\sum_{i=1}^P (\sigma_i^TD)^2\psi(x)\approx L^\sigma_h\psi(x)$.
\smallskip
\item $\Levy^\mu[\psi](x)=\sum_{\alpha \in \mathbb{Z^N}}\int_{z_\alpha+R_h}
\psi(x+z)-\psi(x) \dd \mu(z)\approx \Levy^\mu_h[\psi](x).$
\smallskip
\item To avoid $\mu(R_h)$ which may be
infinite, we do not sum over $\alpha=0$ in $\Levy^\mu_h$.
\end{enumerate}
\end{remark}

We now show that the scheme has many good properties, including
convergence.

\begin{proposition}[Properties of approximation]\label{semi-discreteapproxproperties}
Assume \eqref{phiassumption}, \eqref{muassumption}, $\sigma\in
\R^{N\times P}$, $u_0, \uu_0\in L^\infty(\R^N)\cap L^1(\R^N)$, and
$h>0$. 
\begin{enumerate}[(a)]
\item \textup{(Existence and uniqueness)} There exists a unique
  distributional solution $u_h\in L^\infty(Q_T)\cap L^1(Q_T)\cap
  C([0,T];L_\textup{loc}^1(\R^N))$ of \eqref{approx} and \eqref{IC}.
\smallskip 
\item \textup{($L^p$-stable)} $\|u_h(\cdot,t)\|_{L^p(\R^N)}\leq
  \|u_0\|_{L^\infty(\R^N)}^{\frac{p-1}{p}}\|u_0\|_{L^1(\R^N)}^{\frac{1}{p}}$,
   $\quad p\in[1,\infty]$, $t\in[0,T]$.
\smallskip
\item \textup{($L^1$-consistent)} For all $\psi\in C_\textup{c}^\infty(\R^N)$ 
$$\|\left(L_h^\sigma + \Levy_h^\mu\right)[\psi]-\left(L^\sigma +
  \Levy^\mu\right)[\psi]\|_{L^1(\R^N)}\to0\qquad\text{as}\qquad
h\to0^+.$$
\item \textup{(Monotone)} If $u_0\leq \uu_0$ a.e. in $\R^N$, then
  $u_h\leq\uu_h$ a.e. in $Q_T$.
\smallskip
\item \textup{(Conservative)} If in addition, there exists $\delta,L>0$ such that $|\varphi(r)|\leq L |r|$ for  $|r|\leq \delta$, then for all $t\in[0,T]$
$$
\int_{\R^N}u_h(x,t)\dd x=\int_{\R^N}u_0(x)\dd x.
$$
\end{enumerate}
\end{proposition}

\begin{proposition}[Compactness of approximation]\label{semi-discreteapprox}
Assume \eqref{phiassumption}, \eqref{muassumption}, $\sigma\in
\R^{N\times P}$, $u_0\in L^\infty(\R^N)\cap L^1(\R^N)$, and
$h>0$. Then there is subsequence of distributional solutions $u_h$ of
\eqref{approx} and \eqref{IC} that converges in $C([0,T];
L_\textup{loc}^1(\R^N))$ as $h\to0^+$ to some function $u$. Moreover, $u\in L^\infty(Q_T)\cap L^1(Q_T)\cap
C([0,T];L_\textup{loc}^1(\R^N))$ and $u$ is
a distributional solution of \eqref{limiteq} and \eqref{IC}.
\end{proposition}
Note that Proposition
\ref{semi-discreteapprox} also provide a new existence result:

\begin{corollary}[Existence for
  \eqref{limiteq}]\label{generallevyprocessexistence} 
Under the assumptions of Proposition \ref{semi-discreteapprox}, there
exists a distributional solution $u\in L^\infty(Q_T)\cap L^1(Q_T)\cap
C([0,T];L_\textup{loc}^1(\R^N))$ of \eqref{limiteq} and \eqref{IC}.
\end{corollary}

In many cases we can combine the compactness result with uniqueness
results for the limit equations, and hence obtain convergence for the
approximation. 

\begin{theorem}[Convergence of approximation]\label{semi-discreteapprox2}
Under the assumptions of Proposition \ref{semi-discreteapprox}, and if
in addition either $\sigma\equiv0$ or $\mu\equiv 0$ and $\sigma=I$ (the identity
matrix), then the distributional solutions $u_h$ of
\eqref{approx} and \eqref{IC} converges in $C([0,T];
L_\textup{loc}^1(\R^N))$ as $h\to0^+$ to the unique distributional solution $u\in
L^\infty(Q_T)\cap L^1(Q_T)\cap C([0,T];L^1_{\mathrm{loc}}(\R^N))$ of
\eqref{limiteq} and \eqref{IC}. 
\end{theorem}

The proofs will be given in Section \ref{sec:pfconvnum}.

\begin{remark}
\begin{enumerate}[(a)]
\item Our approximation is well-defined and converge for {\em any} problem
of the type \eqref{limiteq}, including strongly degenerate
Stefan problems and fast diffusion equations.
The scheme and convergence result thus cover cases that have not
been considered before in the literature. For nonlocal problems of
this type, there are very few results, and only for locally Lipschitz $\varphi$
\cite{TeVa13, CiJa14,Te14}. 
\smallskip
\item To obtain a fully discrete numerical method, it remains to (i)
  restrict the method to some spacial grid and (ii) discretize also
  in time. Time discretization is easier and leads to a problem that no
  longer has the form \eqref{E}; we will discuss it in a future work.
 Restriction to a spacial
  grid can always be done after a change of coordinate system: see
  Section \ref{sec:rems} below. 
\item The existence result is a result where existence for problems
  involving nonlocal operators $\Levy^\mu$ are exported to problems
  involving the ``closure'' of this class of operators -- namely, operators of the form
$L^\sigma+\Levy^\mu$. The proof is completely different from proofs
based on nonlinear semigroup theory; see e.g. Chp. 10 in
\cite{Vaz07}, and \cite{PaQuRoVa12}.
\end{enumerate}
\end{remark}

\subsection{Remarks and extensions}
\label{sec:rems}

\subsubsection*{Alternative definition of distributional solutions}
\begin{enumerate}[(1)]
\item A more compact form  that we will use in the
  proofs  is the following:
\begin{lemma}\label{equivdistsol}
 Assume \eqref{phiassumption}, \eqref{u_0assumption}, \eqref{muassumption}
  and $u\in  L^\infty(Q_T)$. Then $u$ is a distributional
  solution of \eqref{E} and \eqref{IC} if and only if
\begin{equation*}
\begin{split}
\int_0^T \int_{\R^N} \lef u(x,t) \dell_t\psi(x,t) + \varphi(u(x,t)) \Levy^{\mu} [\psi(\cdot,t)](x)\rig\dd x\dd t+\int_{\R^N}u_0(x)\psi(x,0)\dd x=0
\end{split}
\end{equation*}
for all $\psi\in C_\textup{c}^\infty(\R^N\times[0,T))$.
\end{lemma}
The easy and standard proof is omitted. 
\end{enumerate}

\subsubsection*{About the initial conditions}
\begin{enumerate}[(1)]
\setcounter{enumi}{1}
\item The solutions provided by Theorem \ref{existenceparabolic} belong to
  $C([0,T];L_\textup{loc}^1(\R^N))$ and hence satisfy the initial condition in the
  strong $L^1_{\textup{loc}}$-sense:  For all compact $K\subset\R^N$,
\begin{align*}
\int_K |u(x,t)-u_0(x)|\dd x\ra 0 \qquad \text{as}\qquad t\to0.
\end{align*}

\item If the initial conditions are satisfied in the strong
  $L^1_{\textup{loc}}$-sense, then they are of course also satisfied in the
  distributional sense of Definition \ref{distsol}.
\end{enumerate}

\subsubsection*{Extensions of the uniqueness result Corollary
  \ref{uniquenessresult}}
\begin{enumerate}[(1)]
\setcounter{enumi}{3}
\item With the same proof, we also get uniqueness for the
initial value problem for the
inhomogenenous equation
$$\dell_tu+\Levy^\mu[\varphi(u)]=g(x,t).$$
\item A close inspection of the proof reveals that we can replace continuity of
  $\varphi$ in \eqref{phiassumption} by continuity at zero, Borel
  measurability, and $\varphi(u)\in L^\infty(Q_T)$
  (cf. \cite{BrCr79}).
\end{enumerate}

\subsubsection*{Extensions of the stability result Theorem
  \ref{stability}}\label{extstab}\ 
\begin{enumerate}[(1)]
\setcounter{enumi}{5}
\item When $\varphi_n$ is independent of $n$, we only need weak
  convergence of $\Levy^{\mu_n}$ in (i):
$$\Levy^{\mu_n}[\psi]\to\Levy[\psi]\quad \text{{\em weakly} in
  $L^1(\R^N)$ for all $\psi\in C_\textup{c}^\infty(Q_T)$}.$$ 
Moreover, by considering subsequences we can replace 
  (iii) by $u_n\to u$ in $L_\textup{loc}^1(Q_T)$.
 These observations follow by slight changes in the proof of Theorem \ref{stability} in Section \ref{sec:stabexapriori}.
\item A general condition for $L^1$-weak convergence of $\Levy^{\mu_n}$
  \cite{ChJa15}:
  There exist $\sigma\in\R^{N\times P}$ and a nonnegative Radon
  measure $\mu$ such that for all $A\in\R^{N\times N}$
\smallskip
\begin{enumerate}[(i)]
\item $\sup_n\int_{|z|>0}\min\{|z|^2,1\}\dd \mu_n(z)<\infty;$
\smallskip
\item $\int_{|z|\leq1}zAz^T\dd
  \mu_n(z)\to\textup{tr}\left(\sigma\sigma^TA\right)+\int_{|z|\leq1}zAz^T\dd
  \mu(z);$
\smallskip
\item $\int_{|z|>1}\dd \mu_n(z)\to\int_{|z|>1}\dd\mu(z)$.
\end{enumerate}
\smallskip
Here $\Levy=\mathrm{tr}[\sigma\sigma^TD^2]+\Levy^{\mu}$: see
\cite{ChJa15} for a general discussion and more examples. 
\end{enumerate}

\subsubsection*{Defining the scheme \eqref{approx} on a grid}\label{remgrid}
\begin{enumerate}[(1)]
\setcounter{enumi}{7}
\item By a coordinate transformation $x=Ay$, $L^\sigma+\Levy^\mu$ can be
  transformed into
\begin{equation*}
L^{I_0}+\Levy^{\tilde\mu}\qquad\text{where}\qquad I_0:=\left[\begin{array}{c|c}
I & 0\\
\hline
0 & 0\\
\end{array}\right]\in \R^{N\times N},
\end{equation*}
$I$ is an identity matrix, and $\dd\tilde{\mu}(z)=\dd\mu(A^{-1}z)$
satisfies \eqref{muassumption}. Up to permutations of the components
of $y$, $A=QJ$ where $Q\in\R^{N\times N}$ is orthonormal, 
$Q\sigma\sigma^TQ^{T}=\textup{diag}(\lambda_i)$ for $\lambda_i\geq0$,
and $J=\textup{diag}(\sqrt{c_i})$ where $c_i=1$ if $\lambda_i=0$ and
$c_i=\frac{1}{\lambda_i}$ if $\lambda_i>0$ for $i=1,\ldots,N$.
\item For the new operator $L^{I_0}+\Levy^{\tilde\mu}$, our approximations
produce an operator $L^{I_0}_h+\Levy^{\tilde\mu}_h$ that can be restricted
to the ($y$-)grid $\mathcal{G}_h:=h\mathbb{Z}^N$ ($h>0$), that is
$L^{I_0}_h+\Levy^{\tilde\mu}_h:\R^{\mathcal{G}_h}\to \R^{\mathcal{G}_h}$ is well-defined.
\end{enumerate}


\section{The proof of uniqueness}
\label{proofofuniqueness}

\subsection{Preliminary results}
A crucial part in the proof is played by the following linear elliptic equation 
\begin{equation} \label{elliptic}
\veps v_\veps(x) - \Levy^\mu [v_\veps](x)=g(x) \quad\text{in}\quad \R^N,
\end{equation}
where $\veps>0$ and $\Levy^\mu$ defined by \eqref{deflevy}. Its
solutions will be denoted by  
$$\B[g](x):=v_\veps(x).$$
 Formally,
$B_\veps^\mu=(\veps I- \Levy^\mu)^{-1}$ is the resolvent of
$\Levy^{\mu}$. Note that $\Levy^\mu$ may be very 
degenerate and therefore Fourier techniques do not easily apply
(cf. Example \ref{ex:Liouville}  and Remark \ref{remarkfouriersymbol}
(a) below). 
The main results about equation \eqref{elliptic} are given below,
while most of the proofs will be given in Section \ref{ellipticsection}. Note that in \cite{BrCr79} such results are easy in view of an
explicit representation formula for $B_\veps^\mu$. Here, on the other
hand, they are not easy and we have to work quite a lot to prove
these estimates. The method of proof is different, more nonlocal, and
requires less of the operator.

\begin{theorem}[Classical and distributional solutions]\label{collectedresultselliptic}
Assume \eqref{muassumption} and $\veps>0$.
\begin{enumerate}[(a)]
\item If $g\in C_\textup{b}^\infty(\R^N)$, then there exists a unique classical
  solution $\B[g]\in C_\textup{b}^\infty(\R^N)$ of \eqref{elliptic}. Moreover, for each multiindex $\alpha\in \mathbb{N}^N$,
$$
\veps\|D^\alpha \B[g]\|_{L^\infty}\leq\|D^\alpha g\|_{L^\infty}.
$$
\item If $g\in L^1(\R^N)$, then there exists a unique distributional solution $\B[g]\in L^1(\R^N)$ of \eqref{elliptic}. Moreover,
\begin{equation*}
\veps\| \B[g]\|_{L^1(\R^N)}\leq \|g\|_{L^1(\R^N)}.
\end{equation*}
\item If $g\in L^\infty(\R^N)$, then there exists a unique distributional solution $\B[g]\in L^\infty(\R^N)$ of \eqref{elliptic}. Moreover,
\begin{equation*}
\veps\| \B[g]\|_{L^\infty(\R^N)}\leq \|g\|_{L^\infty(\R^N)}.
\end{equation*}
\end{enumerate}
\end{theorem}
\begin{remark} If $g\in L^1\cap L^\infty$, then
  $\veps\|\B[g]\|_{L^p}\leq\|g\|_{L^\infty}^{\frac{p-1}{p}}\|g\|^{\frac{1}{p}}_{L^1}$
  for any $p\in(1,\infty)$.
\end{remark}

When a smooth $g$ depends also on time, then $\B[g]$ will be smooth in time
and space. 
\begin{corollary}\label{propertiesofB2}
Assume \eqref{muassumption}, $\veps>0$, and $\gamma\in
C_\textup{c}^\infty(\R^N\times[0,T))$. Then  
\begin{enumerate}[(a)]
\item $\B[\gamma]\in C_\textup{b}^\infty(\R^N\times[0,T))$.
\item $\B[\gamma](x,\cdot)$ is compactly supported in $[0,T)$.
\item $\dell_t(\B[\gamma])=\B[\dell_t\gamma]$\quad and\quad $\B[\gamma], \B[\dell_t\gamma], \Levy^\mu \left[ \B[\gamma] \right] \in L^1(Q_T)$.
\end{enumerate}
\end{corollary}

\begin{proof}
(a) A standard
argument using difference quotients, linearity and uniqueness of the
problem, the $L^\infty$-bound of Theorem \ref{collectedresultselliptic} (a), and induction on $n$, gives
that 
\begin{align}
\label{ppp}
\dell_t^n D^\alpha \B[\gamma]=\B[\dell_t^n D^\alpha
\gamma]\qquad\text{in}\qquad Q_T
\end{align}
for every $n\in\N$ and $\alpha \in \mathbb{N}^N$. This argument is
almost exactly the same as the one given in the proof of Proposition
\ref{existapproxelliptic} (d) below.
Then by Theorem \ref{collectedresultselliptic} (a),
$$\veps\|\dell_t^n D^\alpha \B[\gamma]\|_{L^\infty(Q_T)}\leq
\|\dell_t^n D^\alp\gamma\|_{L^\infty(Q_T)}.$$

\smallskip
\noindent(b) Holds since $\B$ is an operator in the spatial variable
$x$ and $\B[0]=0$.

\smallskip
\noindent(c) Note that $\dell_t \B[\gamma]=\B[\dell_t \gamma]$ by \eqref{ppp}, and by Theorem \ref{collectedresultselliptic} (b) and the time continuity of $\gamma$ and $\B[\gamma]$,
$$
\veps\|\B[\gamma] \|_{L^1(Q_T)}\leq \|\gamma\|_{L^1(Q_T)},
$$ 
which is finite because $\gamma\in  C_\textup{c}^\infty(Q_T)$. Hence it follows that
$$
\veps\|\dell_t(\B[\gamma])\|_{L^1(Q_T)}=\veps\|\B[\dell_t\gamma]\|_{L^1(Q_T)}\leq\|\dell_t\gamma\|_{L^1(Q_T)},
$$ 
By equation \eqref{elliptic},
$\Levy^\mu[\B[\gamma]]=\veps\B[\gamma]-\gamma$ for all $(x,t)\in
Q_T$. Since both $\B[\gamma]$ and $\gamma$ are in $L^1(Q_T)$, it
follows that also $\Levy^\mu[\B[\gamma]]\in L^1(Q_T)$.
\end{proof}

The operator $\B$ is self-adjoint in the following sense:
\begin{lemma}\label{symmetryB}
Assume \eqref{muassumption}, $g \in L^\infty(\R^N)$, $f \in L^1(\R^N)$, and $\veps>0$. Then
\begin{equation*}
\int_{\R^N}\B[g](x)f(x) \dd x=\int_{\R^N}g(x)\B[f](x) \dd x.
\end{equation*}
\end{lemma}
The proof is given in section \ref{ellipticsection}.
To prove these and other results in this paper, we will need some
properties of the nonlocal operator $\Levy^\mu$ that are given below.
\begin{lemma}\label{propnonlocal}
Assume \eqref{muassumption}.
\begin{enumerate}[(a)]
\item If $\psi\in C^2(\R^N)\cap L^\infty(\R^N)$, then
\begin{equation*}
\begin{split}
|\Levy^\mu[\psi](x)|\leq\frac{1}{2}\max_{|z|\leq1}|D^2\psi(x+z)|\int_{|z|\leq1}|z|^2\dd \mu(z)+2\|\psi\|_{L^\infty(\R^N)}\int_{|z|>1}\dd\mu(z).
\end{split}
\end{equation*}
\item Let $p\in\{1,\infty\}$ be fixed. If $\psi\in W^{2,p}(\R^N)$, then
$$
\|\Levy^\mu[\psi]\|_{L^p(\R^N)}\leq \frac{1}{2}\|D^2\psi\|_{L^p(\R^N)}\int_{|z|\leq1}|z|^2\dd\mu(z)+2\|\psi\|_{L^p(\R^N)}\int_{|z|>1}\dd\mu(z).
$$
\item If $\psi_1\in W^{2,1}(\R^N)$ and $\psi_2\in W^{2,\infty}(\R^N)$, then
$$
\int_{\R^N}\psi_1\Levy^\mu[\psi_2]\dd x = \int_{\R^N}\Levy^\mu[\psi_1]\psi_2\dd x.
$$
\end{enumerate}
\end{lemma}

\begin{remark}\label{nonlocalwelldefined}
\begin{enumerate}[(a)]
\item If $\psi\in C^2(\R^N)\cap L^\infty(\R^N)$, then $\Levy^\mu[\psi](x)$
is well-defined by (a).

\item If $\mu(\R^N)<\infty$, a density argument and the symmetry of
$\mu$ reveals that
$$\Levy^{\mu}[\phi](x)=\int_{|z|>0}\lef\phi(x+z)-\phi(x)\rig \dd \mu(z),$$
and the assumptions of Lemma \ref{symmetryB} can be relaxed to
$g\in L^\infty(\R^N)$, $f\in L^p(\R^N)$ for $p\in\{1,\infty\}$,
and $\psi_1\in L^1(\R^N)$ and $\psi_2\in L^\infty(\R^N)$ respectively
in (a), (b), and (c). The second derivative part of the estimates in
(a) and (b) then have to be dropped and the remaining term modified
accordingly.
\end{enumerate}
\end{remark}

A proof of Lemma \ref{propnonlocal} can be found e.g. in Sections 1 and 4 in \cite{AlCiJa12}.

\begin{lemma}\label{fouriersymbol}
Assume \eqref{muassumption} and $\psi\in C_\textup{c}^\infty(\R^N)$. Then
$$
\mathcal{F}(\Levy^\mu[\psi])(\xi)=-\sigma_{\Levy^\mu}(\xi)\mathcal{F}(\psi)(\xi),
$$
where
$$
\sigma_{\Levy^\mu}(\xi):=\int_{|z|>0}1- \cos(z\cdot\xi)\dd\mu(z).
$$
Moreover, $\sigma_{\Levy^\mu}(\xi)\geq0$ and 
$$
\Big(\psi,\Levy^\mu[\psi]\Big)_{L^2(\R^N)}=-\left\| (\Levy^\mu)^{\frac{1}{2}}[\psi] \right\|_{L^2(\R^N)}^2.
$$
\end{lemma}

\begin{remark}\label{remarkfouriersymbol}
\begin{enumerate}[(a)]
\item $\sigma_{\Levy^\mu}$ is the Fourier symbol of $\Levy^\mu$. In
  our generality it may not be invertible or have any smoothing
  properties. An extreme example is $\mu=\delta_{z_0}$ for $z_0\not=0$, where
  $\sigma_{\Levy^\mu}(\xi)=1-\cos z_0\cdot\xi$; this is a bounded function
  with infinitly many zeros.
\item If $\psi, \Levy^\mu[\psi] \in L^2(\R^N)$, then a density argument
  shows that the Fourier symbol exists and the conclusions of Lemma
  \ref{fouriersymbol} still hold. 
 \item The notation $ (\Levy^\mu)^{\frac{1}{2}}$ is used to denote the square root of the operator $ \Levy^\mu$ in the Fourier transform sense. 
 \end{enumerate}
\end{remark}

\begin{proof}
By the definition of $\Levy^\mu$, Fubini's theorem, and the symmetry of $\mu$,
\begin{equation*}
\begin{split}
\mathcal{F}(\Levy^\mu[\psi])(\xi)=&\,(2\pi)^{-\frac{N}{2}}\int_{\R^N}\textup{e}^{-\textup{i}x\cdot\xi}\int_{|z|>0}\\
&\qquad\qquad\qquad\lef\psi(x+z)-\psi(x)-z\cdot D\psi(x)\indikator\rig \dd\mu(z)\dd x\\
=&\, \int_{|z|>0} \lef\textup{e}^{\textup{i} z\cdot \xi} \mathcal{F}(\psi)(\xi)-\mathcal{F}(\psi)(\xi)-\textup{i}z\cdot\xi\indikator  \mathcal{F}(\psi)(\xi)\rig\dd\mu(z)\\
=&\, \mathcal{F}(\psi)(\xi) \int_{|z|>0} \lef\cos(z\cdot\xi)-1\rig\dd\mu(z).
\end{split}
\end{equation*}

To show the second part of the lemma, note that $\sigma_{\Levy^\mu}\geq0$ and $\psi, \Levy^\mu[\psi]\in L^2(\R^N)$ (cf. Lemma \ref{propnonlocal} (b)). It follows that $\mathcal{F}(\psi), \sigma_{\Levy^\mu}\mathcal{F}(\psi)\in L^2(\R^N)$, and then by the inequality $2ab\leq a^2 + b^2$, $(\sigma_{\Levy^\mu})^\frac{1}{2}\mathcal{F}(\psi)\in L^2(\R^N)$. By Plancherel's theorem,
\begin{equation*}
\begin{split}
\Big(\psi,\Levy^\mu[\psi]\Big)_{L^2(\R^N)}=&\,\Big(\mathcal{F}(\psi),\mathcal{F}(\Levy^\mu[\psi])\Big)_{L^2(\R^N)}=\Big(\mathcal{F}(\psi),-\sigma_{\Levy^\mu}\mathcal{F}(\psi)\Big)_{L^2(\R^N)}\\
=&\, -\Big((\sigma_{\Levy^\mu})^{\frac{1}{2}}\mathcal{F}(\psi),(\sigma_{\Levy^\mu})^{\frac{1}{2}}\mathcal{F}(\psi)\Big)_{L^2(\R^N)}=-\left\| (\Levy^\mu)^{\frac{1}{2}}[\psi] \right\|_{L^2(\R^N)}^2, 
\end{split}
\end{equation*}
which completes the proof.
\end{proof}

The following theorem is a key technical tool in our uniqueness argument.
\begin{theorem}\label{Liouville}
Assume \eqref{muassumption} and $\supp \mu\neq \emptyset$. If $v\in C_0(\R^N)$ solves
$$
\Levy^{\mu}[v]=0\qquad\text{in}\qquad \mathcal{D}'(\R^N),
$$
then $v\equiv0$ for all $x\in\R^N$.
\end{theorem}

We give the proof of Theorem \ref{Liouville} in Appendix 
\ref{appendix}. In the local case \cite{BrCr79} such a result follows
for example from the Liouville theorem for the Laplacian. On one hand, our result
is much weaker since we need to ask for some kind of decay at
infinity. On the other hand, Theorem \ref{Liouville} covers very degenerate operators
$\Levy^{\mu}$ which do not satisfy any sort of Liouville theorem.   

\begin{example}
\label{ex:Liouville}
Let $\mu=\delta_{2\pi}+\delta_{-2\pi}$. Note that \eqref{muassumption} holds and that for smooth functions $v$,
\begin{equation*}
\Levy^{\mu}[v](x)=v(x+2\pi)-2v(x)+v(x-2\pi).
\end{equation*}
The function $v=\cos\in C_\textup{b}^\infty(\R)$ is an example of a
nonconstant  function that satisfies $\Levy^{\mu}[v](x)=0$ in $\R$, and
hence the Liouville theorem does not hold for $\Levy^{\mu}$.
\end{example}

\subsection{The proof of Theorem \ref{uniqueness}}
We define
$$U(x,t):=u(x,t)-\hat u(x,t)\qquad\text{and}\qquad
\Phi(x,t):=\varphi(u(x,t))-\varphi(\hat u(x,t)).$$ 
By the assumptions \eqref{cond1}, \eqref{cond3}, 
and \eqref{phiassumption}, 
$$U\in L^1(Q_T)\cap L^\infty(Q_T),\quad \Phi\in L^\infty(Q_T),$$
and by \eqref{cond2}, \eqref{cond4}, and Lemma \ref{equivdistsol}
\begin{equation}\label{diffparabolic}
\int_0^T\int_{\R^N}\lef U\dell_t\psi +\Phi\Levy^\mu[\psi]\rig\dd x \dd t=0\qquad\text{for all}\qquad \psi\in C_\textup{c}^\infty(\R^N\times[0,T)).
\end{equation}
We emphasize that this equation also incorporates a zero intitial
condition for $U$.

We now define the function $h_\veps(t)$ which will play the main role in the proof:
\begin{equation}\label{defh}
h_\veps(t):=\left(\B[U](\cdot,t),U(\cdot,t)\right)=\int_{\R^N}\B[U(\cdot,t)](x)U(x,t)\dd x.
\end{equation}
Note that $h_\veps\in L^1((0,T))$ since $\|h_\veps\|_{L^1((0,T))}\leq
\frac{1}{\veps}\|U\|_{L^\infty(Q_T)}\|U\|_{L^1(Q_T)}$ by Theorem
\ref{collectedresultselliptic} (b). For the proof of Theorem
\ref{uniqueness}, we will now show that there is a sequence $\veps_n\to
0^+$ such that $\lim_{\veps_n\to0^+}h_{\veps_n}(t)=0$. To do that we start by the
following lemma:

\begin{lemma}\label{dissolprop}
Assume  \eqref{muassumption}, $U\in L^1(Q_T)\cap L^\infty(Q_T)$, $\Phi\in L^\infty(Q_T)$, and  \eqref{diffparabolic} holds. Then 
\begin{enumerate}[(a)]
\item $\displaystyle \iint_{Q_T}\lef\B [U]\dell_t\psi + (\veps \B[\Phi] -\Phi)\psi\rig
\dd x \dd t=0\quad\text{for all}\quad
\psi\in C_\textup{c}^\infty(\R^N\times [0,T))$.
\item $\displaystyle
  \B[U(\cdot,t)](x)=\int_0^t\lef\veps\B[\Phi(\cdot,s)](x)-\Phi(x,s)\rig \dd s
  \quad\text{a.e.}\quad (x,t)\in\R^N\times(0,T)$.
\item For a.e. $t\in(0,T)$, $\|\B[U](\cdot,t)\|_{L^\infty(\R^N)}\leq2
  t\|\Phi\|_{L^\infty(Q_T)}$. 
\end{enumerate}
\end{lemma}

\begin{proof}
(a) We fix $\gamma \in C_\textup{c}^\infty(\R^N\times[0,T))$ and take
$\psi=\B[\gamma]$ as a test function in \eqref{diffparabolic}.
Note that $\psi$ is an admissible test function by a density argument
using Corollary \ref{propertiesofB2} (a)--(c) and $U,\Phi\in
L^\infty(Q_T)$. Then by  \eqref{elliptic} and Corollary 
\ref{propertiesofB2} (c), 
\begin{equation*}
\begin{split}
 0=&\iint_{Q_T} \lef U \dell_t(\B[\gamma]) + \Phi \Levy^{\mu} \left[ \B[\gamma] \right]\rig\dd x\dd t\\
=&  \iint_{Q_T} \lef U \B[\dell_t\gamma] + \Phi \lef \veps \B[\gamma]-\gamma\rig\rig\dd x\dd t.
\end{split}
\end{equation*}
Finally, the self-adjointness of $\B$ (cf. Lemma \ref{symmetryB}) yields
\[
\int_0^T \int_{\R^N}\lef \B[U] \dell_t\gamma + \lef \veps \B[\Phi]-\Phi \rig\gamma \rig \dd x\dd t=0,
\]
which completes the proof.

\smallskip
\noindent (b) This result follows from (a) and a special choice of test
function. For $0< s<T$, $a>0$, and $0<\delta<T-a$, we define 
\begin{equation*}
\theta_a(t)=\begin{cases}
1 & t\leq s-a\\
1-\frac{1}{a}(t-s+a) & s-a< t<s\\
0 & t\geq s
\end{cases}\qquad \text{and}\qquad \theta_{a,\delta}(t)=\theta_a\ast\rho_\delta(t),
\end{equation*}
where the mollifier $\rho_\delta$ is defined in \eqref{mollifiertime}.
Then $\theta_{a,\delta}\in C_\textup{b}^\infty((0,T))\cap L^1((0,T))$ and
$\textup{supp}\{\theta_{a,\delta}\}\subset[-\infty,T)$. Let
$\gamma\in C_\textup{c}^{\infty}(\R^N)$ and take
$\psi(x,t)=\theta_{a,\delta}(t)\gamma(x)\in C_\textup{c}^{\infty}(\R^N\times
[0,T))$ as a test function in part (a). Then we use properties of
mollifiers and Lebesgue's dominated convergence theorem to send
$\delta\to0^+$ and get
\begin{equation*}
\iint_{Q_T}\Big(\B[U]\theta_{a}' + (\veps
\B[\Phi] -\Phi)\theta_{a}\Big)\gamma\ \dd x\dd t=0.
\end{equation*}
By Fubini's theorem and since $\theta'_a(t)=-\frac{1}{a}\mathbf{1}_{s-a<t<s}$ and
$\textup{supp}\{\theta_{a}\}=[0,s]$, we find that
\begin{equation*}
\int_{\R^N}\bigg(\frac1a\int_{s-a}^s \B[U]\dd t + \int_0^s(\veps
\B[\Phi] -\Phi)\theta_{a}\dd t\bigg)\gamma\ \dd x=0.
\end{equation*}
We now send $a\to0^+$.
Since $\int_{\R^N}\B[U(\cdot,t)](x)\gamma(x) \dd x \in L^1(0,T)$ by Fubini's theorem, 
$$\frac1a\int_{s-a}^s \int_{\R^N}\B[U(\cdot,t)](x)\gamma(x) \dd x\dd t \to
\int_{\R^N}\B[U(\cdot,s)](x)\gamma(x)\dd x\quad\text{as}\quad a\to0^+$$
for a.e. $s$ by Lebesgue's differentiation theorem. For the other term,
we may use Lebesgue's dominated convergence theorem to pass to the
limit. Since $\theta_a\to \mathbf{1}_{[0,s)}$ pointwise, we find that for a.e. $s\in[0,T]$,
\begin{equation*}
\int_{\R^N} \bigg(\B[U(\cdot,s)](x) + \int_0^s\lef\veps
\B[\Phi(\cdot,t)](x) -\Phi(x,t)\rig\dd t\bigg)\gamma(x)\ \dd x=0.
\end{equation*}
Since $\gamma\in C_\textup{c}^\infty(\R^N)$ is arbitrary, part (b) follows.

\smallskip
\noindent(c) By part (b) and Theorem \ref{collectedresultselliptic} (c), $\|\B[U](\cdot,t)\|_{L^\infty(\R^N)}\leq2 t\|\Phi\|_{L^\infty(Q_T)}$ a.e.
\end{proof}

\begin{proposition}\label{habscont}
Assume \eqref{muassumption}, $U\in L^1(Q_T)\cap L^\infty(Q_T)$,
$\Phi\in L^\infty(Q_T)$, and \eqref{diffparabolic} holds. Then $h_\veps(t)$ defined by \eqref{defh} is absolutely continuous and
\begin{equation*}
h_\veps'(t)=2\big(\veps\B[\Phi](\cdot,t)-\Phi(\cdot,t),U(\cdot,t)\big)\qquad\text{in}\qquad \mathcal{D}'\left((0,T)\right).
\end{equation*}
\end{proposition}

The proof below is an adaptation of the proof in \cite[pp. 157--158]{BrCr79}.

\begin{proof}
 Let the mollifier $\rho_\delta=\rho_\delta(t)$ be defined in \eqref{mollifiertime}, the extension $\bar{U}$ be $U$ on $Q_T$ and zero outside $Q_T$, and
$$
\bar{U}_\delta(x,t):=\bar{U}(x,\cdot)\ast\rho_\delta(t)=\int_{\R}\bar{U}(x,s)\rho_\delta(t-s)\dd s.
$$ 
By Young's inequality, $\|\bar{U}_\delta\|_{L^\infty(Q_T)}\leq\|U\|_{L^\infty(Q_T)}$ and $\|\bar{U}_\delta\|_{L^1(Q_T)}\leq\|U\|_{L^1(Q_T)}$.
Moreover, the time continuity of $\bar{U}_\delta$, Corollary \ref{propertiesofB2} (c), and Lemma \ref{symmetryB} yields
\begin{align}
\label{ttt}
\frac{\dd }{\dd t}\int_{\R^N}\B[\bar{U}_\delta]\bar{U}_\delta\dd x=2\int_{\R^N}\dell_t\left(\B\left[ \bar{U}_\delta\right]\right)\bar{U}_\delta\dd x=2\int_{\R^N}\dell_t(\bar{U}_\delta)\B[\bar{U}_\delta]\dd x
\end{align}
for $t\in\R$.

Let us show that
\begin{equation}\label{movingBinsideintegral}
\B[\bar{U}_{\delta}(\cdot,t)](x)=\int_{\R}
\B[\bar{U}(\cdot,s)](x)\rho_\delta(t-s)\dd s \qquad\text{in}\qquad Q_T.
\end{equation}
First assume that $\bar{U}\in C_\textup{b}^\infty(Q_T)\cap
L^1(Q_T)$. Then $\B[\bar{U}(\cdot,t)]\in C_\textup{b}^\infty(\R^N)\cap L^1(\R^N)$ for
 $t\in[0,T]$, and thus, it solves \eqref{elliptic} pointwise in
$\R^N$. Multiply this equation by $\rho_\delta(s-t)$, integrate over
$\R$, and use Fubini's theorem and the uniqueness in Theorem
\ref{collectedresultselliptic} (b) and (c) to find that
\eqref{movingBinsideintegral} holds. A density/mollification argument
using uniqueness and $L^1(\R^N)$ and 
$L^\infty(\R^N)$ estimates from Theorem \ref{collectedresultselliptic}
then shows that \eqref{movingBinsideintegral} also holds (a.e.!) for
$\bar U\in
L^1(Q_T)\cap L^\infty(Q_T)$.  

Let the extension $\bar{\Phi}$ be $\Phi$ on $Q_T$ and zero outside
$Q_T$. Using Lemma \ref{dissolprop} (a) with test functions $\psi\in
C_\textup{c}^\infty(\R^N\times(\delta,T-\delta))$ we get that
$$
\dell_t\B[\bar{U}_\delta(\cdot,t)](x)=\bigg(\big(\veps\B[\bar{\Phi}]-\bar{\Phi}\big)(x,\cdot)\ast\rho_\delta\bigg)(t)\quad\text{a.e. in}\quad \R^N\times(\delta,T-\delta).
$$
For any $\Theta\in C_\textup{c}^\infty((0,T))$ and sufficiently small
$\delta$, we then conclude from \eqref{ttt} that
\begin{equation*}
\begin{split}
-\int_0^T\big(\B[\bar{U}_\delta](\cdot,t), \bar{U}_\delta(\cdot,t)\big)\Theta'(t)\dd t=2\int_0^T\big((\veps\B[\bar{\Phi}]-\bar{\Phi})\ast\rho_\delta(t),\bar{U}_\delta(\cdot, t)\big)\Theta(s)\dd t.
\end{split}
\end{equation*}
By properties of mollifiers and Theorem \ref{collectedresultselliptic}
(b) and (c),
\begin{equation*}
\begin{split}
&\bar{U}_\delta\to
U\qquad\text{in}\qquad L^1(Q_T),\\
&(\veps\B[\bar{\Phi}]-\bar{\Phi})\ast\rho_\delta\to\veps\B[\Phi]-\Phi
\qquad\text{a.e. in}\qquad Q_T,\\
&\veps\|\B[\bar{U}_\delta]\|_{L^\infty(Q_T)}\leq\|U\|_{L^\infty(Q_T)},\\
&|(\veps\B[\bar{\Phi}]-\bar{\Phi})\ast\rho_\delta|\leq2\|\Phi\|_{L^\infty(Q_T)}.
\end{split}
\end{equation*}
Now we send $\delta\to0^+$ using Lebesgue's dominated convergence theorem,
and then by the definition of $h_\veps$, we find that 
\begin{equation*}
-\int_0^Th_\veps(t)\Theta'(t)\dd t=2\int_0^T\big(\veps\B[\Phi](\cdot,t)-\Phi(\cdot,t),U(\cdot,t)\big)\Theta(t)\dd t.
\end{equation*}
That is, $h_\veps$ is weakly differentiable and the weak derivative is
\[
h_\veps'(t)=2\big(\veps\B[\Phi](\cdot,t)-\Phi(\cdot,t),U(\cdot,t)\big).
\]
Moreover, $h_\veps'\in L^1((0,T))$ since by Theorem
\ref{collectedresultselliptic} (c),
\begin{equation*}
\begin{split}
\int_0^T|h'_\veps(t)|\dd t \leq4\|\Phi\|_{L^\infty(Q_T)}\|U\|_{L^1(Q_T)}.
\end{split}
\end{equation*}
Hence, $h_\veps(t)$ is absolutely
continuous, and the proof is complete. 
\end{proof}

\begin{proposition}\label{proph}
Assume \eqref{phiassumption}, \eqref{muassumption},
$U\in L^1(Q_T)\cap L^\infty(Q_T)$, $\Phi\in L^\infty(Q_T)$ and
\eqref{diffparabolic} holds. Then
\begin{enumerate}[(a)]
\item For a.e. $t\in[0,T]$
\begin{equation*}
h_\veps(t)=\veps\|\B[U](\cdot,t)\|_{L^2}^2+\|(\Levy^\mu)^{\frac{1}{2}}[\B[U]](\cdot,t)\|_{L^2}^2.
\end{equation*}
\item If a sequence $\veps_n\Bn[U]\to0$ a.e. in $Q_T$ as $\veps_n\to0^+$, then for a.e. $t\in[0,T]$,
\begin{equation*}
\lim_{\veps_n\to0^+}h_{\veps_n}(t)= 0.
\end{equation*}
\end{enumerate}
\end{proposition}

We need a technical lemma (cf. \cite{BrCr79}).

\begin{lemma}\label{contfiniteset}
Assume \eqref{phiassumption} and \eqref{cond3}. Then the Lebesgue measure of the set 
$$
S^\xi:=\{(x,t)\in Q_T: |\varphi(u(x,t))-\varphi(\uu(x,t))|>\xi\},
$$
is finite for all $\xi>0$.
\end{lemma}

\begin{proof}
Define the set
$$
S_u^{\delta}=\{(x,t)\in Q_T:|u(x,t)-\uu(x,t)|>\delta\}.
$$
If $(x,t)\in S^\xi$, then by the continuity of $\varphi$ there exists a $\delta>0$ such that $|u(x,t)-\uu(x,t)|>\delta$, that is, $S^\xi\subset S_u^{\delta}$. By \eqref{cond3}, 
\begin{equation*}
\begin{split}
\delta|S_u^{\delta}|<\iint_{Q_T}|u(x,t)-\uu(x,t)|\dd x \dd t<\infty,
\end{split}
\end{equation*}
and thus, $S^\xi$ also has finite Lebesgue measure. 
\end{proof}

\begin{proof}[Proof of Proposition \ref{proph}]
(a) By the assumptions, Theorem \ref{collectedresultselliptic} (b) and
(c), interpolation between $L^1(\R^N)$ and $L^\infty(\R^N)$, and Fubini's theorem, we have for a.e. $t\in[0,T]$ that $U,\B[U]\in L^2(\R^N)$ and
\begin{equation}\label{disteqpoint}
\veps \B[U]-\Levy^\mu[\B[U]]=U\qquad\text{in}\qquad\mathcal{D}'(\R^N).
\end{equation}
Hence it follows that $\Levyd^\mu[\B[U]]\in L^2(\R^N)$, where $\Levyd^\mu$ is defined through the relation
\begin{equation*}
\int_{\R^N}\Levyd^\mu[\B[U]]\psi \dd x\dd t=\int_{\R^N}\B[U] \Levy^\mu[\psi]\dd x\dd t \qquad \text{for all}\qquad \psi \in C_\textup{c}^\infty(\R^N).
\end{equation*}

Using Plancherel's theorem and Lemma \ref{fouriersymbol}, we then find that for any $\psi\in C_\textup{c}^\infty(\R^N)$,
\begin{equation*}
\begin{split}
\int_{\R^N}\mathcal{F}\big(\Levyd^\mu[\B[U]]\big)\mathcal{F}(\psi )\dd \xi&=\int_{\R^N}\mathcal{F}(\B[U]) \mathcal{F}(\Levy^\mu[\psi])\dd \xi\\
&=-\int_{\R^N}\mathcal{F}(\B[U]) \sigma_{\Levy^\mu}(\xi)\mathcal{F}(\psi) \dd \xi,
\end{split}
\end{equation*}
and hence
\[
\int_{\R^N}\mathcal{F}(\psi )(\xi)\Big( \mathcal{F}\big(\Levyd^\mu[\B[U]]\big)(\xi)+ \sigma_{\Levy^\mu}(\xi)\mathcal{F}(\B[U])(\xi) \Big)\dd \xi=0.
\]
Then by a density argument, we conclude that 
$$ 
\mathcal{F}\big(\Levyd^\mu[\B[U]]\big)(\xi)=-\sigma_{\Levy^\mu}(\xi)\mathcal{F}(\B[U])(\xi)\qquad\text{in}\qquad L^2(\R^N),
$$ 
and thus, for a.e. $t\in [0,T]$, we have $\Levyd^\mu[\B[U]]=\Levy^\mu[\B[U]]$ in $L^2(\R^N)$.

Since $U, \B[U], \Levy^\mu[\B[U]]\in L^2(\R^N)$, equation \eqref{disteqpoint} holds in $L^2(\R^N)$. By Lemma \ref{fouriersymbol}, Remark \ref{remarkfouriersymbol} (b), and the definition of $h_\veps$ (see \eqref{defh}), we have for a.e. $t\in[0,T]$ that
\begin{equation*}
\begin{split}
h_\veps(t)=\,&\big( \B[U](\cdot,t),U(\cdot,t) \big)_{L^2(\R^N)}\\
=&\,\big(\B[U](\cdot,t),\veps \B[U](\cdot,t)-\Levy^\mu[\B[U]](\cdot,t) \big)_{L^2(\R^N)}\\
=&\,\veps\|\B[U](\cdot,t)\|_{L^2(\R^N)}^2-\big(\B[U](\cdot,t),\Levy^\mu[\B[U]](\cdot,t) \big)_{L^2(\R^N)}.\\
=&\,\veps\|\B[U](\cdot,t)\|_{L^2(\R^N)}^2+\|(\Levy^\mu)^{\frac{1}{2}}[\B[U]]\|_{L^2(\R^N)}^2.
\end{split}
\end{equation*}

\smallskip
\noindent
(b) By part (a), Proposition \ref{habscont}, and $U\Phi=(u-\uu)(\varphi(u)-\varphi(\uu))\geq0$,
\begin{equation}\label{hposabscont}
\begin{split}
0\leq h_\veps(t)=&\,h_\veps(0+)+\int_0^th_\veps'(s)\dd s\\
\leq&\, h_\veps(0+) +2\int_0^t \big( \veps \B[\Phi](\cdot,s), U(\cdot,s) \big) \dd s.
\end{split}
\end{equation}

By the (absolute) continuity of $h_\veps$, H\"older's inequality,
Lemma \ref{dissolprop} (c), and Lebesgue's dominated convergence
theorem (valid since $U\in L^1(Q_T)$), 
\begin{equation*}
\begin{split}
h_\veps(0+)&=\lim_{t\to0^+}\frac{1}{t}\int_{0}^th_\veps(s)\dd s\leq\lim_{t\to0^+}\frac{1}{t}\int_0^t\|\B[U](\cdot,s)\|_{L^\infty(\R^N)}\|U(\cdot,s)\|_{L^1(\R^N)}\dd s\\
&\leq2\|\Phi\|_{L^\infty(Q_T)}\lim_{t\to0^+}\int_0^T\|U(\cdot,s)\|_{L^1(\R^N)}{\mathbf
  1}_{(0,t)}(s)\dd s=0.\\
\end{split}
\end{equation*}

Let $\xi>0$. By the self-adjointness of $\B$ (cf. Lemma \ref{symmetryB}) and Theorem \ref{collectedresultselliptic} (b), we get for a.e. $t\in[0,T]$
\begin{equation*}
\begin{split}
&\big(\veps\B[\Phi](\cdot,t), U(\cdot,t) \big)=\int_{\R^N} \Phi(x,t) \veps\B[U(\cdot,t)](x)\dd x\\
&\leq\, \|\Phi\|_{L^\infty}\int_{\{|\Phi(x,t)|>\xi\}}\left|\veps  \B[U]\right|\dd x+\xi \int_{\{|\Phi(x,t)|\leq\,\xi\}}\left|\veps\B[U]\right|\dd x\\
&\leq\,\|\Phi\|_{L^\infty}\int_{\R^N}\left|\veps  \B[U(\cdot,t)]\right|\mathbf{1}_{|\Phi(x,t)|>\xi}\dd x+\xi \|U(\cdot,t)\|_{L^1(\R^N)}.
\end{split}
\end{equation*}
Let $t$ be a point where this inequality holds and $\veps_n
B_{\veps_n}^\mu[U(\cdot,t)]\to0$ a.e. $x$ and $|\veps  \B[U(\cdot,t)](x)|\leq
\|U\|_{L^\infty(Q_T)}$ a.e. $x$ (using Theorem
\ref{collectedresultselliptic}  (c)). For any $\eta>0$, take $\xi$ such that
$\xi\|U(\cdot,t)\|_{L^1}<\frac{1}{2}\eta$. Then note that
$\left|\veps  \B[U]\right|\mathbf{1}_{|\Phi(x,t)|>\xi}$ is dominated
by $\|U\|_{L^\infty}\mathbf{1}_{|\Phi(x,t)|>\xi}$ which is integrable
by Lemma \ref{contfiniteset}.  By Lebesgue's dominated convergence theorem it then follows that
 $\int_{\R^N}|\veps_n
 B_{\veps_n}^\mu[U(\cdot,t)]|\mathbf{1}_{|\Phi(x,t)|>\xi}\dd
 x$ $<\frac{1}{2}\eta$ when $\veps_n$ is small enough. Since this
 holds for a.e. $t\in[0,T]$, we have proven that
\[
\lim_{\veps_n\to0^+}\big(  \veps_n B_{\veps_n}^\mu[\Phi](\cdot,t), U(\cdot,t) \big)\leq0\qquad\text{for a.e.}\qquad t\in[0,T].
\]

We conclude the proof using Lebesgue's dominated convergence
theorem to send $\veps_n\to0^+$ in $\eqref{hposabscont}$ (the integrand is
dominated by $\|\Phi\|_{L^\infty(Q_T)}\|U(\cdot,t)\|_{L^1(\R^N)}\in L^1((0,T))$ since $U\in L^1(Q_T)\cap L^\infty(Q_T)$).
\end{proof} 

\begin{proposition}\label{goestozero}
Assume \eqref{muassumption}, $\supp \mu\neq\emptyset$, and $g\in L^1(\R^N)\cap L^\infty(\R^N)$. Then there exists a sequence such that $\veps_n\Bn[g]\to0$ a.e. in $\R^N$ as $\veps_n\to0^+$.
\end{proposition}

This proposition will be proven later in this section. We are now ready to prove our main result.

\begin{proof}[Proof of Theorem \ref{uniqueness}]
In the case that $\supp \mu=\emptyset$, $\mu\equiv0$ and
$\Levy^\mu\equiv0$. Then equation \eqref{E} becomes the ODE $u_t=0$, and
uniqueness follows by standard arguments (e.g. one can easily deduce
that $\int_{\R^N} |u(x,t)-\hat u(x,t)| \dd x\leq\int_{\R^N} |u(x,0)-\hat u(x,0)| \dd x$).

Now consider the case $\supp \mu\neq\emptyset$.
By Proposition \ref{goestozero} and \ref{proph} (a) and (b),  there is
a sequence such that for a.e. $t\in[0,T]$,
\begin{equation}\label{normstozero}
\veps_n\|\Bn[U](\cdot,t)\|_{L^2}^2+\|(\Levy^\mu)^{\frac{1}{2}}[\Bn[U]](\cdot,t)\|_{L^2}^2\to0\qquad\text{as}\qquad \veps_n\to0^+.
\end{equation}

Let $\psi\in C_\textup{c}^\infty(\R^N)$. By Plancherel's theorem, Lemma
\ref{fouriersymbol}, and Cauchy-Schwarz' inequality, and finally, by
\eqref{normstozero}, we get for a.e. $t\in[0,T]$ that
\begin{equation*}
\begin{split}
\left| \int_{\R^N} \Bn[U]\Levy^\mu[\psi]\dd x\right|&=\left|- \int_{\R^N} (\Levy^\mu)^{\frac{1}{2}}[\Bn[U]](\Levy^\mu)^{\frac{1}{2}}[\psi]\dd x\right|\\
&\leq\|(\Levy^\mu)^{\frac{1}{2}}[\Bn[U]]\|_{L^2(\R^N)}\|(\Levy^\mu)^{\frac{1}{2}}[\psi]\|_{L^2(\R^N)}\to0
\end{split}
\end{equation*}
as $\veps_n\to0^+$. Moreover, by Cauchy-Schwarz' inequality and \eqref{normstozero}, we have for a.e. $t\in[0,T]$
\begin{equation*}
\begin{split}
\left|\int_{\R^N}\veps_n\Bn[U]\psi\dd x\right|\leq\|\veps_n\Bn[U]\|_{L^2(\R^N)}\|\psi\|_{L^2(\R^N)}\to0\quad\text{as}\quad \veps_n\to0^+.
\end{split}
\end{equation*}
Hence we conclude that as $\veps_n\to0^+$,
$$
U=\veps_n\Bn[U]-\Levy^\mu[\Bn[U]]\to0 \qquad\text{in}\qquad \mathcal{D}'(\R^N),
$$
for a.e. $t\in[0,T]$. That is, 
$$u-\hat u=U\equiv 0  \qquad\text{in} \qquad \mathcal{D}'(\R^N)$$ 
for a.e. $t\in[0,T]$, and then a.e. in $Q_T$ by
du Bois-Reymond's lemma.
\end{proof}

In the rest of this section, we prove Proposition \ref{goestozero}. For $\gamma\in C_\textup{c}^\infty(\R^N)$, we let $v_\veps:=\veps\B[\gamma]$ be the unique smooth classical solution (see Theorem \ref{collectedresultselliptic} (a) and Corollary
\ref{propertiesofB2} (a)) of  
\begin{equation}\label{classicalhelpingeq}
\veps v_\veps(x)-\Levy^\mu[v_\veps](x)=\veps\gamma(x)\qquad\text{for all}\qquad x\in\R^N.
\end{equation}
We want to prove that there exists a sequence such that $v_{\veps_n}={\veps_n}B_{\veps_n}^\mu[\gamma]\to0$ as $\veps_n\to0^+$ for every $x\in\R^N$ and every $\gamma\in C_\textup{c}^\infty(\R^N)$ .

\begin{lemma}\label{convdistributionaleq}
Assume \eqref{muassumption} and $\gamma\in
C_\textup{c}^\infty(\R^N)$. Then there exists a sequence $\left\{\veps_n\Bn[\gamma]\right\}_{n\in \mathbb{N}}$
that converges  locally uniformly in $\R^N$ as 
$\veps_n\to0^+$. Moreover, the corresponding limit $v$ is uniformly
continuous, $\lim_{|x|\to \infty}v=0$ and satisfies 
$$
\Levy^\mu[v](x)=0 \qquad\text{in}\qquad \mathcal{D}'(\R^N).
$$
\end{lemma}

\begin{lemma}[Barb\u{a}lat]\label{Barbalat}
If $\psi\in L^1(\R^N)$ is uniformly continuous,
then\linebreak $\displaystyle\lim_{|x|\to\infty}\psi(x)=0$. 
\end{lemma}

For a proof, see e.g. Lemma 5.2 in \cite{KaLaQu15} (take $G=\R^N$
and $B=\R$). 

\begin{proof}[Proof of Lemma \ref{convdistributionaleq}]
We recall that $v_\veps:=\veps\B[\gamma]$. By Theorem \ref{collectedresultselliptic} (a),
$$
\|D^\alpha v_\veps\|_{L^\infty(\R^N)}\leq\|D^\alpha \gamma\|_{L^\infty(\R^N)}
$$
for each multiindex $\alpha\in \mathbb{N}^N$. So, then any sequence $\{v_{\veps_n}\}_{n\in\N}$ is equibounded and equilipschitz. By
Arzel\`a-Ascoli's theorem, there exists a subsequence such that
$v_{\veps_n}\to v$ locally uniformly as $n\to\infty$. Since
$v_{\veps_n}$ is uniformly continuous (the derivative of $v_{\veps_n}$
exists and is bounded) and by the local uniform convergence, for every $\eta>0$ and $R>0$ we can
find some $n>0$  such that $\max\{|v(x)-v_{\veps_n}(x)|
: |x|\leq R\}<\eta$. Thus, we have the following estimate for every
$R>0$ and $|x|,|y|\leq R$,
\begin{equation*}
\begin{split}
|v(x)-v(y)|\leq&\, |v(x)-v_{\veps_n}(x)|+|v_{\veps_n}(x)-v_{\veps_n}(y)|+|v_{\veps_n}(y)-v(y)|\\
\leq&\, 2\eta+ \|D\gamma\|_{L^\infty(\R^N)}|x-y|
\end{split}
\end{equation*}
As $R$ is arbitrary, $v$ is Lipschitz continuous with Lipschitz constant $\|D\gamma\|_{L^\infty(\R^N)}$, and thus, uniformly continuous. Furthermore, Fatou's lemma and Theorem \ref{collectedresultselliptic} (b) give that $\|v\|_{L^1}\leq\liminf_{n\to\infty}\|v_{\veps_n}\|_{L^1}\leq\|\gamma\|_{L^1}$. By Lemma \ref{Barbalat}, $\lim_{|x|\to\infty}v(x)=0$.

Multiplying \eqref{classicalhelpingeq} by a test function, integrating over $\R^N$, and using self-adjointness (cf. Lemma \ref{propnonlocal}) of $\Levy^\mu$ we get
$$
\veps_n\int_{\R^N}v_{\veps_n}\psi\dd x-\int_{\R^N}v_{\veps_n}\Levy^\mu[\psi]\dd x={\veps_n}\int_{\R^N}\gamma\psi\dd x \qquad\text{for all}\qquad \psi\in C_\textup{c}^\infty(\R^N).
$$
Since $\|v_{\veps_n}\|_{L^\infty}\leq\|\gamma\|_{L^\infty}$ by Theorem \ref{collectedresultselliptic} (c), we use Lebesgue's dominated convergence theorem to take the limit as $\veps_n\to0^+$ , to find that
$$
0=\lim_{\veps_n\to0^+}\int_{\R^N}v_{\veps_n}\Levy^\mu[\psi]\dd x=\int_{\R^N}v\Levy^\mu[\psi]\dd x \qquad\text{for all}\qquad \psi\in C_\textup{c}^\infty(\R^N),
$$
which completes the proof.
 \end{proof}

\begin{lemma}\label{existenceofsubsequence}
Assume \eqref{muassumption} and $g\in L^1(\R^N)\cap  L^\infty(\R^N)$. Then there exists a
sequence $\{\veps_n\Bn[g]\}_{n\in\N}$ that converges in
$L_{\textup{loc}}^1(\R^N)$ as $\veps_n\to0^+$. 
\end{lemma}

\begin{proof}
Note that $u_\veps:=\veps\B[g]$ is the unique distributional
solution (see Theorem \ref{collectedresultselliptic} (b) and (c)) of
the following elliptic problem 
\begin{equation*}
\veps u_{\veps}(x)-\Levy^{\mu}[u_{\veps}](x)=\veps g(x) \qquad\text{in}\qquad \mathcal{D}'(\R^N).
\end{equation*}
By Theorem \ref{collectedresultselliptic} (b) and (c)
and the linearity of the above equation, for any $h\in\R^N$,
$$\|u_\veps\|_{L^\infty}\leq\|g\|_{L^\infty},\quad\|u_\veps\|_{L^1}\leq\|g\|_{L^1}\quad\text{and}\quad \|u_\veps(\cdot+h)-u_\veps\|_{L^1}\leq\|g(\cdot+h)-g\|_{L^1}.$$ 

Now let $K\subset \R^N$ be any compact set, and define
$w_\veps^{K}(x)=u_\veps(x)\mathbf{1}_{K}(x)$. The uniform in
$\veps$ bound ensures that the family $M:=\{w_\veps^K\}_{\veps>0}\subset
L^1(\R^N)$ is uniformly bounded  in $L^1(\R^N)$. Moreover, by
continuity of the $L^1$-translation, Theorem \ref{collectedresultselliptic} (b) and (c), and Lebesgue's dominated convergence theorem,
\begin{equation*}
\begin{split}
&\|w_\veps^K(\cdot+h)-w_\veps^K\|_{L^1}\\
&\leq
\|\left(u_\veps(\cdot+h)-u_\veps\right)\mathbf{1}_{K}(\cdot+h)\|_{L^1}+\|u_\veps\left(\mathbf{1}_{K}(\cdot+h)-\mathbf{1}_{K}\right)\|_{L^1}\\ 
&\leq
\|g(\cdot+h)-g\|_{L^1}+\|g\|_{L^\infty}\int_{\R^N}\left|\mathbf{1}_{K}(x+h)-\mathbf{1}_{K}(x)\right|\dd
x\to0\quad\text{as}\quad |h|\to0.
\end{split}
\end{equation*}
Combining the above results, we see that $M$ is relatively compact by
Kolmogorov's compactness theorem (see e.g. \cite[Theorem
A.5]{HoRi07}). Hence, there is a convergent subsequence in
$L^1(K)$. 

Now, cover $\R^N$ by a countable number of
balls $B_n$. Then the above
argument holds for $K:=\overline{B}_n$ for every $n\in
\mathbb{N}$. A diagonal argument then allows us to 
pick a subsequence which converges in
$L^1(\overline{B}_n)$ for each $n$, and thus in
$L_\textup{loc}^1(\R^N)$.
\end{proof}

\begin{remark}
By Theorem \ref{collectedresultselliptic} (a) and Arzel\`a-Ascoli, we can have $D^\alpha v_\veps\to w_\alpha$ locally uniformly in $\R^N$ as $\veps\to0^+$ for all multiindex $\alpha\in \mathbb{N}^N$. However, because of the lack of uniqueness in $\Levy^\mu[v](x)=0$, we do not know if $D^\alpha v=w_\alpha$. Hence, we are forced to work with distributional solutions of $\Levy^\mu[v](x)=0$.
\end{remark}

\begin{lemma}\label{classicalconvgivesconv}
Assume \eqref{muassumption}, $g\in L^1(\R^N)\cap L^\infty(\R^N) $, and $\{\veps_nB_{\veps_n}^\mu[g]\}_{n\in\N}$ converges in $L_\textup{loc}^1(\R^N)$. If $\veps_nB_{\veps_n}^\mu[\gamma](x)\to0$ as $\veps_n\to0^+$ for every $x\in\R^N$ and every $\gamma\in C_\textup{c}^\infty(\R^N)$, then $\veps_nB_{\veps_n}^\mu[g]\to0$ in $L_\textup{loc}^1(\R^N)$ as $\veps_n\to0^+$.
\end{lemma}

\begin{proof}
By the self-adjointness given in Lemma \ref{symmetryB}, and the definitions $u_{\veps_n}:=\veps_n\B[g]$,  $v_{\veps_n}:=\veps_n\B[\gamma]$, we have
$$
\int_{\R^N}u_{\veps_n}(x)\gamma(x)\dd x=\int_{\R^N}g(x)v_{\veps_n}(x)\dd x.
$$
Since $\|v_{\veps_n}\|_{L^\infty}\leq\|\gamma\|_{L^\infty}$ by Theorem
\ref{collectedresultselliptic} (c),
$|g(x)v_{\veps_n}(x)|\leq|g(x)|\|\gamma\|_{L^\infty}$. Then by the
assumption and Lebesgue's dominated convergence theorem, 
$$
\lim_{\veps_n\to0^+}\int_{\R^N}u_{\veps_n}(x)\gamma(x)\dd x=0 \qquad\text{for all}\qquad \gamma\in C_\textup{c}^\infty(\R^N),
$$
Hence $u_{\veps_n}\to0$ in $\mathcal{D}'(\R^N)$, and since the
distributional and $L_\textup{loc}^1$ limits coincide (by uniqueness),
it follows that $u_{\veps_n}\to0$ in 
$L_\textup{loc}^1(\R^N)$ as $\veps_n\to0^+$. 
\end{proof}

\begin{proof}[Proof of Proposition \ref{goestozero}]
Let $\gamma\in C_\textup{c}^\infty(\R^N)$ be arbitrary, and recall the definitions $\veps\B[\gamma]=v_\veps$ and $\veps\B[g]=u_\veps$. Lemma \ref{convdistributionaleq} yields a subsequence such that $v_{\veps_n}\to v$ locally uniformly as $\veps_n\to0^+$ with $v\in C_0(\R^N)$ and $\Levy^\mu[v](x)=0$ in $D'(\R^N)$. Then, Theorem \ref{Liouville} ensures that $v(x)=0$ for every $x\in\R^N$.

Hence, Lemma \ref{existenceofsubsequence} and \ref{classicalconvgivesconv} give that $u_{\veps_n}\to0$ in $L_\textup{loc}^1(\R^N)$ as $\veps_n\to0^+$. Finally, take a further subsequence (still denoted by $\veps_n$) such that $u_{\veps_n}\to0$ a.e. in $\R^N$ as $\veps_n\to0^+$.
\end{proof}

\section{Stability, existence and a priori results}\label{sec:stabexapriori}
In this section, we will start by showing the stability result stated
in Section \ref{sec:mainresults}, and then we continue by showing
existence and a priori results for \eqref{E}. The latter part will
follow by regularization and compactness from results in \cite{CiJa14} for
the case $\varphi\in W_\textup{loc}^{1,\infty}(\R)$ and $u_0\in
L^\infty(\R^N)\cap L^1(\R^N)$.

\begin{proof}[Proof of Theorem \ref{stability}]
Since $u_n$ are distributional solutions of \eqref{E}, we will take
the limit as $n\to\infty$ to see that so are also $u$. 

Assumption (iii) and the uniformly boundedness of $\|u_n\|_{L^\infty(Q_T)}$ gives for all $\psi\in C_\textup{c}^\infty(Q_T)$ that
\begin{equation*}
\int_0^T\int_{\R^N} u_n \dell_t\psi\dd x\dd t\to\int_0^T\int_{\R^N} u \dell_t\psi\dd x\dd t\qquad\text{as}\qquad n\to\infty.
\end{equation*} 

To prove convergence of the $\Levy^{\mu_n}$-term in the distributional
formulation we proceed as follows 
\begin{equation*}
\begin{split}
&\int_0^T\int_{\R^N} \lef\varphi_n(u_n) \Levy^{\mu_n}[\psi]-\varphi(u)\Levy[\psi]\rig\dd x\dd t\\
&=\int_0^T\int_{\R^N} \varphi_n(u_n) \big(\Levy^{\mu_n}[\psi]-\Levy[\psi]\big)\dd x\dd t+\int_0^T\int_{\R^N} \big(\varphi_n(u_n) -\varphi(u_n)\big)\Levy[\psi]\dd x\dd t\\
&\quad+\int_{0}^T\int_{\R^N}\big(\varphi(u_n)-\varphi(u)\big)\Levy[\psi]\dd x\dd t.
\end{split}
\end{equation*}
Since $\|u_n\|_{L^\infty(Q_T)}$ is uniformly bounded, $\varphi_n\to\varphi$ locally uniformly in $\R$ by assumption (ii), and $|\varphi_n(u_n)|\leq|\varphi_n(u_n)-\varphi(u_n)|+|\varphi(u_n)|$, we obtain for $n$ sufficiently large
\begin{equation}\label{varphiboundedstability}
\|\varphi_n(u_n)\|_{L^\infty(Q_T)}\leq\sup\{|\varphi(r)|:|r|\leq C\}+1=:C_\varphi. 
\end{equation}
Then, using assumption (i), we get 
\begin{equation*}
\left|\int_0^T\int_{\R^N} \varphi_n(u_n) \left(\Levy^{\mu_n}[\psi]-\Levy[\psi]\right)\dd x\dd t\right|\leq C_\varphi \int_0^T\int_{\R^N} \big|\Levy^{\mu_n}[\psi]-\Levy[\psi]\big|\dd x\dd t \to0
\end{equation*}
as $n\to\infty$. By the uniformly boundedness of $\|u_n\|_{L^\infty(Q_T)}$, and since $\varphi_n\to\varphi$ locally uniformly in $\R$ by assumption (ii),
$$
\|\varphi_n(u_n)-\varphi(u_n)\|_{L^\infty(Q_T)}\leq \sup\{|\varphi_n(r)-\varphi(r)| : |r|\leq C\}\to0\quad\text{as}\quad n\to\infty.
$$
Since we assume that $\Levy[\psi]\in L^1(Q_T)$,
\begin{equation*}
\begin{split}
\left|\int_0^T\int_{\R^N} \big(\varphi_n(u_n) -\varphi(u_n)\big)\Levy[\psi]\dd x\dd t\right|\leq\|\varphi_n(u_n)-\varphi(u_n)\|_{L^\infty}\|\Levy[\psi]\|_{L^1}\to0
\end{split}
\end{equation*}
as $n\to\infty$. By assumption (iii) and \eqref{phiassumption}, $|\varphi(u_{n})-\varphi(u)|\to0$ a.e. in $Q_T$ as $n\to\infty$, and $\|\varphi(u_n)\|_{L^\infty(Q_T)}\leq C$ for some $C$ independent of $n$. Hence, $|\varphi(u_n)-\varphi(u)|$ is bounded by $2C$. Moreover, since $\Levy[\psi]\in L^1(Q_T)$, Lebesgue's dominated convergence theorem yields
\begin{equation*}
\begin{split}
\left|\int_0^T\int_{\R^N}(\varphi(u_{n})-\varphi(u))\Levy[\psi]\dd x \dd t\right| \leq \int_0^T\int_{\R^N}|\varphi(u_{n})-\varphi(u)||\Levy[\psi]|\dd x\dd t\to0
\end{split}
\end{equation*}
as $n\to\infty$. The proof is complete.
\end{proof}

Let us turn our attention to proving the other main results in this section.

\begin{theorem}\label{propapproxparabolic}
Assume \eqref{phiassumption}, \eqref{muassumption}, $\varphi \in W_{\textup{loc}}^{1,\infty}(\R^N)$, $\varphi(0)=0$, and $u_0,\uu_0\in L^\infty(\R^N)\cap L^1(\R^N)$.
\begin{enumerate}[(a)]
\item There exists a unique entropy solution $u\in L^\infty(Q_T)\cap C([0,T];L^1(\R^N))$ of \eqref{E}.\item If $u, \uu$ are entropy solutions of \eqref{E} with initial data $u_0, \uu_0$ respectively, then for all $t\in [0,T]$
\begin{equation*}
\|u(\cdot,t)-\uu(\cdot, t)\|_{L^1(\R^N)}\leq\|u_0-\uu_0\|_{L^1(\R^N)}.
\end{equation*}
\item If $u$ is a entropy solution of \eqref{E} with initial data $u_0$, then for all $t\in[0,T]$
\begin{equation*}
\|u(\cdot,t)\|_{L^1(\R^N)}\leq\|u_0\|_{L^1(\R^N)}\qquad\text{and}\qquad\|u(\cdot,t)\|_{L^\infty(\R^N)}\leq\|u_0\|_{L^\infty(\R^N)}.
\end{equation*}
\end{enumerate}
\end{theorem}
Entropy solutions are defined in Definition 2.1 in  \cite{CiJa14}, and the result holds by Theorem 5.5 in  \cite{CiJa14} and Theorem 5.2 in \cite{CiJa11}.

In what follows, we let $u_0\in L^\infty(\R^N)\cap L^1(\R^N)$ and define

\begin{equation}\label{mollifieddiffusionfunction}
\varphi_\eta(x):=\varphi\ast\omega_\eta(x)-\varphi\ast\omega_\eta(0)\quad\text{where}\quad \omega_\eta \quad \text{is given by \eqref{mollifierspace} with $N=1$}. 
\end{equation}

Hence $\varphi_\eta\in W_{\textup{loc}}^{1,\infty}(\R)\subset C(\R)$, it is nondecreasing by \eqref{phiassumption}, $\varphi_\eta(0)=0$, and $\varphi_\eta\to\varphi$ locally uniformly in $\R$. Let $u_\eta$ be the entropy solution of \eqref{E} with $\varphi_\eta$ replacing $\varphi$. Since entropy solutions are distributional solutions (cf. Theorem 2.5 ii{\small)} and Section 5 in \cite{CiJa11}), 
\begin{equation}\label{approxdistrparabolic}
\int_0^T\int_{\R^N}\lef u_\eta\dell_t\psi+\varphi_\eta(u_\eta)\Levy^\mu[\psi]\rig\dd x \dd t+\int_{\R^N}u_0\psi\rvert_{t=0}\dd x=0 \quad \forall\psi\in C_\textup{c}^\infty(\R^N\times[0,T)).
\end{equation}
Going to the limit as $\eta\to0^+$ in \eqref{approxdistrparabolic}, we will prove the existence and the a priori results given in Theorems \ref{existenceparabolic} and \ref{consequences}.
\begin{remark}
  We will prove that the $L^1$-contraction holds for limits
  of the functions $\{u_{\eta}\}_{\eta>0}$. As a consequence of uniqueness
  (Corollary \ref{uniquenessresult}), this result then holds for all
  $L^\infty\cap L^1$-distributional solutions of \eqref{E}.
\end{remark}
 Before these results can be proven, we need an auxiliary lemma.

\begin{lemma}\label{convuapproxparabolic}
Assume $\eqref{muassumption}$, $u_0\in L^\infty(\R^N)\cap L^1(\R^N)$,
$\varphi_\eta$ satisfy \eqref{phiassumption} for all $\eta>0$, and
$\varphi_\eta\to\varphi$ locally uniformly as $\eta\to0^+$. If
$u_\eta$ solves \eqref{approxdistrparabolic} and satisfies Theorem
\ref{propapproxparabolic} (b) and (c), then there exists a subsequence
$\{u_{\eta_n}\}_{n\in\N}$  and a $u\in
C([0,T];L_\textup{loc}^1(\R^N))$ such that as $\eta_n\to0^+$ 
$$
u_{\eta_n}\to u \qquad\text{in}\qquad C([0,T];L_\textup{loc}^1(\R^N)).
$$
Moreover, for all $t\in[0,T]$
\begin{equation*}
\|u(\cdot,t)\|_{L^1(\R^N)}\leq\|u_0\|_{L^1(\R^N)}\qquad\text{and}\qquad\|u(\cdot,t)\|_{L^\infty(\R^N)}\leq\|u_0\|_{L^\infty(\R^N)}.
\end{equation*}
\end{lemma}

\begin{proof} We will use Kolmogorov's compactness theorem in the form of Theorem A.8 in \cite{HoRi07}. Let $K\subset \R^N$ be any compact set.
\medskip

\noindent
{\em Step 1:} $u_\eta$ is bounded independently of $\eta$ in $Q_T$ by Theorem \ref{propapproxparabolic} (c).
\medskip

\noindent
{\em Step 2:} Since \eqref{E} is translation invariant, $v(x,t)=u_\eta(x+h,t)$ solves \eqref{approxdistrparabolic} with initial data $v_0(x)=u_0(x+h)$ for every $h\in \R^N$. Let $\gamma\in \R^N$. By Theorem \ref{propapproxparabolic} (b) and since translations are continuous in $L^1$,
\begin{equation*}
\begin{split}
&\sup_{|h|\leq|\gamma|}\int_K|u_\eta(x+h,t)-u_\eta(x,t)|\dd x\leq\sup_{|h|\leq|\gamma|}\int_{\R^N}|u_\eta(x+h,t)-u_\eta(x,t)|\dd x\\
&\leq\sup_{|h|\leq|\gamma|}\int_{\R^N}|u_0(x+h)-u_0(x)|\dd x\leq\max_{|h|\leq|\gamma|}\tilde{\lambda}_{u_0}(|h|)=:\lambda_{u_0}(|\gamma|)
\end{split}
\end{equation*}
for some moduli of continuity $\tilde{\lambda}_{u_0}, \lambda_{u_0}$.
\medskip

\noindent
{\em Step 3:} Let $\omega_\delta$ be defined by \eqref{mollifierspace} and let $\Theta\in C_\textup{c}^\infty((0,T))$. For any $x \in \R^N$ take $\psi(y,t)=\Theta(t)\omega_{\delta}(x-y)$ as a test function in \eqref{approxdistrparabolic} to find that
\begin{equation}\label{insertedtimetestfunction}
\begin{split}
0=&\int_0^T\int_{\R^N}\lef u_\eta(y,t) \omega_\delta(x-y)\Theta'(t)+\varphi_\eta(u_\eta(y,t))\Levy^\mu[\omega_\delta](x-y)\Theta(t)\rig\dd y\dd t\\
=&\int_0^T\lef(u_\eta(\cdot,t)\ast\omega_\delta )(x)\Theta'(t)+(\varphi_\eta(u_\eta(\cdot,t))\ast\Levy^\mu[\omega_\delta])(x)\Theta(t)\rig\dd t.
\end{split}
\end{equation}
For $\rho_\delta(t)$ defined by \eqref{mollifiertime}, we choose
$$
\Theta(t):=\Theta_{\tilde{\delta}}(t)=\int_{-\infty}^t\lef\rho_{\tilde{\delta}}(\tau-t_1)-\rho_{\tilde{\delta}}(\tau-t_2)\rig\dd \tau,
$$
where $0<t_1<t_2<T$. For $\tilde{\delta}>0$ small enough, $\Theta_{\tilde{\delta}}(t)$ is supported in $[0,T]$ and is a smooth approximation to a square pulse which is one in $[t_1,t_2]$ and zero otherwise. By \eqref{insertedtimetestfunction},
\begin{equation*}
\begin{split}
\int_0^T\rho_{\tilde{\delta}}(t-t_2)\big(u_\eta(\cdot,t)\ast\omega_\delta\big)(x)\dd t=&\,\int_0^T\rho_{\tilde{\delta}}(t-t_1)\big(u_\eta(\cdot,t)\ast\omega_\delta\big)(x)\dd t\\
&\,+\int_0^T\Theta_{\tilde{\delta}}(t)\big(\varphi_\eta(u_\eta(\cdot,t))\ast\Levy^\mu[\omega_\delta]\big)(x)\dd t.
\end{split}
\end{equation*}
Let $u_\eta^\delta(x,t):=u_\eta(\cdot,t)\ast\omega_\delta(x)$. By Theorem \ref{propapproxparabolic} (c) and the properties of mollifiers, we send $\tilde{\delta}\to0^+$ in the previous equality to obtain the following pointwise identity,
\begin{equation}\label{difftimemollifiedapprox}
u_\eta^\delta(x,t_2)-u_\eta^\delta(x,t_1)=\int_{t_1}^{t_2}\big(\varphi_\eta(u_\eta(\cdot,t))\ast\Levy^\mu[\omega_{\delta}]\big)(x)\dd t.
\end{equation}

Now, we need to estimate the integral involving the mollified version of $u_\eta$.  Let $t,s\in[0,T]$ and take $\delta<\min\{t,s\}$. Use \eqref{difftimemollifiedapprox} to find that 
\begin{equation*}
\begin{split}
\int_K|u_\eta^\delta(x,t)-u_\eta^\delta(x,s)|\dd x&\leq\int_K\int_s^t\left|\big(\varphi_\eta(u_\eta(\cdot,\tau))\ast\Levy^\mu[\omega_\delta]\big)(x)\right|\dd \tau\dd x\\
&=\int_s^t\int_K\int_{\R^N}\left|\varphi_\eta(u_\eta(x-y,\tau))\right|\left|\Levy^\mu[\omega_\delta](y)\right|\dd y\dd x\dd \tau\\
&\leq\|\varphi_\eta(u_\eta)\|_{L^\infty(Q_T)}\|\Levy^\mu[\omega_\delta]\|_{L^1(\R^N)}|K||t-s|,
\end{split}
\end{equation*}
where $|K|$ denotes the Lebesgue measure of the compact set $K$. 
As in the proof of Theorem \ref{stability} (see \eqref{varphiboundedstability}), we obtain for $\eta$ sufficiently small
\begin{equation*}
\|\varphi_\eta(u_\eta)\|_{L^\infty(Q_T)}\leq\sup\{|\varphi(r)|:|r|\leq \|u_0\|_{L^\infty(\R^N)}\}+1.
\end{equation*}
Moreover, Lemma \ref{propnonlocal} (b) and the properties of mollifiers yield 
\begin{equation*}
\|\Levy^{\mu}[\omega_\delta]\|_{L^1(\R^N)}\leq (\delta^{-2}\|D^2\omega\|_{L^1(\R^N)}+1)\int_{|z|>0}\min\{|z|^2,1\}\dd \mu(z).
\end{equation*}
Hence, taking $\delta^2:=|t-s|^\frac{2}{3}$ we see that 
\begin{equation}\label{rem1}
\int_K|u_\eta^\delta(x,t)-u_\eta^\delta(x,s)|\dd x\leq \tilde{C}_{K,\varphi,u_0,\mu}\left(|t-s|^{\frac{1}{3}}+|t-s|\right),
\end{equation}
where 
\begin{equation*}
\begin{split}
&\tilde{C}_{K,\varphi,u_0,\mu}\\
&=|K|\Big(\|D^2\omega\|_{L^1(\R^N)}+1\Big)\Big(\sup_{|r|\leq \|u_0\|_{L^\infty}}|\varphi(r)|+1\Big)\int_{|z|>0}\min\{|z|^2,1\}\dd \mu(z).
\end{split}
\end{equation*}

By the triangle inequality and Theorem \ref{propapproxparabolic} (b),
\begin{equation*}
\begin{split}
&\int_K|u_\eta(x,t)-u_\eta(x,s)|\dd x\\
&\leq \int_K|u_\eta(x,t)-u_\eta^\delta(x,t)|\dd x + \int_K|u_\eta^\delta(x,t)-u_\eta^\delta(x,s)|\dd x\\
&\quad + \int_K|u_\eta^\delta(x,s)-u_\eta(x,s)|\dd x\\
&\leq \sup_{|\sigma|\leq\delta}\|u_\eta(\cdot,t)-u_\eta(\cdot+\sigma,t)\|_{L^1(\R^N)}+\int_K|u_\eta^\delta(x,t)-u_\eta^\delta(x,s)|\dd x\\
&\quad+\sup_{|\sigma|\leq\delta}\|u_\eta(\cdot,s)-u_\eta(\cdot+\sigma,s)\|_{L^1(\R^N)}\\
&\leq 2\sup_{|\sigma|\leq\delta}\|u_0-u_0(\cdot+\sigma)\|_{L^1(\R^N)}+\int_K|u_\eta^\delta(x,t)-u_\eta^\delta(x,s)|\dd x\\
&\leq2\max_{|\sigma|\leq \delta}\tilde{\lambda}_{u_0}(\delta)+\int_K|u_\eta^\delta(x,t)-u_\eta^\delta(x,s)|\dd x,
\end{split}
\end{equation*}
where $\tilde{\lambda}_{u_0}$ is defined in Step 2 . Hence, by \eqref{rem1}
\begin{equation*}
\begin{split}
\int_K|u_\eta(x,t)-u_\eta(x,s)|\dd x &\leq \lambda_{u_0}\left(|t-s|^{\frac{1}{3}}\right)+\tilde{C}_{K,\varphi,u_0,\mu}\left(|t-s|^{\frac{1}{3}}+|t-s|\right)\\
&=:\Lambda_{K,\varphi,u_0,\mu}(|t-s|)
\end{split}
\end{equation*}
for some moduli of continuity $\lambda_{u_0}$ and $ \Lambda_{K,\varphi,u_0,\mu}$.
\medskip

\noindent
{\em Step 4:} The assumptions of Theorem A.8 in \cite{HoRi07} hold by Steps 1--3, so we conclude that there is a subsequence $\{u_{\eta_n}\}_{n\in\N}$ such that
$$
u_{\eta_n}\to u \qquad\text{in}\qquad C([0,T];L_\textup{loc}^1(\R^N))
$$
as $\eta_n\to0^+$. Finally, $u$ inherits the properties of $u_\eta$ given in Theorem \ref{propapproxparabolic} (c) by Fatou's lemma, and the fact that the limit of a uniformly bounded sequence which converges a.e. is also bounded. 
\end{proof}

\begin{remark}\label{remark:convuapproxparabolic}
If $\Levy^\mu$ was not fixed in
  the above result, but rather $\mu=\mu_n$ (with $\mu_n$ satisfying
  \eqref{muassumption}), then the result still holds and the proof is
  the same provided we also assume that for some $a>0$ there exists a
  function $f\in 
  L_\textup{loc}^\infty((0,a))$ such that
$$
\|\Levy^{\mu_n}[\omega_\delta]\|_{L^1(\R^N)}\leq f(\delta) \qquad\text{for every}\qquad\delta\in(0,a),
$$ 
where $\omega_\delta$ is defined by \eqref{mollifierspace}. 
Observe that the above inequality follows from the assumption $\sup_n\int_{|z|>0}\min\{|z|^2,1\}\dd\mu_n(z)<\infty$ in Theorem \ref{compactness}.
\end{remark}

Now, the proofs of the existence and the a priori results follow.

\begin{proof}[Proof of Theorem \ref{existenceparabolic}]
Let $u_{\eta_n}$ be the solutions of
\eqref{mollifieddiffusionfunction} (cf. Theorem \ref{propapproxparabolic}), 
$u\in L^\infty(Q_T)\cap L^1(Q_T)\cap C([0,T];L_\textup{loc}^1(\R^N))$
the function provided by Lemma \ref{convuapproxparabolic}, and 
define $\Levy^{\mu_n}:=\Levy^{\mu}$ (that is,
$\Levy^{\mu_n}:=\Levy^{\mu}=\Levy$), $\varphi_n:=\varphi_{\eta_n}$,
and $u_n:=u_{\eta_n}$. Then assumptions (i), (ii), and (iii) in
Theorem \ref{stability} are satisfied by the $n$-independence of
$\Levy^\mu$, \eqref{mollifieddiffusionfunction}, and Lemma
\ref{convuapproxparabolic}. Moreover,
$\sup_n\|u_n\|_{L^\infty(Q_T)}\leq \|u_0\|_{L^\infty(\R^N)}<\infty $ by Theorem \ref{propapproxparabolic}
(c).  
Hence, by Theorem \ref{stability}, $u$ satisfies \eqref{E} in the sense of distributions: cf. Lemma \ref{equivdistsol} and Definition \ref{distsol}. Moreover, we have that $u-u_0\in L^1(Q_T)$. So, $u$ is in fact a distributional solution of \eqref{E} in the sense of Definition \ref{distsol}, and it is unique by Corollary \ref{uniquenessresult}.

Thus, any subsequence has the same limit, and hence, the whole sequence $\{u_{\eta}\}_{\eta>0}$ converges since it is bounded by Theorem \ref{propapproxparabolic} (c).
\end{proof}

\begin{proof}[Proof of Theorem \ref{consequences}]
(a) Let $u_\eta$ be the entropy solution of
\eqref{mollifieddiffusionfunction} (cf. Theorem
\ref{propapproxparabolic}). Using the semi entropy-entropy flux pairs 
$$
(u_\eta-k)^{\pm}\qquad\text{and}\qquad \pm\sgn^{\pm}(u_\eta-k)(f(u_\eta)-f(k))\qquad\text{for all $k\in\R$},
$$
and the corresponding definitions for entropy solutions in stead of the Kru\v{z}kov entropy-entropy flux pairs in \cite{CiJa11}, we obtain
$$
\int_{\R^N}(u_\eta(x,t)-\uu_\eta(x,t))^+\dd x \leq \int_{\R^N}(u_0(x)-\uu_0(x))^+\dd x
$$
for $u_\eta, \uu_\eta \in L^\infty(Q_T)\cap L^1(Q_T)\cap
C([0,T];L_\textup{loc}^1(\R^N))$ with initial data $u_0, \uu_0\in
L^\infty(\R^N)\cap L^1(\R^N)$. See \cite{EnJa14} for the
result and a proof.

By Lemma \ref{convuapproxparabolic}, we can take subsequences such that $u_{\eta_n},\uu_{\eta_n}\to u,\uu$ a.e. in $Q_T$ as ${\eta_n}\to0^+$. Thus, Fatou's lemma yield the result.
\medskip

\noindent
(b) By the contraction estimate obtained in part (a) and $u_0\leq\uu_0$ a.e. in $\R^N$, for all $t\in(0,T)$, $\int_{\R^N}(u(x,t)-\uu(x,t))^+\dd x\leq0$. Hence, $(u-\uu)^+=0$ and $u\leq\uu$ a.e. in $Q_T$.
\medskip

\noindent
(c) Follows by Lemma \ref{convuapproxparabolic}.
\medskip

\noindent
(d) Follows by Lemma \ref{convuapproxparabolic}.
\medskip

\noindent
(e) Using the triangle inequality, and taking $u,u_{\eta_n}$ as in Lemma \ref{convuapproxparabolic}, we obtain by Step 3 in the proof of that lemma that for all $t,s\in[0,T]$ and any compact set $K\subset\R^N$
\begin{equation*}
\begin{split}
&\|u(\cdot,t)-u(\cdot,s)\|_{L^1(K)}\\
&\leq\|u(\cdot,t)-u_{\eta_n}(\cdot,t)\|_{L^1(K)}+\|u_{\eta_n}(\cdot,t)-u_{\eta_n}(\cdot,s)\|_{L^1(K)}+\|u_{\eta_n}(\cdot,s)-u(\cdot,s)\|_{L^1(K)}\\
&\leq 2\|u(\cdot,t)-u_{\eta_n}(\cdot,t)\|_{C([0,T];L_{\textup{loc}}^1(\R^N))}+\Lambda_{K,\varphi,u_0,\mu}(|t-s|)
\end{split}
\end{equation*}
for the modulus of continuity $\Lambda_{K,\varphi,u_0,\mu}$ (see the above
mentioned proof). Since $u_{\eta_n}\to u$ in $C([0,T];L_{\textup{loc}}^1(\R^N))$ by Lemma \ref{convuapproxparabolic}, the proof is complete.
\medskip

\noindent
(f) Consider a standard cut-off function $0\leq\mathcal{X}\in C_\textup{c}^\infty(\R^N)$ such that $\mathcal{X}(x)=1$ for $|x|\leq1$ and $\mathcal{X}(x)=0$ for  $|x|\geq2$. We will write $\mathcal{X}_R(x)=\mathcal{X}(\frac{x}{R})$ for $R>0$. 
Following the proof of Lemma \ref{dissolprop} (b), with $\theta_a$ as defined there, we can take $\psi(x,t)=\mathcal{X}_R(x)\theta_a(t)$ for any $R>0$ as a test function in Definition \ref{distsol} (cf. Lemma \ref{equivdistsol}). Hence
\begin{equation*}
\begin{split}
\frac{1}{a}\int_{s-a}^{s} \int_{\R^N}u(x,t)\mathcal{X}_R(x)\dd x \dd t=& \int_{0}^{s}\theta_a(t)\int_{\R^N} \varphi(u(x,t)) \Levy^\mu[\mathcal{X}_R](x) \dd x\dd t\\
&+ \int_{\R^N}u_0(x)\mathcal{X}_R(x)\dd x.
\end{split}
\end{equation*}

Since $\mathcal{X}_R$ is compactly supported and $u\in C([0,T]; L^1_\textup{loc}(\R^N))$, we can pass to the limit as $a \to0^+$ in the first integral to get $\int_{\R^N} u(x,s)\mathcal{X}_R(x)\dd x$.  For the second integral, we know that $\varphi(u)\in L^\infty(Q_T)$, $\Levy^\mu[\mathcal{X}_R]\in L^1(\R^N)$ and $\theta_a\to \mathbf{1}_{[0,s)}$ pointwise a.e. as $a\to 0^+$, and thus, it converges as $a\to0^+$ to $\int_0^s\int_{\R^N}\varphi(u(x,t)) \Levy^\mu[\mathcal{X}_R](x) \dd x\dd t$ by Lebesgue's dominated convergence theorem. In this way, we get
$$
\int_{\R^N} u(x,s)\mathcal{X}_R(x)\dd x=\int_0^s\int_{\R^N}\varphi(u(x,t)) \Levy^\mu[\mathcal{X}_R](x) \dd x\dd t+ \int_{\R^N}u_0(x)\mathcal{X}_R(x)\dd x.
$$

The function $\mathcal{X}_R$ converges pointwise as $R\to \infty$ to 1, and it is also bounded by 1. Then, since $u(\cdot,s),u_0 \in L^1(\R^N)$, Lebesgue's dominated convergence theorem allows us to pass to the limit as $R\to \infty$ in the first and the last integrals to get $\int_{\R^N} u(x,s)\dd x$ and  $\int_{\R^N} u_0(x)\dd x$, respectively, for all $s\in(0,T)$. Consider the nonsingular part of the L\'evy operator, i.e.,
$
\int_{|z|>1} \mathcal{X}_R\left(x+z\right)- \mathcal{X}_R\left(x\right)\dd \mu(z)
$
which is bounded by $2\mu(\{z\in\R^N: |z|>1\})$ for every $x\in\R^N$. Since $\mathcal{X}_R(y)\to 1$ pointwise as $R\to\infty$ for all $y\in \R^N$, Lebesgue's dominated convergence theorem shows the pointwise  convergence to 0 of the nonsingular part. For the singular part, Lemma \ref{propnonlocal} (b) gives
\begin{equation*}
\left|\int_{0<|z|\leq1} \mathcal{X}_R\left(x+z\right)- \mathcal{X}_R\left(x\right)\dd \mu(z)\right|\leq \frac{1}{R^2}\|D^2 \mathcal{X} \|_{L^\infty(\R^N)}\int_{|z|\leq1}|z|^2 \dd \mu(z)
\end{equation*}
which also goes to 0 as $R\to \infty$. Moreover, by the assumption $|\varphi(r)|\leq L_\delta |r|$ for $|r|\leq \delta$,
\begin{equation*}
\begin{split}
\|\varphi(u(x,t))\|_{L^1(Q_T)}\leq \int_0^T \int_{|u|\leq \delta} L_\delta|u(x,t)|\dd x\dd t+ \|\varphi(u)\|_{L^\infty(Q_T)}\int_0^T \int_{|u|> \delta} \dd x\dd t.
\end{split}
\end{equation*}
Since $u\in L^1(Q_T)$, both terms on the right-hand side of the estimate above are finite.  Then by Lebesgue's dominated convergence theorem, 
$$
\left|\int_0^s\int_{\R^N}\varphi(u(x,t)) \Levy^\mu[\mathcal{X}_R](x) \dd x\dd t\right| \to 0 \qquad \text{as} \qquad R\to\infty.
$$
The proof is complete.
\end{proof}

\section{Applications of stability} 
\label{sec:appl}

This section focuses on proving the results stated in Sections \ref{sec:rescompconv} and \ref{sec:num}.

\subsection{Compactness, local limits and continuous dependence}
\label{sec:compconv} 
\begin{proof}[Proof of Theorem \ref{compactness}]
(a) Note that the sequence of solutions $\{u_n\}_{n\in\N}$ satisfy the hypothesis of Theorem \ref{consequences}. By the assumptions, Remark \ref{remark:convuapproxparabolic}, and Lemma \ref{convuapproxparabolic}, the result follows.
\medskip

\noindent
(b) This is a consequence of the stability given in Theorem \ref{stability}. For the initial condition, note that by the assumption $\sup_n\|u_{0,n}\|_{L^\infty(\R^N)}<\infty$ and Fatou's lemma, $u_0\in L^\infty(\R^N)\cap L^1(\R^N)$, and the convergence of $\int_{\R^N}u_{0,n}(x)\psi(x,0)\dd x$ follows by the $L_\textup{loc}^1$-convergence of $\{u_{0,n}\}_{n\in\N}$.
\end{proof}

\begin{lemma}\label{convergenceoffraclaptolap}
Assume \eqref{muassumption}, $s\in(0,2)$,
$\Levy^\mu=-(-\Delta)^{\frac{s}{2}}$, and $\psi\in
C_\textup{c}^\infty(\R^N)$. Then 
$$
\lim_{s\to2^-}\left\|-(-\Delta)^{\frac{s}{2}}\psi-\Delta\psi\right\|_{L^1(\R^N)}=0.
$$
\end{lemma}

\begin{proof}
The fractional Laplacian has a representation in the form
\eqref{deflevy} with measure
$$\dd \mu=c_{N,s}\frac{\dd z}{|z|^{N+s}}\qquad \text{for}\qquad
c_{N,s}=\left(N\int_{\R^N}\frac{1-\cos(z_1)}{|z|^{N+s}}\dd
  z\right)^{-1},$$
where
\begin{equation}\label{constantfraclap}
\lim_{s\to2^-}c_{N,s}=0,
\end{equation}
see e.g. Proposition 4.1 in \cite{Hitch2012}. Hence
\begin{equation*}
\begin{split}
-(-\Delta)^{\frac{s}{2}}\psi(x)&=
c_{N,s}\int_{|z|\leq1}\frac{\psi(x+z)-\psi(x)-z\cdot
  D\psi(x)}{|z|^{N+s}}\dd z\\ 
&\quad +c_{N,s}\int_{|z|>1}\frac{\psi(x+z)-\psi(x)}{|z|^{N+s}}\dd z,
\end{split}
\end{equation*}
where the last term goes to zero in $L^1(\R^N)$ as $s\to2^-$ since it is bounded in $L^1(\R^N)$ by $c_{N,s}2\|\psi\|_{L^1(\R^N)} \int_{|z|>1}|z|^{-N-1}\dd z$  for $s \geq 1$.

The explicit form of $c_{N,s}$ given in \eqref{constantfraclap} yields
\begin{equation*}
\begin{split}
\Delta\psi(x) &=\Delta\psi(x)c_{N,s}N\int_{|z|\leq1}\frac{1-\cos(z_1)}{|z|^{N+s}}\dd z+\Delta\psi(x)c_{N,s}N\int_{|z|>1}\frac{1-\cos(z_1)}{|z|^{N+s}}\dd z.
\end{split}
\end{equation*}
Again, the last term goes to zero in $L^1(\R^N)$ as $s\to2^-$ since $|1-\cos(z_1)|\leq2$ and then it is bounded in $L^1(\R^N)$ by $c_{N,s}2N \int_{|z|>1}|z|^{-N-1}\dd z$ for $s \geq1$. Using Taylor's theorem, we see that $1-\cos(z_1)=\frac{1}{2}z_1^2-\frac{1}{24}z_1^4\cos(\xi)$ for some $\xi\in[0,z_1]$. Hence,
\begin{equation}\label{tayloronconstant}
\begin{split}
\int_{|z|\leq1}\frac{1-\cos(z_1)}{|z|^{N+s}}\dd z=\frac{1}{2}\int_{|z|\leq1}\frac{z_1^2}{|z|^{N+s}}\dd z-\frac{1}{24}\cos(\xi)\int_{|z|\leq1}\frac{z_1^4}{|z|^{N+s}}\dd z,\\
\end{split}
\end{equation}
and the following estimate holds:
\begin{equation*}
\begin{split}
\left\|\frac{N}{24}\Delta\psi(x)c_{N,s}\int_{|z|\leq1}\frac{\cos(\xi)z_1^4}{|z|^{N+s}}\dd z\right\|_{L^1(\R^N)}&\leq\frac{N}{24}c_{N,s}\|\Delta\psi\|_{L^1(\R^N)}\int_{|z|\leq1}\frac{1}{|z|^{N-2}}\dd z
\end{split}
\end{equation*}
which goes to zero since $\int_{|z|\leq 1}\frac{1}{|z|^{N-2}}\dd
z<\infty$ and \eqref{constantfraclap} hold.

To estimate the remaining
term in \eqref{tayloronconstant}, note that for all $r>0$,
\begin{equation*}
N\Delta\psi(x)\int_{|z|\leq r}z_1^2\dd z=\Delta\psi(x)\int_{|z|\leq r }|z|^2\dd z=\int_{|z|\leq r}D^2\psi(x) z\cdot z\dd z,
\end{equation*}
and then
\begin{equation*}
\begin{split}
\frac{1}{2}c_{N,s}N\Delta\psi(x)\int_{|z|\leq1}\frac{z_1^2}{|z|^{N+s}}\dd z
=c_{N,s}\int_{|z|\leq1}\frac{\frac{1}{2}D^2\psi(x) z\cdot z}{|z|^{N+s}}\dd z.
\end{split}
\end{equation*}
We combine the all above estimates to get
\begin{equation*}
\begin{split}
&\lim_{s\to2^-}\|-(-\Delta)^{\frac{s}{2}}\psi(x)-\Delta\psi(x)\|_{L^1(\R^N)}\\
&=\lim_{s\to2^-}c_{N,s}\left\|\int_{|z|\leq1}\frac{\psi(x+z)-\psi(x)-z\cdot
    D\psi(x)-\frac{1}{2}D^2\psi(x) z\cdot z}{|z|^{N+s}}\dd
  z\right\|_{L^1(\R^N)}+0\\
&\leq \lim_{s\to2^-}c_{N,s}\frac16\|D^3\psi\|_{L^1(\R^N)}\int_{|z|\leq1}\frac{|z|^3}{|z|^{N+s}}\dd z,
\end{split}
\end{equation*}
where the last inequality follows from Taylor's and Fubini's
theorems. Since the $z$-integral is bounded by $\int_{|z|\leq
  1}\frac{1}{|z|^{N-1}}\dd z<\infty$ and \eqref{constantfraclap} hold,
the limit is zero and the proof is complete.
\end{proof}

\begin{proof}[Proof of Corollary \ref{convtothelocalprob}]
(a) We will use Theorem \ref{compactness} and Remark
\ref{remark:convuapproxparabolic} to prove the result, and now we verify the
assumptions. By Lemma \ref{convergenceoffraclaptolap},
$-(-\Delta)^{\frac{s}{2}}\psi\to\Delta\psi$ in $L^1(\R^N)$ as
$s\to2^-$ for all $\psi\in C_\textup{c}^\infty(\R^N)$. Moreover, by
Lemma \ref{propnonlocal} (b), properties of mollifiers, 
$\lim_{s\to2^-}\int_{|z|\leq1}|z|^2\frac{c_{N,s}\dd z}{|z|^{N+s}}=1$, and
$\lim_{s\to2^-}\int_{|z|>1}\frac{c_{N,s}\dd z}{|z|^{N+s}}=0$ (see
previous proof), 
\begin{equation*}
\begin{split}
\|(-\Delta)^{\frac{s}{2}}\omega_\delta\|_{L^1}&=\frac12\|D^2\omega_\delta\|_{L^1}\int_{|z|\leq1}|z|^2\frac{c_{N,s}\dd z}{|z|^{N+s}}
+2\|\omega_\delta\|_{L^1}\int_{|z|>1}\frac{c_{N,s}\dd z}{|z|^{N+s}}\\
&\leq C\Big(1+\frac{1}{\delta^2}\|D^2\omega\|_{L^1}\Big) 
\end{split}
\end{equation*}
for $s$ close to $2$. Hence, since also $\varphi$ is fixed
(independently of $s$), we may use Theorem \ref{compactness} and
Remark \ref{remark:convuapproxparabolic} to get a
subsequence $\{u_{s_j}\}_{j\in\N}$ and a $u\in
C([0,T];L_\textup{loc}^1(\R^N))$ such that $u_{s_j}\to u$ in
$C([0,T];L_\textup{loc}^1(\R^N))$ as $j\to\infty$. Finally, the
uniqueness for the limit (equation) \cite{BrCr79}, and the boundedness
of the sequence $\{u_s\}_{s\in(0,2)}$ (Theorem \ref{consequences} (d)), ensures
that the whole sequence converges.  

\medskip
\noindent (b) Since
$(-\Delta)^{\frac{s_n}{2}}\psi\to(-\Delta)^{\frac{\bar{s}}{2}}\psi$ in
$L^1(\R^N)$ as $n\to\infty$ (a similar argument as in Lemma
\ref{convergenceoffraclaptolap}), $\varphi_{m_n}(r)=r^{m_n}\to
\varphi_{\bar m}(r)=r^{\bar{m}}$ locally
uniformly as $n\to\infty$, and $\|(-\Delta)^{\frac{s_n}{2}}\omega_\delta\|_{L^1}\leq C(1+\delta^{-2})$ by
the proof of part (a), convergence for a subsequence follows by Theorem
\ref{compactness}. Moreover, the convergence of the whole sequence 
follows from uniquenes of the limit (Corollary
\ref{uniquenessresult}) and boundedness of the sequence (Theorem
\ref{consequences} (d)). 
\end{proof}

\subsection{Numerical approximation, convergence and existence}
\label{sec:pfconvnum}

We start by showing that a standard finite difference approximations of the
Laplacian can be written in the from \eqref{deflevy} and that
convergence of the resulting scheme then follows from our theory.
\begin{example}\label{numericalmethodexample}
Let $e_i\in \mathbb{R}^n$ for $i=1,...,N$ be points with $i$-th
component $1$ and the other components $0$. Using $\delta$-measures
and $h>0$, we define
\begin{equation*}
\mu_h=\sum_{i=1}^N \frac{\delta_{he_i}+\delta_{-he_i}}{h^2}.
\end{equation*}
It is clear that $\mu_h$ is a measure satisfying \eqref{muassumption}  for every
$h>0$. Moreover,
\begin{equation*}
\begin{split}
\Levy^{\mu_h}[v](x):=&\int_{\R^N} v(x+z)-v(x)\ d\mu_h(z)=\displaystyle\sum_{i=1}^N \frac{v(x+he_i) +
v(x-he_i)-2v(x)}{h^2} .
\end{split}
\end{equation*}
With $\mu=\mu_h$, problem \eqref{En} can be reformulated as
\begin{equation}
\label{**}
\dell_t u_h(x,t) -\displaystyle\sum_{i=1}^N
\frac{\varphi(u_h(x+he_i,t)) +
  \varphi(u_h(x-he_i,t))-2\varphi(u_h(x,t))}{h^2}=0
\end{equation}
in  $\mathcal{D}'(Q_T)$. 

For $\psi \in
C_\textup{c}^\infty(\R^N)$, an application of Taylor's theorem reveals that
there is a $C>0$ such that
\begin{equation*}
\int_{\R^N}\left|\Levy^{\mu_h}[\psi](x)- \Delta \psi(x)\right|\dd x\leq
h^2 C \|D^4 \psi\|_{L^1(\R^N)}\to0 \quad \text{as} \quad h\to0^+.
\end{equation*}
Moreover, for $h$ small enough,
$$
\sup_h \int_{|z|>0} \min\{|z|^2,1\}\dd \mu_h(z)= \sup_h\sum_{i=1}^N\frac{{|he_i|^2}+|-he_i|^2}{h^2} =2N
$$
Then by Theorem \ref{compactness}, there existis a subsequence
$\{u_{h_j}\}_{j\in\N}$ of solutions of \eqref{**}, and a $u\in
C([0,T]; L_\textup{loc}^1(\R^N))$ such that $u_{h_j}\to u$ in
$C([0,T]; L_\textup{loc}^1(\R^N))$ as $ j\to \infty$. Moreover, the
limit $u$ satisfies equation \eqref{GPME}: 
\begin{equation*}
\dell_t u -\Delta\varphi(u)=0 \qquad\text{in}\qquad \mathcal{D}'(Q_T).
\end{equation*}
In fact, as in the proof of Corollary \ref{convtothelocalprob}, the
whole sequence $\{u_h\}_{h>0}$ converges. 
\end{example}

We can proceed as in this example to get convergence for a more general class of second order local operators.
\begin{lemma}\label{discretelocal}
Assume $h>0$, $P\in \N$, $\sigma=(\sigma_1,...,\sigma_P)$,
$\sigma_i\in \R^N$ for $i=1,...,P$, $L^\sigma_h$ is defined by
\eqref{AproxLoc}, and $\psi\in C_\textup{c}^\infty(\R^N)$. Then
\begin{equation*}
L^\sigma_h[\psi](x)=\int_{|z|>0} \lef\psi(x+z)-\psi(x) \rig\dd \mu_h(z)=:\Levy^{\mu_{h,\sigma}}[\psi](x),
\end{equation*}
where the measure $\mu_{h,\sigma}=\frac{1}{h^2}\sum_{i=1}^P
(\delta_{h\sigma_i}+\delta_{-h\sigma_i})$. Moreover, $\mu_{h,\sigma}$ satisfies \eqref{muassumption},
$$
\sup_h\int_{|z|>0}\min\{|z|^2,1\}\dd \mu_{h,\sigma}(z)<\infty,
$$
and
\begin{equation*}
\|\Levy^{\mu_{h,\sigma}}[\psi]- \textup{tr}[\sigma \sigma^T D^2
\psi]\|_{L^1}\to 0 \qquad \text{as} \qquad h\to 0^+.
\end{equation*}
\end{lemma} 
\begin{proof}
By an elementary identity and
Talyor's theorem,
\begin{align*}
&\textup{tr}[\sigma \sigma^T D^2 \psi(x)]=\sum_{i=1}^P (\sigma_i\cdot D)^2\psi(x)
\\
&= \sum_{i=1}^P \frac{v(x+h\sigma_i) +
v(x-h\sigma_i)-2v(x)}{h^2}+h^2\sum_{i=1}^N\sum_{|\beta|=4}\frac1{\beta!}\sigma_i^\beta
D^\beta\psi(\xi_i)
\end{align*}
Here we use standard multiindex notation, with multiindex
$\beta=(\beta_1,\dots,\beta_n)\in \mathbb{N}^N$, to account for the $4$-th
order derivatives. Since the first term of the last line is
$L_h^\sigma[\psi](x)$, the rest of the proof follows along the
arguments of Example \ref{numericalmethodexample}.
\end{proof}

We aim to consider the general operator $\Levy^\mu$ defined in
\eqref{deflevy}. In order to use our stability result, we would like
to prove that the operator $\Levy^\mu_h$ defined in \eqref{AproxOp} is
a particular case of the operators studied in this paper. The
following result ensures this fact.
\begin{lemma}\label{discretenonlocal}
Assume \eqref{muassumption}, $h>0$, $\Levy^\mu_h$ is defined in
\eqref{AproxOp}, and $\psi\in C_\textup{c}^\infty(\R^N)$. Then
\begin{equation*}
\Levy^\mu_h[\psi](x)=\int_{|z|>0} \lef\psi(x+z)-\psi(x) \rig\dd \nu_h(z)=:\Levy^{\nu_h}[\psi](x)
\end{equation*}
where the measure
$\nu_h=\sum_{\alpha\not=0}\mu\left(z_\alpha+R_h\right)\delta_{z_\alpha}$.
Moreover, $\nu_h$ satisfies \eqref{muassumption} and
$$
\sup_h\int_{|z|>0}\min\{|z|^2,1\}\dd\nu_h(z)<\infty.
$$
\end{lemma} 

\begin{proof}
By the definition of $\delta_{z_{\alpha}}$, it immetiatly follows that
$\Levy^\mu_h=\Levy^{\nu_h}$.
It remains to show that $\nu_h$ satisfies \eqref{muassumption}. For
$h<1/\sqrt{N}$, 
\begin{equation*}
\int_{|z|>1}\dd\nu_h(z)=\sum_{|z_\alpha|>1}\mu\left(z_\alpha+R_h\right)\leq\mu\left(\left\{|z|>1-\sqrt{N}\frac{h}{2}\right\}\right)\leq \mu\left(\left\{|z|>\frac{1}{2}\right\}\right),
\end{equation*}
which is finite since $\mu$ satisfies \eqref{muassumption}. Moreover,
for $h>0$ small enough,
\begin{equation*}
\begin{split}
&\int_{|z|\leq1}|z|^2\dd\nu_h(z)\\
&\leq\sum_{0<|z_\alpha|\leq1}\int_{z_\alpha+R_h}\left|z_\alpha \right|^2\dd \mu(z)\leq \sum_{0<|z_\alpha|\leq1}\int_{z_\alpha+R_h}\left(|z|+\sqrt{N}\frac{h}{2} \right)^2\dd\mu(z)\\
&\leq \int_{h/2\leq |z|\leq 1+\sqrt{N}\frac{h}{2}}\left(|z|+\sqrt{N}\frac{h}{2} \right)^2\dd\mu(z)\leq\left(1+\sqrt{N}\right)^2\int_{ |z|\leq 2}|z|^2\dd\mu(z),
\end{split}
\end{equation*} 
which is also finite since $\mu$ satisfies \eqref{muassumption}. The
proof is complete. 
\end{proof}

\begin{lemma}\label{discretenonlocalconv}
Assume \eqref{muassumption}, $\Levy^\mu$ and $\Levy^\mu_h$ are defined
in \eqref{deflevy} and \eqref{AproxOp} respectively, and $\psi \in
C_\textup{c}^\infty(\R^N)$. Then 
\begin{equation*}
\|\Levy^\mu_h[\psi]-\Levy^\mu[\psi]\|_{L^1}\to 0 \qquad \text{as} \qquad h\to0^+.
\end{equation*}
\end{lemma}
\begin{proof}
The following inequality is just a use of the definitions,
\begin{equation*}
\begin{split}
&\int_{\R^N}\left| \Levy_{h}^\mu[\psi](x)-\Levy^\mu[\psi](x) \right|\dd x\\
&= \int_{\R^N} \Bigg|\sum_{\alpha\not=0}\left(\psi(x+z_\alpha) -\psi(x)\right)\int_{z_\alpha+R_h}\dd \mu(z)\\
&\qquad-\sum_{\alpha\in \mathbb{Z}^N}\int_{z_\alpha+R_h}\lef\psi(x+z) -\psi(x)\rig\dd \mu(z)\Bigg|\dd x\\
 &\leq \int_{\R^N}\Bigg(\left| \int_{R_h}\lef\psi(x+z) -\psi(x)\rig\dd \mu(z)  \right|\\
 &\qquad+ \left|\sum_{\alpha\not=0} \int_{z_\alpha+R_h}\lef\psi(x+z_\alpha) -\psi(x+z)\rig\dd \mu(z)\right|\Bigg)\dd x.
\end{split}
\end{equation*}
We will show that both terms go to zero with $h$. Indeed, for $|z|\leq1$ we have that $|z|^2 \mathbf{1}_{R_h}(z)\to 0$ pointwise as $h\to0^+$. Then, by Lebesgue's dominated convergence theorem, \eqref{muassumption}, and  Lemma \ref{propnonlocal} (b), we have as $h\to0^+$
\[
\int_{\R^N}\left| \int_{R_h}\lef\psi(x+z) -\psi(x)\rig\dd \mu(z)  \right|\dd x\leq \frac{1}{2}\|D^2\psi\|_{L^1}\int_{|z|\leq1}|z|^2\mathbf{1}_{R_h}(z)\dd \mu(z)\to0.
\]

For the second term, we need to consider separately the cases when
when we are close or far from the origin. First note that for any $z
\in z_\alpha+{R_h}$ we have that $|z_\alpha-z| \leq
\sqrt{N}\frac{h}{2} $. Since $\mu$ satisfies \eqref{muassumption} and
$\psi\in C_\textup{c}^\infty(\R^N)$, 
\begin{equation*}
\begin{split}
I_{\text{ext}}&:=\int_{\R^N}\left|\sum_{|\alpha|h>1}\int_{z_\alpha+{R_h}}\lef\psi(x+z_\alpha) -\psi(x+z)\rig\dd \mu(z)\right|\dd x\\
&\leq\|D\psi\|_{L^1(\R^N)}\sum_{|\alpha|h>1}\int_{z_\alpha+{R_h}}|z_\alpha-z|\dd \mu(z)\\
&\leq h \frac{\sqrt{N}}{2} \|D\psi\|_{L^1(\R^N)}\int_{|z|>1/2}\dd \mu(z)\to 0 \quad \text{as} \quad h\to0^+.
\end{split}
\end{equation*}
On the other hand, by the symmetry of $\mu$ and also of the term in the sum, we have that 
\begin{equation*}
\begin{split}
I_{\text{int}}&:=\int_{\R^N}\left|\sum_{0<|\alpha|h\leq1}\int_{z_\alpha+{R_h}}\lef\psi(x+z_\alpha) -\psi(x+z)\rig\dd \mu(z)\right|\dd x\\
&\leq \int_{\R^N}\left|\sum_{0<|\alpha|h\leq1}\int_{z_\alpha+{R_h}}\lef\psi(x+z_\alpha) -\psi(x+z)-(z_\alpha-z)\cdot D \psi(x)\rig \dd \mu(z)\right|\dd x.
\end{split}
\end{equation*}
We make use of the Taylor expansions
\begin{equation*}
\begin{split}
\psi(x+z_\alpha)&=\psi(x+z)+D \psi(x+z)\cdot(z_\alpha-z)+ G_1(x,z,z_\alpha)(z_\alpha-z)\cdot (z_\alpha-z) \\
D\psi(x+z)&=D\psi(x)+G_2(x,z)\cdot z
\end{split}
\end{equation*}
where   $\| G_1\|_{L^1(\R^N,\dd x)}+\| G_2\|_{L^1(\R^N, \dd x)}\leq C \|D^2\psi\|_{L^1(\R^N)}$ for some constant $C>0$. In this way,
\begin{equation*}
\begin{split}
I_{\text{int}}&\leq \int_{\R^N}\left|\sum_{0<|\alpha|h\leq1}\int_{z_\alpha+{R_h}}\lef G_2 z\cdot (z_\alpha-z) +G_1 (z_\alpha-z)\cdot (z_\alpha-z)\rig \dd \mu(z)\right|\dd x\\
&=C \|D^2\psi\|_{L^1(\R^N)}\sum_{0<|\alpha|h\leq1}\int_{z_\alpha+{R_h}}\lef |z||z_\alpha-z|+|z_\alpha-z|^2\rig\dd\mu(z)\\
&\leq C \|D^2\psi\|_{L^1(\R^N)} \int_{\frac{h}{2}<|z|\leq1+ \sqrt{N}\frac{h}{2}}\frac{h}{2}\left(|z|+\frac{h}{2}\right)^2\dd\mu(z) \to0 \quad \text{as} \quad h \to0^+.
\end{split}
\end{equation*}
Since the integrand is dominated by $2|z|^2$  which is an integrable function with respect to the measure $\mu$ on the set $\{z\in\R^N : |z|\leq1\}$ by \eqref{muassumption}, the last term goes to zero by Lebesgue's dominated convergence theorem.
\end{proof}

\begin{proof}[Proof of Proposition \ref{semi-discreteapproxproperties}]
Note that by Lemmas \ref{discretelocal} and \ref{discretenonlocal}, $L_h^\sigma$
and $\Levy_h^\mu$ are in the class of operators defined by
\eqref{deflevy} and \eqref{muassumption}.
\medskip 

\noindent(a) Existence, uniqueness and regularity follow from Theorem
\ref{existenceparabolic}. 

\medskip 

\noindent(b) Follows from Theorem \ref{consequences} (c) and (d) and interpolation.

\medskip 

\noindent(c) Lemmas \ref{discretelocal} and \ref{discretenonlocalconv} ensure the $L^1$-consistency.

\medskip 

\noindent(d) Follows from Theorem \ref{consequences} (b).

\medskip 

\noindent(e) Follows from Theorem \ref{consequences} (f).
\end{proof}

\begin{proof} [Proof of Proposition \ref{semi-discreteapprox}]
By Lemmas \ref{discretelocal} and \ref{discretenonlocal} and
Proposition \ref{semi-discreteapproxproperties} (c),
\begin{gather*}
\sup_{h}\int_{|z|>0}\min\{|z|^2,1\}\dd
(\mu_{h,\sigma}+\nu_h)(z)<\infty,
\intertext{and}
\|\left(L_h^\sigma + \Levy_h^\mu\right)[\psi]-\left(L^\sigma +
  \Levy^\mu\right)[\psi]\|_{L^1}\to 0\qquad\text{as}\qquad h\to 0^+.
\end{gather*}
Since also $\varphi$ and $u_0$ are fixed (that is, independent of
$h$), by Theorem \ref{compactness}
there is a subsequence $\{u_{h_n}\}_{n\in\N}$ of solutions of
\eqref{En}, that converge
in $C([0,T];L^1_{\mathrm{loc}}(\R^N))$ to
a function $u$. Moreover, this function $u$ is a
distributional solution of \eqref{limitprob}. Finally, $u$ also
belongs to $L^\infty(Q_T)\cap L^1(Q_T)$ by Proposition
  \ref{semi-discreteapproxproperties} (b) and Fatou's lemma.
\end{proof}

\begin{proof}[Proof of Corollary \ref{generallevyprocessexistence}]
  Any limit point $u$ from Proposition \ref{semi-discreteapprox} is a
  distributional solution of \eqref{limiteq} and \eqref{IC}. 
\end{proof}

\begin{proof}[Proof of Theorem \ref{semi-discreteapprox2}]
By Proposition \ref{semi-discreteapprox} 
 there is a converging subsequence with a
limit $u$ which has the right regularity and is a 
distributional solution of \eqref{limitprob}. Assume there is a
subsequence that converge to another limit $v$. Then by Proposition
\ref{semi-discreteapprox} again, there is a subsubsequence that
converge to a limit which is a distributional solution. By
uniqueness 
of the limit, $v$ is a distributional solution. But then $v=u$ by the
uniqueness given in Corollary \ref{uniquenessresult} for the case 
$\sigma\equiv0$ or the local result in \cite{BrCr79}. Hence all
subsequence limits are equal to $u$ and since the sequence
itself is bounded (Proposition \ref{semi-discreteapproxproperties}
(b)), the whole sequence converges to $u$.
\end{proof}

\section{Auxiliary elliptic equation}
\label{ellipticsection}

In this section we study the elliptic equation \eqref{elliptic}
introduced in Section \ref{proofofuniqueness} with the ultimate goal 
to prove Theorem \ref{collectedresultselliptic} and Lemma
\ref{symmetryB}.  
We will also need the following approximation of \eqref{elliptic} where the measure $\mu$ is replaced by $\mu_r:=\mathbf{1}_{|z|>r}\mu$:
\begin{equation}\label{approxelliptic}
\veps v_{\veps,r}(x)-\Levy^{\mu_r}[v_{\veps,r}](x)=g(x)  \quad\text{in}\quad \R^N,
\end{equation}
with $\veps>0$,
$$
\Levy^{\mu_r}[\psi](x)=\int_{|z|>0} \lef\psi(x+z)-\psi(x)\rig\dd\mu_r(z).$$
 Note that for any $r>0$, the operator $\Levy^{\mu_r}[\psi]$ is well-defined for merely bounded $\psi$, and that Lemma \ref{propnonlocal} also holds for $\Levy^{\mu_r}$: see Remark \ref{nonlocalwelldefined} (b).
Also recall the notation $B_\veps^\mu=(\veps 
I- \Levy^\mu)^{-1}$ and define $B_{\veps}^{\mu_r}:=(\veps I-\Levy^{\mu_r})^{-1}$.
\begin{remark}\label{propapproxelliptic}
\begin{enumerate}[(a)]
\item Since \eqref{elliptic} and \eqref{approxelliptic} are linear
  equations, we have formally, for any multiindex $\alpha\in
  \mathbb{N}^N$, that $D^\alpha v$ is a solution of \eqref{elliptic}
  or \eqref{approxelliptic} with right hand side $D^\alpha g$ if $v$
  is a solution of the same equation with right hand side $g$.
\item Let $\psi\in C_\textup{b}^2(\R^N)$, and let $p\in\{1,\infty\}$. Since 
$$
\left(\Levy^\mu-\Levy^{\mu_r}\right)[\psi](x)=\int_{|z|\leq r}\lef\psi(x+z)-\psi(x)-z\cdot D\psi(x)\rig\dd\mu(z),
$$
we have that $\Levy^{\mu_r}[\psi]\to\Levy^\mu[\psi]$ in $L^p(\R^N)$ as $r\to0^+$ by Lemma \ref{propnonlocal} (b) and Lebesgue's dominated convergence theorem.
\end{enumerate}
\end{remark}

\subsection{Preliminary results}
We will state and prove a very general Stroock-Varopoulos type of inequality
which is of independent interest. First we consider the bounded
operators $\Levy^{\mu_r}$.

\begin{lemma}\label{Stroock-Var}
Assume \eqref{muassumption}, $\psi\in L^\infty(\R^N)\cap L^1(\R^N)$ and $\zeta\in C(\R)$ is nondecreasing. Then for any $r>0$ we have,
\begin{equation*}
\begin{split}
I_r&:=\int_{\R^N} \zeta(\psi(x)) \Levy^{\mu_r} [\psi](x) \dd x\\
&=-\frac{1}{2}\int_{\R^N}\int_{|z|>0} \big(\zeta(\psi(x+z))-\zeta(\psi(x)) \big)\big(\psi(x+z)-\psi(x)\big)\dd\mu_r(z)\dd x,
\end{split}
\end{equation*}
and in particular, $I_r\leq0$. 
\end{lemma}

\begin{remark}\label{stroockwelldef}
More generally, the above lemma holds as long as the integral $I_r$ is
well-defined for $\psi$ and $\zeta(\psi)$. 
\end{remark}

In the proof we need a technical lemma
which wil be proven in Appendix \ref{appendix}. 

\begin{lemma}\label{lemmameasure}
Assume $\nu$ is a nonnegative, symmetric and locally finite Borel measure on $\R^N$. Let $\As,\Bs$ be Borel sets on $\R^N$, and let
\begin{equation*}
M_{1}(\As,\Bs)=\int_{\As}\left( \int_{\Bs-z}\dd\nu(x)\right)\dd z=\int_{\As}\nu(\Bs-z)\dd z.
\end{equation*}
\begin{equation*}
M_2(\As,\Bs)=\int_{\Bs}\left( \int_{\As-z}\dd\nu(x)\right)\dd z=\int_{\Bs}\nu(\As-z)\dd z.
\end{equation*}
Then $M_{1}(\As,\Bs)=M_2(\As,\Bs)$.
\end{lemma}

\begin{proof}[Proof of Lemma \ref{Stroock-Var}]
Observe that $\zeta(\psi)\in L^\infty(\R^N)$, and since $\int_{\R^N}|\Levy^{\mu_r}[\psi]|\dd x\leq2\|\psi\|_{L^1}\int_{|z|>r}\dd\mu(z)$, $\Levy^{\mu_r}[\psi]\in L^1(\R^N)$. Hence $I_r$ is well-defined.

By the symmetry of $\mu$, the gradient term in the nonlocal operator vanishes. Fubini's theorem and a relabelling of the variables gives
\begin{equation*}
\begin{split}
I_r
=&\,\int_{\R^N}  \zeta(\psi(x)) \int_{|z|>0} \lef\psi(x+z)-\psi(x)\rig  \dd\mu_r(z)\dd x
\\
=&\,\int_{\R^N}\int_{|z-x|>0} \zeta(\psi(x))\big(\psi(z))-\psi(x)  \big)\mathbf{1}_{|z-x|>r}\dd\mu_{-x}(z)\dd x
\\
=&\,\int_{\R^N}\int_{|x-z|>0} \zeta(\psi(z))\big(\psi(x)-\psi(z)\big)\mathbf{1}_{|x-z|>r}\dd\mu_{-z}(x)\dd z.
\end{split}
\end{equation*}
Since $\mu$ is a nonnegative, symmetric and finite Radon measure
on $\R^N$ (and hence a Borel measures), we can use Lemma
\ref{lemmameasure} to see that
\begin{align*}&\int_{\R^N}\int_{|x-z|>0}
  \zeta(\psi(z))\big(\psi(x)-\psi(z)\big)\mathbf{1}_{|x-z|>r}\dd\mu_{-z}(x)\dd
  z\\
&=\int_{\R^N}\int_{|x-z|>0}
\zeta(\psi(z))\big(\psi(x)-\psi(z)\big)\mathbf{1}_{|x-z|>r}\dd\mu_{-x}(z)\dd
x.
\end{align*}
It then follows that
\begin{equation*}
\begin{split}
2I_r=&\, -\int_{\R^N}\int_{|z-x|>0} \zeta(\psi(x))\big(\psi(x))-\psi(z)  \big)\mathbf{1}_{|z-x|>r}\dd\mu_{-x}(z)\dd x\\
&\,+\int_{\R^N}\int_{|x-z|>0} \zeta(\psi(z))\big(\psi(x)-\psi(z)\big)\mathbf{1}_{|x-z|>r}\dd\mu_{-x}(z)\dd x\\
=&-\int_{\R^N}\int_{|z-x|>0}\big(\zeta(\psi(x))-\zeta(\psi(z)) \big)\big(\psi(x)-\psi(z)\big)\mathbf{1}_{|z-x|>r}\dd\mu_{-x}(z) \dd x.
\end{split}
\end{equation*}
Since $ \big(\zeta(\psi(x))-\zeta(\psi(z)) \big)\big(\psi(x)-\psi(z)\big)\geq0$ for all $x,z\in \R^N$, $I_r\leq0$.
\end{proof}

Now we give the general result, considering the general nonlocal operator $\Levy^\mu$. 

\begin{corollary}[General Stroock-Varopoulos]
Assume \eqref{muassumption}, and $\zeta\in C^1(\R)$ such that $\zeta'\geq0$.
\begin{enumerate}[(a)]
\item Let $\psi\in W^{1,\infty}(\R^N)\cap W^{1,1}(\R^N)$. Then
\begin{equation*}
\begin{split}
I&:=\int_{\R^N} \zeta(\psi(x)) \Levy^{\mu} [\psi](x) \dd x\\
&=-\frac{1}{2}\int_{\R^N}\int_{|z-x|>0} \big(\zeta(\psi(z))-\zeta(\psi(x)) \big)\big(\psi(z)-\psi(x)\big)\dd\mu(z)\dd x\\
&\leq 0.
\end{split}
\end{equation*}
\item Let $\psi\in W^{2,\infty}(\R^N)\cap W^{2,1}(\R^N)$. If $Z\in C^2(\R)$ is such that $Z(0)=0$ and $(Z')^2=\zeta'$, then
\begin{equation*}
\begin{split}
\int_{\R^N} \zeta(\psi(x)) \Levy^{\mu} [\psi](x) \dd x\leq&-\frac{1}{2}\int_{\R^N}\int_{|z-x|>0} \big(Z(\psi(z))-Z(\psi(x)) \big)^2 d\mu_{-x}(z)\dd x.
\end{split}
\end{equation*}
Moreover,
\begin{equation*}
\begin{split}
\Big(Z(\psi),\Levy^\mu[Z(\psi)]\Big)_{L^2(\R^N)}&=-\frac{1}{2}\int_{\R^N}\int_{|z-x|>0} \big(Z(\psi(z))-Z(\psi(x)) \big)^2 d\mu_{-x}(z)\dd x\\
&=-\left\| (\Levy^\mu)^{\frac{1}{2}}[Z(\psi)]\right\|_{L^2}^2.
\end{split}
\end{equation*}
\end{enumerate}
\end{corollary}

\begin{remark}
The (energy) norm in part (b) is much studied when
$\Levy^\mu=-(-\Delta)^\frac{s}{2}$, $s\in(0,2)$, and $Z=I$
(see \cite{App09,Hitch2012}). In this case 
\[
\Big(\psi,(-\Delta)^\frac{s}{2}\psi\Big)_{L^2(\R^N)}=\frac{1}{2}\int_{\R^N}\int_{|z-x|>0} \frac{\big(\psi(z)-\psi(x) \big)^2}{|z-x|^{N+s}}\dd z\dd x=\Big\|(-\Delta)^{\frac{s}{4}}\psi\Big\|_{L^2(\R^N)}^2.
\]
This is called the Gagliardo (semi)norm of $\psi$ and is denoted by $[\psi]_{W^{\frac{s}{2},2}(\R^N)}$.
\end{remark}

\begin{proof}
(a) By Remark \ref{propapproxelliptic} (b),
\begin{equation*}
\left|\int_{\R^N}\zeta(\psi)(\Levy^{\mu}-\Levy^{\mu_r})[\psi]\dd x\right|\leq\|\zeta(\psi)\|_{L^\infty}\|(\Levy^{\mu}-\Levy^{\mu_r})[\psi]\|_{L^1}\to0 \qquad\text{as}\qquad r\to0^+,
\end{equation*}
and we may send $r\to0^+$ in Lemma \ref{Stroock-Var} to get
\begin{equation*}
\begin{split}
&\int_{\R^N} \zeta(\psi(x)) \Levy^{\mu} [\psi](x) \dd x\\
&=-\frac{1}{2}\lim_{r\to0^+}\int_{\R^N}\int_{|z-x|>r} \big(\zeta(\psi(z))-\zeta(\psi(x)) \big)\big(\psi(z)-\psi(x)\big)\dd\mu_{-x}(z)\dd x.
\end{split}
\end{equation*}
By the assumptions on $\zeta, \psi$ and \eqref{muassumption}, $\big(\zeta(\psi(z))-\zeta(\psi(x)) \big)\big(\psi(z)-\psi(x)\big)\geq0$ is integrable with respect to $\dd\mu_{-x}(z)\dd x$ on $\R^N\times \R^N\backslash \{0\}$ since 
\begin{equation}\label{integrabilityofzetapsi}
\begin{split}
&\int_{\R^N}\left|\int_{|z-x|>0}\big(\zeta(\psi(z))-\zeta(\psi(x)) \big)\big(\psi(z)-\psi(x)\big)\dd\mu_{-x}(z)\right|\dd x\\
&\leq \|\zeta'(\psi)\|_{L^\infty}\|D\psi\|_{L^\infty}\|D\psi\|_{L^1}\int_{|z|\leq1}|z|^2\dd\mu(z)\\
&\qquad+4\|\zeta(\psi)\|_{L^\infty}\|\psi\|_{L^1}\int_{|z|>1}\dd \mu(z).
\end{split}
\end{equation}
Thus, Lebesgue's dominated convergence theorem gives the desired result.
\medskip

\noindent
(b) For $a,b\in \R$, the Fundamental Theorem of Calculus and Jensen's inequality gives the following pointwise inequality:
\begin{equation}\label{fundineq}
\begin{split}
\left(Z(b)-Z(a) \right)^2=&\left(\int_a^b Z'(t)\dd t\right)^2\leq (b-a) \int_a^b \big(Z'(t)\big)^2\dd t\\
=&\,(b-a) \int_a^b \zeta'(t)\dd t =(b-a)(\zeta(b)-\zeta(a)).
\end{split}
\end{equation}

By the assumptions, we can easily check that $Z(\psi)\in W^{2,\infty}(\R^N)\cap W^{2,1}(\R^N)$. So the integral
$$
-\frac{1}{2}\int_{\R^N}\int_{|z-x|>0} \big(Z(\psi(z))-Z(\psi(x)) \big)^2 d\mu_{-x}(z)\dd x
$$
is well-defined using a similar argument as in \eqref{integrabilityofzetapsi}. Then, part (a) and \eqref{fundineq} gives the first result of part (b).

Next, part (a) yields
\[
\Big(Z(\psi),\Levy^\mu[Z(\psi)]\Big)_{L^2(\R^N)}=-\frac{1}{2}\int_{\R^N}\int_{|z-x|>0} \big(Z(\psi(z))-Z(\psi(x)) \big)^2 d\mu_{-x}(z)\dd x.
\]
Moreover, since $Z(\psi)\in W^{2,\infty}(\R^N)\cap W^{2,1}(\R^N)$, then by Lemma  \ref{propnonlocal} (b) and interpolation, both $Z(\psi)$ and $\Levy^\mu[Z(\psi)]$ are in $L^2(\R^N)$. We then conclude the proof by application of Lemma \ref{fouriersymbol} and Remark \ref{remarkfouriersymbol} (b).
\end{proof}

\subsection{Results for the approximate elliptic equation \eqref{approxelliptic}}
We will now focus on proving some a priori, uniqueness, existence, and
stability results for \eqref{approxelliptic}. 


\begin{proposition}\label{compapproxelliptic}
Assume \eqref{muassumption}.
\begin{enumerate}[(a)]
\item If $g\in L^\infty(\R^N)$ and $v_{\veps,r}\in L^\infty(\R^N)$ solves $\veps
  v_{\veps,r}-\Levy^{\mu_r}[v_{\veps,r}]\leq g$ a.e., then
$$
\veps\|(v_{\veps,r})^+\|_{L^\infty(\R^N)}\leq\|(g)^+\|_{L^\infty(\R^N)}
.$$
\item If $g\in L^\infty(\R^N)$ and $v_{\veps,r}\in L^\infty(\R^N)$ is
  an a.e. solution of \eqref{approxelliptic}, then
$$\veps\|v_{\veps,r}\|_{L^\infty(\R^N)}\leq \|g\|_{L^\infty(\R^N)}.$$
\item Let $g, \hat{g},v_{\veps,r},\vv_{\veps,r}\in L^\infty(\R^N)$, $\veps
  v_{\veps,r}-\Levy^{\mu_r}[v_{\veps,r}]\leq g$ a.e. and $\veps
  \vv_{\veps,r}-\Levy^{\mu_r}[\vv_{\veps,r}]\geq \hat g$ a.e. If $g\leq \hat{g}$ a.e., then
  $v_{\veps,r}\leq\vv_{\veps,r}$ a.e.
\end{enumerate}
\end{proposition}

\begin{proof}
(a) Assume first that $g,v_{\veps,r}\in
C_\textup{b}(\R^N)$. Then for all $\delta >0$ there exists a $x_\delta \in \R^N$ such that
\[
v_{\veps,r}(x_\delta)+\delta> \sup\{v_{\veps,r}\}.
\]
Then, since $v_{\veps,r}$ is an a.e. solution, 
\begin{equation*}
\begin{split}
\veps v_{\veps,r}(x_\delta)&\leq g(x_\delta)+\int_{|z|>0}\lef v_{\veps,r}(x_\delta+z)-v_{\veps,r}(x_\delta)\rig\dd \mu_r(z)\\
&\leq \|(g)^+\|_{L^\infty(\R^N)}+\int_{|z|>r} \lef\sup\{v_{\veps,r}\}-v_{\veps,r}(x_\delta)\rig\dd \mu(z)\\
&\leq  \|(g)^+\|_{L^\infty(\R^N)} +\delta \mu(\{z\in \R^N: |z|>r\}).
\end{split}
\end{equation*}
Hence,
\[
\veps  \sup\{v_{\veps,r}\} < \veps v_{\veps,r}(x)+\veps \delta \leq   \|(g)^+\|_{L^\infty(\R^N)} +\delta \big(\veps+\mu(\{z\in \R^N: |z|>r\})\big),
\]
and we pass to the limit as $\delta \to 0^+$ to get
\[
\veps  \sup\{v_{\veps,r}\} \leq \|(g)^+\|_{L^\infty}.
\]

In the general case, when $g,v_{\veps,r}\in L^{\infty}(\R^N)$, we need
a regularization argument. Let $v_{\veps,r}^\delta:=\omega_\delta\ast
v_{\veps,r}$ and mollify the inequality to see that 
$$\veps v_{\veps,r}^\delta-\Levy^{\mu_r}[v_{\veps,r}^\delta]\leq
g_\delta\qquad\text{in}\qquad \R^N.$$
By the first part of the proof and
the properties of mollifiers,
\begin{equation*}
v_{\veps,r}(x)\leq|v_{\veps,r}^\delta(x)-v_{\veps,r}(x)|+v_{\veps,r}^\delta(x)\leq o(1) + \frac{1}{\veps}\|(g)^+\|_{L^\infty(\R^N)} \quad \text{as} \quad \delta \to 0^+
\end{equation*}
for a.e. $x\in \R^N$. Part (a) follows.
\medskip

\noindent (b)
In a similar way as in (a), we find that $\veps  \sup\{-v_{\veps,r}\} \leq
\|(g)^-\|_{L^\infty(\R^N)}$ and combine with (a) to conclude that $\veps \|v_{\veps,r}\|_{L^\infty(\R^N)}\leq \|g\|_{L^\infty(\R^N)}$.
\medskip

\noindent (c) Since $w=v_{\veps,r}-\hat v_{\veps,r}$ solves $\veps
w-\Levy^{\mu_r}[w] \leq g-\hat g$, by (a) and the
assumptions, it follows that $\veps\sup \{w\}\leq \|(g-\hat g)^+\|_{L^\infty(\R^N)}=0$.
\end{proof}

\begin{proposition}[Existence and
  uniqueness]\label{existapproxelliptic}
Assume \eqref{muassumption}.
\begin{enumerate}[(a)]
\item If $g\in C_\textup{b}(\R^N)$, then there exists a unique classical solution $v_{\veps,r}\in C_\textup{b}(\R^N)$ of \eqref{approxelliptic}.
\item If $g\in L^\infty(\R^N)$, then there exists a unique a.e. solution $v_{\veps,r}\in L^\infty(\R^N)$ of \eqref{approxelliptic}.
\item If $g\in L^1(\R^N)$, then there exists a unique a.e. solution $v_{\veps,r}\in L^1(\R^N)$ of \eqref{approxelliptic}.
\item If $g\in C_\textup{b}^\infty(\R^N)$, then there exists a unique classical solution $v_{\veps,r}\in C_\textup{b}^\infty(\R^N)$ of \eqref{approxelliptic}. Moreover,
$$
\veps\|D^\alpha v_{\veps,r}\|_{L^\infty}\leq\|D^\alpha g\|_{L^\infty}
$$
for each multiindex $\alpha\in \mathbb{N}^N$.
\end{enumerate}
\end{proposition}

\begin{proof}
The proofs of (a), (b), and (c) follow from standard
  arguments using Banach's fixed point theorem. Let $X$ denote any of
  one of the spaces $C_\textup{b}(\R^N), L^\infty(\R^N)$, and
  $L^1(\R^N)$, and note that $X$ is a Banach space. Let the operator
  $T$ be such that \eqref{approxelliptic} is equivalent to the fixed
  point equation $T[u]=u$: 
$$
T[v_\veps](x):=\frac{1}{\veps+\int_{|z|>r}\dd \mu(z)}\left(\int_{|z|>r}v_\veps(x+z)\dd\mu(z)+g(x)\right).
$$
It is easy to check that $T$ is a bounded linear operator on
$X$, and straightforward computations also shows it is a contraction:
$$
\|T[v_\veps]-T[\hat{v}_\veps]\|_{X}\leq \alp\|v_\veps-\hat{v}_\veps\|_{X}\quad\text{for}\quad\alp=\frac{\int_{|z|>r}\dd\mu(z)}{\veps+\int_{|z|>r}\dd \mu(z)}<1.
$$
Hence by Banach's fixed point theorem there exists a unique
$v_\veps\in X$ such that $v_\veps=T[v_\veps]$ in $X$ and then also
a.e. (everywhere if $X=C_b$).
\medskip

\noindent (d) Let $v_{\veps,r}=B_{\veps}^{\mu_r}[g]$ and define $\delta_{i,h}\psi$ by
$$
\delta_{i,h}\psi(x)=\frac{\psi(x+he_i)-\psi(x)}{h}.
$$

By part (a), we have uniqueness for $C_\textup{b}(\R^N)$ solutions of \eqref{approxelliptic}. Hence, $v_{\veps,r}(x+he_i)=B_{\veps}^{\mu_r}[g(\cdot+he_i)](x)$, and then by uniqueness and linearity $\delta_{i,h}v=B_{\veps}^{\mu_r}[\delta_{i,h}g]$. In addition, there exists a unique $w_{i,\veps,r}\in C_\textup{b}(\R^N)$ such that $w_{i,\veps,r}=B_{\veps}^{\mu_r}[\dell_{x_i}g]$.

Using linearity and Proposition \ref{compapproxelliptic} (b), we get
$$
\veps\|w_{i,\veps,r}-\delta_{i,h}v_{\veps,r}\|_{L^\infty}\leq\|\dell_{x_i} g-\delta_{i,h}g\|_{L^\infty}.
$$
When $h\to0^+$, $\delta_{i,h}g\to\dell_{x_i}g$ uniformly on $\R^N$, and hence $\delta_{i,h}v_{\veps,r}\to w_{i,\veps,r}$ in $L^\infty$. This implies that $\dell_{x_i}v_{\veps,r}=w_{i,\veps,r}$. Moreover, by Proposition \ref{compapproxelliptic} (b), 
$$
\|\dell_{x_i}v_{\veps,r}\|_{L^\infty}=\|w_{i,\veps,r}\|_{L^\infty}\leq\|\dell_{x_i}g\|_{L^\infty}.
$$

A similar argument shows that for each multiindex $\alpha\in \mathbb{N}^N$, $D^\alpha v_{\veps,r}=B_{\veps}^{\mu_r}[D^\alpha g]$, and hence belongs to $C_\textup{b}(\R^N)$.
\end{proof}

\begin{corollary}\label{g+}
Assume \eqref{muassumption} and $g\in C_\textup{b}(\R^N)$. If $(g)^+ \in L^1(\R^N)$, then $(B_\veps^{{\mu}_r}[g])^+\in L^1(\R^N)$.
\end{corollary}

\begin{proof}
Note that $g\in C_\textup{b}(\R^N)$ implies that $(g)^+\in C_\textup{b}(\R^N)$. By Proposition \ref{existapproxelliptic} (a) and (c), and the assumption on $(g)^+$, we have that have that $B_\veps^{{\mu}_r}[g]\in C_\textup{b}(\R^N)$ and $B_\veps^{{\mu}_r}[(g)^+] \in L^1(\R^N)\cap C_\textup{b}(\R^N)$ are the unique classical solutions of \eqref{approxelliptic} with right-hand sides $g, (g)^+$, respectively. Proposition \ref{compapproxelliptic} (c) ensures that $B_\veps^{{\mu}_r}[(g)^+] \geq 0$ since $(g)^+\geq0$. In the same way, we get $B_\veps^{{\mu}_r}[-(g)^-] \in C_\textup{b}(\R^N) $ and $B_\veps^{{\mu}_r}[-(g)^-] \leq 0$. 

Adding the equations for $B_{\veps}^{\mu_r}[(g)^+]$ and $B_{\veps}^{\mu_r}[-(g)^-]$, and noting that $(g)^+-(g)^-=g \in C_\textup{b}(\R^N)$, we get
\[
\veps \Big(B_\veps^{{\mu}_r}[(g)^+]- B_\veps^{{\mu}_r}[(g)^-]  \Big)-\Levy^{\mu_r}\Big[ B_\veps^{{\mu}_r}[(g)^+]- B_\veps^{{\mu}_r}[(g)^-] \Big]= g.
\]
It follows that $B_\veps^{{\mu}_r}[g]=B_\veps^{{\mu}_r}[(g)^+]- B_\veps^{{\mu}_r}[(g)^-]$ by uniqueness. We conclude that $0\leq (B_\veps^{{\mu}_r}[g])^+\leq B_\veps^{{\mu}_r}[(g)^+]$, and thus, $(B_\veps^{{\mu}_r}[g])^+\in L^1(\R^N)$.
\end{proof}

\subsection{Results for the elliptic equation \eqref{elliptic}}
Now, we state and prove comparison, uniqueness and existence results
for classical solutions of \eqref{elliptic}. 
These results will be
obtained from the corresponding results for \eqref{approxelliptic} and
limit procedures. 

\begin{lemma}[Comparison]\label{classicalcompelliptic}
Assume \eqref{muassumption}, $g,\hat{g}\in L^\infty(\R^N)$, and $v_\veps, \vv_\veps\in C^2(\R^N)\cap L^\infty(\R^N)$ are solutions of \eqref{elliptic} with right-hand sides $g,\hat{g}$ respectively. If $g\leq \hat{g}$ a.e., then $v_\veps\leq \vv_\veps$ in $\R^N$.
\end{lemma}

\begin{proof} Note that $w=v_\veps-\vv_\veps$ solves $\veps w-\Levy^{\mu}[w]\leq
0$, and hence, also 
$$\veps w-\Levy^{\mu_r}[w]\leq
\|(\Levy^{\mu}-\Levy^{\mu_r})[w]\|_{L^\infty(\R^N)}.$$
By Proposition \ref{compapproxelliptic} (a), it then follows that
$$\veps \|(w)^+\|_{L^\infty(\R^N)}\leq \|(\Levy^{\mu}-\Levy^{\mu_r})[w]\|_{L^\infty(\R^N)}.$$
Assume for moment that $w\in C_\textup{b}^2(\R^N)$. Then  by
Remark \ref{propapproxelliptic} (b), 
$$\|(\Levy^{\mu}-\Levy^{\mu_r})[w]\|_{L^\infty(\R^N)}\to
0 \qquad\text{as}\qquad r\to 0^+,$$
and we conclude that $w\leq0$.

The general case follows by mollification:
$w_\delta=w\ast\omega_\delta$ (cf. \eqref{mollifierspace})
satisfies $\veps w_\delta-\Levy^{\mu}[w_\delta]\leq 0$ and hence by
the first part of the proof and properties of mollifiers,
$$w(x)\leq w_\delta(x)+|w(x)-w_{\delta}(x)|\leq 0 + o(1)
\qquad\text{as}\qquad\delta\to 0^+$$
for every $x\in\R^N$. The proof is complete.
\end{proof}

\begin{corollary}[Uniqueness]\label{classicaluniqueelliptic}
  Assume \eqref{muassumption}, and $g\in L^\infty(\R^N)$. Then there
  is at most one classical solution 
  $v_\veps\in C^2(\R^N)\cap L^\infty(\R^N)$ of \eqref{elliptic}.
\end{corollary}

\begin{proof}
If $g=\hat{g}$ a.e., then Lemma \ref{classicalcompelliptic} gives $v_\veps=\vv_\veps$ in $\R^N$.
\end{proof}

\begin{proposition}[Existence and Stability]\label{approxtoelliptic}
Assume \eqref{muassumption}, $g\in C_\textup{b}^\infty(\R^N)$, $\veps>0$. 
\begin{enumerate}[(a)]
\item There exists a unique classical solution $\B[g]=v_\veps\in C_\textup{b}^2$ of \eqref{elliptic}.
\item Any sequence $\{v_{\veps,r_n}\}_{n\in\N}$ of solutions of \eqref{approxelliptic} converges locally uniformly to $v_\veps=\B[g]$ of part (a) as $r_n\to0^+$.
\end{enumerate}
\end{proposition}

\begin{proof}
(a) Let $0<r_n\to0^+$ as $n\to\infty$, and let $v_n:=v_{\veps,r_n}\in C_\textup{b}^\infty(\R^N)$ be the unique solution of $\eqref{approxelliptic}$ given by Proposition \ref{existapproxelliptic} (d). Moreover, for all $n>0$,
\begin{equation*}
\begin{split}
\veps\|v_{n}\|_{L^\infty}\leq\|g\|_{L^\infty}, &\qquad\veps\|Dv_{n}\|_{L^\infty}\leq\|Dg\|_{L^\infty},\\ 
\veps\|D^2v_{n}\|_{L^\infty}\leq\|D^2g\|_{L^\infty}, &\qquad\veps\|D^3 v_{n}\|_{L^\infty}\leq\|D^3g\|_{L^\infty}.
\end{split}
\end{equation*}
The sequences $\{v_{n}\}_{n>0}$, $\{Dv_{n}\}_{n>0}$ and $\{D^2v_{n}\}_{n>0}$ are thus equibounded and equilipschitz. By Arzel\`a-Ascoli's theorem there exists a subsequence (still denoted by $v_{n}$, $Dv_{n}$ and $D^2v_{n}$) such that $(v_{n}, Dv_{n}, D^2v_{n})$ converges locally uniformly (and hence a.e.) as $n\to\infty$ to a limit $(\overline{v}, \overline{Dv}, \overline{D^2v})$ which is bounded and continuous. 

We check that $D\overline{v}=\overline{Dv}$ and $D^2\overline{v}=\overline{D^2v}$. Let $\alpha \in \mathbb{N}^N$ denote a multiindex. By Taylor's theorem
\begin{equation}\label{taylorapproxelliptic}
\begin{split}
v_n(y)=&\,v_n(x)+Dv_n(x)\cdot(y-x)+\frac{1}{2}D^2v_n(x)(y-x)\cdot(y-x)\\
&\,+\sum_{|\alpha|=3}\frac{3}{\alpha!}(y-x)^{\alpha}\int_0^1(1-t)^2D^\alpha v_n(x+t(y-x))\dd t.
\end{split}
\end{equation}
Since
$$
\left|\sum_{|\alpha|=3}\frac{3}{\alpha!}(y-x)^{\alpha}\int_0^1(1-t)^2D^\alpha v_n(x+t(y-x))\dd t\right|\leq\frac{1}{\veps}\|D^3 g\|_{L^\infty}\sum_{|\alpha|=3}\frac{|y-x|^\alpha}{\alpha !},
$$
we can take the locally uniform limit in \eqref{taylorapproxelliptic} as $n\to\infty$ to obtain that
\begin{equation*}
\begin{split}
\overline{v}(y)=\overline{v}(x)+\overline{Dv}(x)\cdot(y-x)+\frac{1}{2}\overline{D^2v}(x)(y-x)\cdot(y-x)+o(|y-x|^2)\quad\text{as}\quad y\to x.
\end{split}
\end{equation*}
By definition, it then follows that $D\overline{v}=\overline{Dv}$ and $D^2\overline{v}=\overline{D^2v}$.

We now go to the limit in \eqref{approxelliptic} as $r_n\to0^+$, and we may assume that $r_n<1$. In order to show the convergence, the nonlocal operator in \eqref{approxelliptic} will be written as
\begin{equation*}
\begin{split}
\Levy^{\mu_{r_n}}[v_n](x)=\Levy_1^{\mu_{r_n}}[v_n](x)+\int_{|z|>1}\lef v_n(x+z)-v_n(x)\rig\dd\mu(z),
\end{split}
\end{equation*}
with
$$
\Levy_1^{\mu_{r_n}}[v_n](x):=\int_{|z|\leq1}\lef v_n(x+z)-v_n(x)-z\cdot Dv_n(x)\rig\dd\mu_{r_n}(z).
$$

By the triangle inequality and Lemma \ref{propnonlocal} (a),
\begin{equation*}
\begin{split}
&\left|\Levy_1^{\mu_{r_n}}[v_n](x)-\Levy_1^{\mu}[\overline{v}](x)\right|\\
&\leq\left|\Levy_1^{\mu_{r_n}}[v_n-\overline{v}](x)\right|+\left|(\Levy_1^{\mu_{r_n}}-\Levy_1^{\mu})[\overline{v}](x)\right|\\
&\leq \frac{1}{2}\max_{|z|\leq1}\left|D^2v_n(x+z)-D^2\overline{v}(x+z)\right|\int_{|z|\leq1}|z|^2\dd\mu(z)\\
&\quad+\frac{1}{2}\max_{|z|\leq1}\left|D^2\overline{v}(x+z)\right|\int_{|z|\leq1}|z|^2\mathbf{1}_{|z|\leq {r_n}}\dd \mu(z).
\end{split}
\end{equation*}
So, the local uniform convergence and Lebesgue's dominated convergence theorem ensures that $\left|\Levy_1^{\mu_{r_n}}[v_n](x)-\Levy_1^{\mu}[\overline{v}](x)\right|\to0$ as ${r_n}\to0^+$ for all $x\in \R^N$. The remaining term in the nonlocal operator also converges by Lebesgue's dominated convergence theorem:
$$
\int_{|z|>1}\lef v_n(x+z)-v_n(x)\rig\dd\mu(z)\to\int_{|z|>1}\lef\overline{v}(x+z)-\overline{v}(x)\rig\dd\mu(z)\qquad\text{as}\qquad r_n\to0^+.
$$ 
Sending ${r_n}\to0^+$ in \eqref{approxelliptic} then shows that $\overline{v}$ solves \eqref{elliptic}. Moreover, the limit is unique by Corollary \ref{classicaluniqueelliptic}.
\medskip

\noindent
(b) In fact, part (a) shows that all limit points of the sequences $\{v_{\veps,r_n}\}_{n\in\N}$ coincide by uniqueness (see Corollary \ref{classicaluniqueelliptic}). By Proposition \ref{existapproxelliptic} (d), every sequence is bounded, and hence, the whole sequence converge locally uniformly to the solution of \eqref{elliptic} as $r_n\to0^+$.
\end{proof}

\begin{proposition}\label{l1+estimate}
Assume \eqref{muassumption}, $g\in C_\textup{b}^\infty(\R^N)$, $\veps>0$, and $v_\veps=\B[g]$. If $(g)^+\in L^1(\R^N)$, then
\begin{equation*}
\veps\int_{\R^N} (v_\veps)^+\dd x  \leq  \int_{\R^N} (g)^+\dd x.
\end{equation*}
\end{proposition}

\begin{proof}
By Proposition \ref{existapproxelliptic} (d), for any $r>0$, there exists a unique function $v_{\veps,r}\in C_\textup{b}^\infty(\R^N)$ such that
\begin{equation*}
\veps v_{\veps,r}(x) - \Levy^{\mu_r}[v_{\veps,r}](x)= g(x)\qquad\text{in}\qquad \R^N.
\end{equation*} 

Consider $\mathcal{X}\in C_\textup{c}^\infty(\R^N)$ such that $0\leq\mathcal{X}$,
\begin{equation*}
\mathcal{X}(x)=\begin{cases}
1 & |x|\leq1\\
0 & |x|>2
\end{cases},
\end{equation*}
and define $\mathcal{X}_R(x)=\mathcal{X}(\frac{x}{R})$ for $R>0$. Then for every $r>0$, by Proposition \ref{existapproxelliptic} (d), there exists a function $u_R\in C_\textup{b}^\infty(\R^N) $ such that 
\begin{equation}\label{eqaux13}
\veps u_R(x) - \Levy^{\mu_r}[u_R](x)= g(x)\mathcal{X}_R(x) \qquad \text{for all}\qquad x\in \R^N.
\end{equation}

Let $\zeta_\delta:\R\to\R_+$ be a smooth approximation of the $\sgn^+$ function. More precisely, $\zeta_\delta(x)=0$ for $x\leq 0$, $\zeta_\delta'(x)\geq0$ and $0<\zeta_\delta(x)\leq1$ for $x>0$. Since $0\leq(g\mathcal{X}_R)^+\leq (g)^+\in L^1(\R^N)$, $(u_R)^+\in L^1(\R^N)$ by Corollary \ref{g+}, and
\begin{equation*}
\begin{split}
\left| \int_{\R^N}u_R \zeta_\delta(u_R)\dd x \right|&\leq \|(u_R)^+\|_{L^1}\|\zeta_\delta(u_R)\|_{L^\infty}\\
\left| \int_{\R^N}g\mathcal{X}_R \zeta_\delta(u_R)\dd x \right|&\leq \|g\|_{L^\infty}  \int_{|x|\leq2R} \dd x.
\end{split}
\end{equation*}
Then by \eqref{eqaux13}, $\left|\int_{\R^N}\Levy^{\mu_r}[u_R]\zeta_\delta(u_R)\dd x\right|<\infty$, and we may multiply \eqref{eqaux13} by $\zeta_\delta$ and integrate over $\R^N$ to find that
\begin{equation*}
 \veps\int_{\R^N} u_{R}\zeta_\delta(u_{R}) \dd x =  \underbrace{\int_{\R^N}   \Levy^{\mu_r} [u_{R}] \zeta_\delta(u_{R}) \dd x}_{I_r} + \int_{\R^N} g\mathcal{X}_R \zeta_\delta(u_{R}) \dd x.
\end{equation*}
So, Lemma \ref{Stroock-Var} and Remark \ref{stroockwelldef} gives that $I_r\leq0$ and hence
$$
\veps\int_{\R^N}u_{R} \zeta_\delta(u_{R})\dd x\leq\int_{\R^N}g\mathcal{X}_R \zeta_\delta(u_{R})\dd x\leq \int_{\R^N} (g)^+ \dd x.
$$

Letting $\zeta_\delta(u_{R})\to \sgn^+(u_{R})$ as $\delta\to0^+$ in the above inequality (using Fatou's lemma on the left-hand side since $u_{R} \zeta_\delta(u_{R})\geq0$) yields
\begin{equation}\label{aux15}
\veps\int_{\R^N}(u_R)^+\dd x\leq\int_{\R^N}(g)^+\dd x.
\end{equation}

We note that the sequence $\{u_R\}_{R>0}$ is equibounded and equilipschitz since Proposition \ref{existapproxelliptic} (d) gives
\begin{equation*}
\begin{split}
\|u_R\|_{L^\infty} \leq &\,\frac{1}{\veps}\|g\|_{L^\infty}\\
\|D u_R\|_{L^\infty} \leq &\,\frac{1}{\veps}\|D\left(g \mathcal{X}_R\right)\|_{L^\infty}\leq\frac{1}{\veps}\|D g \|_{L^\infty}+ \frac{1}{\veps}\mathcal{O}\left(\frac{1}{R}\right).
\end{split}
\end{equation*}
Hence, by Arzel\'a-Ascoli, $u_R\to u$ as $R\to\infty$ locally uniformly in $\R^N$ (and thus a.e. in $\R^N$). Sending $R\to\infty$ in \eqref{eqaux13} shows that $u=v_{\veps,r}$, that is, the unique solution of \eqref{elliptic} given by Proposition \ref{existapproxelliptic} (d). Furthermore, we can send $R \to \infty$ in \eqref{aux15} (again using Fatou's lemma) to obtain
\[
\veps\int_{\R^N}(v_{\veps,r})^+\dd x\leq\int_{\R^N}(g)^+\dd x.
\]

By Fatou's lemma and Proposition \ref{approxtoelliptic} (b), we can let $r_n\to0^+$ in the above estimate to get
$$
\veps\int_{\R^N}(v_{\veps})^+\dd x\leq\int_{\R^N}(g)^+\dd x,
$$
where $v_\veps$ is the classical solution of \eqref{elliptic}. 
\end{proof}
 
\begin{corollary}\label{classicalsolnL1contr}
Assume \eqref{muassumption}, $g\in C_\textup{b}^\infty(\R^N)$, $\veps>0$, and $v_\veps=\B[g]$.
\begin{enumerate}[(a)]
\item If $(g)^-\in L^1(\R^N)$, then $\veps\int_{\R^N}(v_{\veps})^-\dd x\leq\int_{\R^N}(g)^-\dd x$.
\item If $g\in L^1(\R^N)$, then $\veps\int_{\R^N}|v_{\veps}|\dd x\leq\int_{\R^N}|g|\dd x$
\end{enumerate}
\end{corollary}

\begin{proof}
(a) Note that $(g)^-\in L^1(\R^N)$ implies that $(-g)^+\in L^1(\R^N)$. Since $\B[-g]=-\B[g]$, we have by Proposition \ref{l1+estimate} that
$$
\veps\int_{R^N}(\B[g])^-\dd x=\veps\int_{\R^N}(\B[-g])^+\dd x\leq\int_{\R^N}(-g)^+=\int_{\R^N}(g)^-\dd x.
$$
\smallskip

\noindent
(b) Follows by noting that $(v_\veps)^++(v_\veps)^-=|v_\veps|$.
\end{proof}

Below, we collect the main results for \eqref{elliptic}.

\begin{theorem}\label{propertiesofB1}
Assume \eqref{muassumption}, $g\in L_{\textup{loc}}^1(\R^N)$, and $v_\veps\in L_{\textup{loc}}^1(\R^N)$ is a distributional solution of \eqref{elliptic}.
\begin{enumerate}[(a)]
\item If $(g)^+\in L^1(\R^N)$, then
$$
\veps\int_{\R^N}(v_\veps)^+\dd x\leq \int_{\R^N}(g)^+\dd x.
$$
\item If $g\geq 0$ a.e. on $\R^N$, then $v_\veps \geq 0$ a.e. on $\R^N$. 
\end{enumerate}
\end{theorem}

\begin{proof}
(a) Let $\omega_\delta\in C_\textup{c}^\infty(\R^N)$ be defined in \eqref{mollifierspace}, and let $v_{\veps, \delta}=v_\veps\ast \omega_\delta\in C_\textup{b}^\infty(\R^N)$. By assumption, 
$$
\veps\int_{\R^N} v_\veps \psi \dd y-\int_{\R^N} v_\veps\Levy^{\mu} [\psi] \dd y= \int_{\R^N} g \psi \dd y
$$
for all $\psi\in C_\textup{c}^\infty(\R^N)$. Taking $\psi(y)=\omega_\delta(x-y)$ for $x\in \R^N$, we get the pointwise equation
\begin{equation*}
\begin{split}
\veps v_{\veps,\delta}-\Levy^\mu[v_{\veps,\delta}]=g\ast\omega_\delta\qquad\text{in}\qquad \R^N.
\end{split}
\end{equation*}

Note that $0\leq\left( g\ast\omega_\delta\right)^+\leq (g)^+\ast\omega_\delta\in L^1(\R^N)$ (see e.g. Lemma 5.1 in \cite{EnJa14} ), so Proposition \ref{l1+estimate} gives
\begin{equation*}
\veps\int_{\R^N}(v_{\veps,\delta})^+\dd x\leq \int_{\R^N}\left(g\ast\omega_\delta\right)^+\dd x. 
\end{equation*}
Then by Fatou's lemma
\begin{equation*}
\begin{split}
\veps\int_{\R^N}\liminf_{\delta\to0^+}(v_{\veps,\delta})^+\dd x\leq\liminf_{\delta\to0^+}\veps\int_{\R^N}(v_{\veps,\delta})^+\dd x \leq \liminf_{\delta\to0^+}\int_{\R^N}(g)^+\ast\omega_\delta\dd x. 
\end{split}
\end{equation*}
Since $(\cdot)^+$ is continuous, $(g)^+\in L^1(\R^N)$, and $v_{\veps,\delta}\in L_{\textup{loc}}^1(\R^N)$, the properties of mollifiers yields
\begin{equation*}
\veps\int_{\R^N}(v_\veps)^+\dd x \leq \int_{\R^N}(g)^+\dd x.
\end{equation*}
\smallskip

\noindent(b) Note that $-v_\veps$ solves \eqref{elliptic} with right-hand side $-g$. If $-g\leq 0$ a.e. on $\R^N$, then $(-g)^+=0\in L^1(\R^N)$. By part (a), we deduce that $\veps\int_{\R^N}(-v_\veps)^+\dd x\leq 0$, and hence that $-v_\veps\leq 0$ a.e. on $\R^N$.
\end{proof}

We are now ready to prove our main theorem for the elliptic equation \eqref{elliptic}.

\begin{proof}[Proof of Theorem \ref{collectedresultselliptic}]
(a) By the assumptions and Proposition \ref{existapproxelliptic} (d), for every $r>0$, there exists a unique classical solution $v_{\veps,r}\in C_\textup{b}^\infty(\R^N)$ of \eqref{approxelliptic} satisfying
$$
\veps\|D^\alpha v_{\veps,r}\|_{L^\infty}\leq\|D^\alpha g\|_{L^\infty} \qquad\text{for all}\qquad \alpha\in\mathbb{N}^N.
$$

An Arzel\`a-Ascoli argument as in the proof of Proposition \ref{approxtoelliptic} (in this case combined with a diagonal extraction argument), shows the existence of classical solutions $v_\veps\in C_\textup{b}^\infty(\R^N)$ of \eqref{elliptic} satisfying
$$
\veps\|D^\alpha v_{\veps}\|_{L^\infty}\leq\|D^\alpha g\|_{L^\infty} \qquad\text{for all}\qquad \alpha\in\mathbb{N}^N.
$$
Moreover, Corollary \ref{classicaluniqueelliptic} ensures that the classical solutions $v_\veps$ are unique.
\medskip

\noindent
(b) {\em Existence of $L^1$-solutions:} Let $\delta>0$, $g_\delta=g\ast
\omega_\delta\in C_\textup{b}^\infty(\R^N)$, where $\omega_\delta$ is defined
by \eqref{mollifierspace}, and $v_{\veps,\delta}\in C_\textup{b}^\infty(\R^N)$ be
the solution of \eqref{elliptic} with $g_\delta$ as right hand side. By Remark
\ref{propapproxelliptic} (a), a difference of solutions is also a
solution, and then by Corollary \ref{classicalsolnL1contr} (b),
$$
\veps\|v_{\veps,\delta_1}-v_{\veps,\delta_2}\|_{L^1}\leq\|g_{\delta_1}-g_{\delta_2}\|_{L^1}\qquad\text{for
  every}\qquad\delta_1,\delta_2>0.
$$
Hence, $\{v_{\veps,\delta}\}_{\delta>0}$ is Cauchy and there exists
$v_\veps\in L^1(\R^N)$ such that
$\|v_{\veps,\delta}-v_\veps\|_{L^1}\to0$ as $\delta\to0^+$. 

Since $v_{\veps,\delta}$ satisfies \eqref{elliptic} with right-hand side $g_\delta$,
$$
\veps\int_{\R^N}v_{\veps,\delta}\psi\dd x - \int_{\R^N}v_{\veps,\delta}\Levy^\mu[\psi]\dd x=\int_{\R^N}g_\delta \psi\dd x \qquad\text{for all}\qquad \psi\in C_\textup{c}^\infty(\R^N),
$$
and since $v_{\veps,\delta}, g_\delta\to v_\veps, g$ in $L^1(\R^N)$ as $\delta\to0^+$, we send $\delta\to0^+$ and find that $v_\veps$ is an $L^1$-distributional solution of \eqref{elliptic}.

{\em Uniqueness:} Note that $L^1\subset L_{\textup{loc}}^1$. Consider two distributional solutions $v_\veps, \vv_\veps$ of \eqref{elliptic} with right-hand sides $g, \hat{g}\in L^1(\R^N)$. If $g-\hat{g}=0$ a.e., then $v_\veps-\vv_\veps=\B[g-\hat{g}]=0$ by Theorem \ref{propertiesofB1} (b).

{\em $L^1$-estimate:} By the assumptions, we can take $v_\veps\in L^1(\R^N)\subset L_{\textup{loc}}^1(\R^N)$ and $g\in L^1(\R^N)$. Then Theorem \ref{propertiesofB1} (a) gives
$$
\veps\|(v_\veps)^+\|_{L^1}\leq \|(g)^+\|_{L^1}.
$$
A similar argument as in the proof of Corollary \ref{classicalsolnL1contr} concludes the proof.
\medskip

\noindent
(c) {\em Existence of $L^\infty$-solutions:} Proposition
\ref{existapproxelliptic} (b) ensures that there exists a unique
a.e. solution $v_{\veps, r}\in L^\infty(\R^N)$  of $$\veps
v_{\veps,r}-\Levy^{\mu_r}[v_{\veps,r}]=g,$$ and $\veps
\|v_{\veps,r}\|_{L^\infty}\leq \|g\|_{L^\infty}$. Then, by Alaoglu's theorem there exists $\overline{v_\veps}\in L^\infty(\R^N)$ such that, up to a subsequence, $v_{\veps,r_n}\stackrel{*}{\rightharpoonup}\overline{v_\veps} $  in $L^\infty(\R^N)$ as $r_n\to 0^+$. That is,
\begin{equation*}
\lim_{r_n\to0^+}\int_{\R^N} v_{\veps,r_n} \psi \dd x=\int_{\R^N} \overline{v_{\veps}} \psi \dd x \qquad \text{for all}\qquad \psi \in L^1(\R^N). 
\end{equation*}

To finish the existence proof, we need to show that $\overline{v_\veps}$ is in fact a distributional solution of \eqref{elliptic}. 
Consider a function $\gamma \in C_\textup{c}^\infty(\R^N)$, then
$\gamma, \Levy^{\mu_{r_n}}[\gamma], \Levy^{\mu}[\gamma]\in L^1(\R^N)$
(see Lemma \ref{propnonlocal} (b)). Since $v_{\veps,r_n}$ is a
pointwise a.e. solution  and $v_{\veps,r_n},\Levy^{\mu_{r_n}}[
v_{\veps,r_n}] \in L^\infty(\R^N)$, we have by integration and
self-adjointness of $\Levy^{\mu_{r_n}}$ (cf. Lemma \ref{propnonlocal}
and Remark \ref{nonlocalwelldefined} (b)) 
 that 
\[
\veps \int_{\R^N}  v_{\veps,r_n}\gamma \dd x- \int_{\R^N}  v_{\veps,r_n}\Levy^{\mu_{r_n}}[\gamma] \dd x=  \int_{\R^N} g \gamma \dd x \qquad \text{for all}\qquad \gamma \in C_\textup{c}^\infty(\R^N).
\]
The weak* $L^\infty$-convergence ensures that 
\[
\lim_{r_n\to0^+}\int_{\R^N} v_{\veps,r_n} \gamma \dd x = \int_{\R^N} \overline{v_{\veps}} \gamma \dd x \qquad\text{for all}\qquad \gamma\in C_\textup{c}^\infty(\R^N). 
\]
By Remark \ref{propapproxelliptic} (b), we have that, for any $\gamma \in C_\textup{c}^\infty (\R^N)$, $\Levy^{\mu_{r_n}}[\gamma]\to\Levy^{\mu}[\gamma]$ in $L^1(\R^N)$ as $r_n\to0^+$. Then, since $\|v_{\veps,{r_n}}\|_{L^\infty}\leq \frac{1}{\veps}\|g\|_{L^\infty}$, we get as $r_n\to0^+$
\begin{equation*}
\begin{split}
&\int_{\R^N} v_{\veps,r_n}\Levy^{\mu_{r_n}}[ \gamma] \dd x\\
& = \int_{\R^N} v_{\veps,r_n}\Levy^{\mu}[ \gamma] \dd x +\int_{\R^N} v_{\veps,r_n}(\Levy^{\mu_{r_n}}-\Levy^{\mu})[ \gamma] \dd x \to\int_{\R^N} \overline{v_{\veps}} \Levy^{\mu}[\gamma] \dd x,
\end{split}
\end{equation*}
for all $ \gamma\in C_\textup{c}^\infty(\R^N)$. This shows that 
\[
\veps \int_{\R^N}  \overline{v_{\veps}}\gamma \dd x- \int_{\R^N} \overline{v_{\veps}} \Levy^{\mu}[\gamma] \dd x=  \int_{\R^N} g \gamma \dd x \qquad \text{for all}\qquad \gamma \in C_\textup{c}^\infty(\R^N),
\]
that is, $\overline{v_{\veps}}$ is an $L^\infty$-distributional solution of \eqref{elliptic}.

{\em Uniqueness:} Note that $L^\infty\subset L_{\textup{loc}}^1$. Consider two distributional solutions $v_\veps, \vv_\veps$ of \eqref{elliptic} with right-hand sides $g, \hat{g}\in L^\infty(\R^N)$. If $g-\hat{g}=0$ on $\R^N$, then $v_\veps-\vv_\veps=0$ by Theorem \ref{propertiesofB1} (b).

{\em $L^\infty$-estimate:} Observe that
$\pm\frac{1}{\veps}\|g\|_{L^\infty}\in L^\infty(\R^N)\subset
L_{\textup{loc}}^1(\R^N)$ are distributional solutions of
\eqref{elliptic} with $\pm\|g\|_{L^\infty}$ as right-hand
sides. Moreover, $-\|g\|_{L^\infty}\leq g\leq\|g\|_{L^\infty}$. Then
Theorem \ref{propertiesofB1} (b) gives
$|v_\veps|\leq\frac{1}{\veps}\|g\|_{L^\infty}$.
\end{proof}

This section is concluded by a proof of the self-adjointness of $\B$.

\begin{proof}[Proof of Lemma \ref{symmetryB}]
Let $f_\delta=f\ast\omega_\delta$ and $g_\delta=g\ast\omega_\delta$
where $\omega_\delta$ is defined by \eqref{mollifierspace}. Then
$f_\delta\in C_\textup{b}^\infty(\R^N)\cap W^{2,1}(\R^N)$ and $g\in C_\textup{b}^\infty(\R^N)$, and then by
Theorem \ref{collectedresultselliptic} (a)--(c),
$\B[f_\delta]\in C_\textup{b}^\infty(\R^N)\cap W^{2,1}(\R^N)$, $\B[g_\delta]\in
C_\textup{b}^\infty(\R^N)$, and 
\begin{align*}
\veps\B[f_\delta]-\Levy^\mu[\B[f_\delta]]&=f_\delta(x)
\qquad\text{in}\qquad \R^N,\\
\veps\B[g_\delta]-\Levy^\mu[\B[g_\delta]]&=g_\delta(x)
\qquad\text{in}\qquad \R^N.
\end{align*}
By the regularity and integrability of the terms of the equations (cf. Lemma
\ref{propnonlocal}), we may multiply the first equation by
$\B[g_\delta]$ and the second by $\B[f_\delta]$, and then integrate
both equations in $x$ over $\R^N$. By self-adjointness of $\Levy^\mu$
(Lemma \ref{propnonlocal} (c)), we then find that
\begin{equation}\label{epsdeltasymmetry}
\begin{split}
\,\int_{\R^N}f_\delta\B[g_\delta]\dd x=&\,\int_{\R^N}
\big(\veps\B[f_\delta]-\Levy^\mu[\B[f_\delta]]\big)\B[g_\delta]\dd x\\
=&\,\int_{\R^N} \big(\veps\B[g_\delta]-\Levy^\mu[\B[g_\delta]]\big)\B[f_\delta] \dd x\\
=&\int_{\R^N}g_\delta\B[f_\delta]\dd x
\end{split}
\end{equation}

To pass to the limit as $\delta\to 0^+$, we first subtract equations to find that 
\[
\veps \B[f]-\veps \B[f_\delta] -\Levy^\mu\big[\B[f]-\B[f_\delta]\big]=f-f_\delta \, \qquad\text{in}\qquad \mathcal{D}'(\R^N),  
\]
and hence by Theorem \ref{collectedresultselliptic} (b), linearity, and properties
of mollifiers,
\[
\veps\|\B[f]-\B[f_\delta] \|_{L^1}=\veps\|\B[f-f_\delta] \|_{L^1}\leq
\|f-f_\delta\|_{L^1} \to 0 \qquad\text{as}\qquad \delta\to 0^+.\]
On the other hand, by Theorem \ref{collectedresultselliptic} (b) and 
(c), and properties of the mollifiers,
$$\veps\|\B[f_\delta]\|_{L^1}\leq
\|f\|_{L^1},\quad \veps\|\B[f_\delta]\|_{L^\infty}\leq
\|f\|_{L^\infty}, \quad \veps\|\B[g_\delta]\|_{L^\infty}\leq
\|g\|_{L^\infty}, $$
and $g_\delta\to g$ a.e. Using $L^1$-convergence for the $f$-terms and the dominated convergence theorem
for the $g$-terms, we may send $\delta\to0^+$ in \eqref{epsdeltasymmetry} to
get the result. 
\end{proof}

\appendix
\section{Technical results}\label{appendix}

\subsection{Proof of Liouville type of theorem}

\begin{proof}[Proof of Theorem \ref{Liouville}] 
By the definition of distributional solutions,
$$
\int_{\R^N}v(y)\Levy^\mu[\psi](y)\dd y=0 \qquad\text{for all}\qquad \psi\in C_\textup{c}^\infty(\R^N).
$$
Let $x\in\R^N$, take $\psi(y)=\omega_\delta(x-y)$, where
$\omega_\delta$ is defined in \eqref{mollifierspace}, and let $v_\delta=v\ast\omega_\delta\in
C_0(\R^N)\cap C_\textup{b}^\infty(\R^N)$. By Lemma \ref{propnonlocal} (b), $\Levy^\mu[\psi]\in
L^1$, and we may use Fubini's theorem to see that 
\begin{equation}\label{LiouSmo}
\Levy^\mu[v_\delta](x)=0 \qquad\text{for every}\qquad x\in\R^N.
\end{equation}

Assume that there exists an $\tilde{x}\in\R^N$ such that
$v_\delta(\tilde{x})\neq0$. We only consider the case $v_\delta(\tilde{x})>0$;
the proof in the other case is similar. Then
$M:=\sup_{x\in\R^N}v_\delta>0$, and since $v_\delta\in C_0(\R^N)$
there exists an $x_0$ such that 
$$0<M=\max_{x\in\R^N}v_\delta=v_\delta(x_0).$$ 
By equation \eqref{LiouSmo} and Lemma \ref{propnonlocal} (b), we then find that
\begin{equation*}
\begin{split}
0=\Levy^{\mu}[v_\delta](x_0)=&\,\int_{|z|\leq\kappa}\lef v_\delta(x_0+z)-v_\delta(x_0)-z\cdot Dv_\delta(x_0)\rig\dd\mu(z)\\
&\,+\int_{|z|>\kappa}\lef v_\delta(x_0+z)-v_\delta(x_0)\rig\dd\mu(z)\\
\leq&\,\|D^2v_\delta\|_{L^\infty\left(\overline{B}(x_0,\kappa)\right)}\int_{|z|\leq\kappa}|z|^2\dd\mu(z)\\
&\,+\int_{|z|>\kappa}\lef v_\delta(x_0+z)-M\rig\dd\mu(z).\\
\end{split}
\end{equation*}
Take any $z_0\in \textup{supp}\; \mu$. By definition, $z_0\neq0$ and
$\mu(B(z_0,r))>0$ for all $r>0$. Hence we can take $r,\kappa\in(0,1)$ small
enough such that
$$B(z_0,r)\cap \{z\in\R^N:|z|\leq\kappa\}=\emptyset.$$
Since $\kappa<1$, $v_\delta(x_0+z)-M\leq0$, and $B(z_0,r)\subset
\{z\in\R^N:|z|>\kappa\}$, the above inequality yields that
$$
\int_{B(z_0,r)}\lef v_\delta(x_0+z)-M\rig\dd\mu(z)\geq\,-\|D^2v_\delta\|_{L^\infty\left(\overline{B}(x_0,1)\right)}\int_{|z|\leq1}|z|^2\mathbf{1}_{|z|\leq\kappa}\dd\mu(z).
$$
Taking the limit as $\kappa\to0^+$ using Lebesgue's dominated convergence theorem (the integrand is dominated by $|z|^2$ which is integrable by \eqref{muassumption}) gives
$$
\int_{B(z_0,r)}\lef v_\delta(x_0+z)\dd\mu(z)- M\mu(B(z_0,r))\rig \dd \mu(z)\geq0.
$$
Then by continuity, $v_\delta(x_0+z)=v_\delta(x_0+z_0)+\lambda(|z-z_0|)$
in $B(z_0,r)$ for some modulus of continuity $\lambda$, and we find
that
$$
v_\delta(x_0+z_0)+\lambda(r)\geq\frac{1}{\mu(B(z_0,r))}\int_{B(z_0,r)}v_\delta(x_0+z)\dd\mu(z)\geq M.
$$
Hence, we may send $r\to0^+$ and get that $v_\delta(x_0+z_0)\geq M$. It follows that
$v_\delta(x_0+z_0)= M$ since $M$ is the maximum of $v_\delta$. 

Repeating the above argument, we find that $v_\delta(x_0+nz_0)=M$ for
every $n\in\N$, and thus
$$\limsup_{n\to\infty}v_\delta(x_0+nz_0)\geq M>0.$$
 This is a contradiction since $\lim_{|x|\to\infty}v_\delta(x)=0$. So,
 we conclude that $v_\delta(x)=0$ for every $x\in \R^N$. 

By the properties of mollifiers, $v_\delta\to v$ locally uniformly in
$\R^N$ as $\delta\to0^+$, and hence it follows that also $v(x)=0$ for every $x\in\R^N$.
\end{proof}

\subsection{Proof of a measure theory result}

\begin{proof}[Proof of Lemma \ref{lemmameasure}]
Remember that we defined
\begin{equation*}
M_{1}(\As,\Bs)=\int_{\As}\left( \int_{\Bs-z}\dd\nu(x)\right)\dd z=\int_{\As}\nu(\Bs-z)\dd z,
\end{equation*}
\begin{equation*}
M_2(\As,\Bs)=\int_{\Bs}\left( \int_{\As-z}\dd\nu(x)\right)\dd z=\int_{\Bs}\nu(\As-z)\dd z,
\end{equation*}
and that we want to show that $M_1(\As, \Bs)=M_2(\As, \Bs)$.

Consider the set $\Cs\subset \R^{2N}$ defined as
\[
\Cs=\{(x,z)\in \R^{2N}: z\in \As, \ \ x\in \Bs-z\}.
\]
Furthermore, define the sets 
\[
\Ss=\{x=x_{\Bs}-x_{\As}: x_{\As}\in \As, \ \ x_{\Bs} \in \Bs\}=\bigcup_{x_{\As}\in \As}(\Bs-x_{\As}),
\]
\[
\Gs_x=\{z\in \As: x\in \Bs-z\}=\{z\in \As: z\in \Bs-x\}=\As\cap(\Bs-x).
\]
Note that $\Cs$ can also be expressed as
\[
\Cs=\{(x,z)\in\R^{2N}: x\in \Ss, \ \ z\in \Gs_x\}.
\]

\begin{figure}[htp] 
\centering{
\includegraphics[scale=0.5]{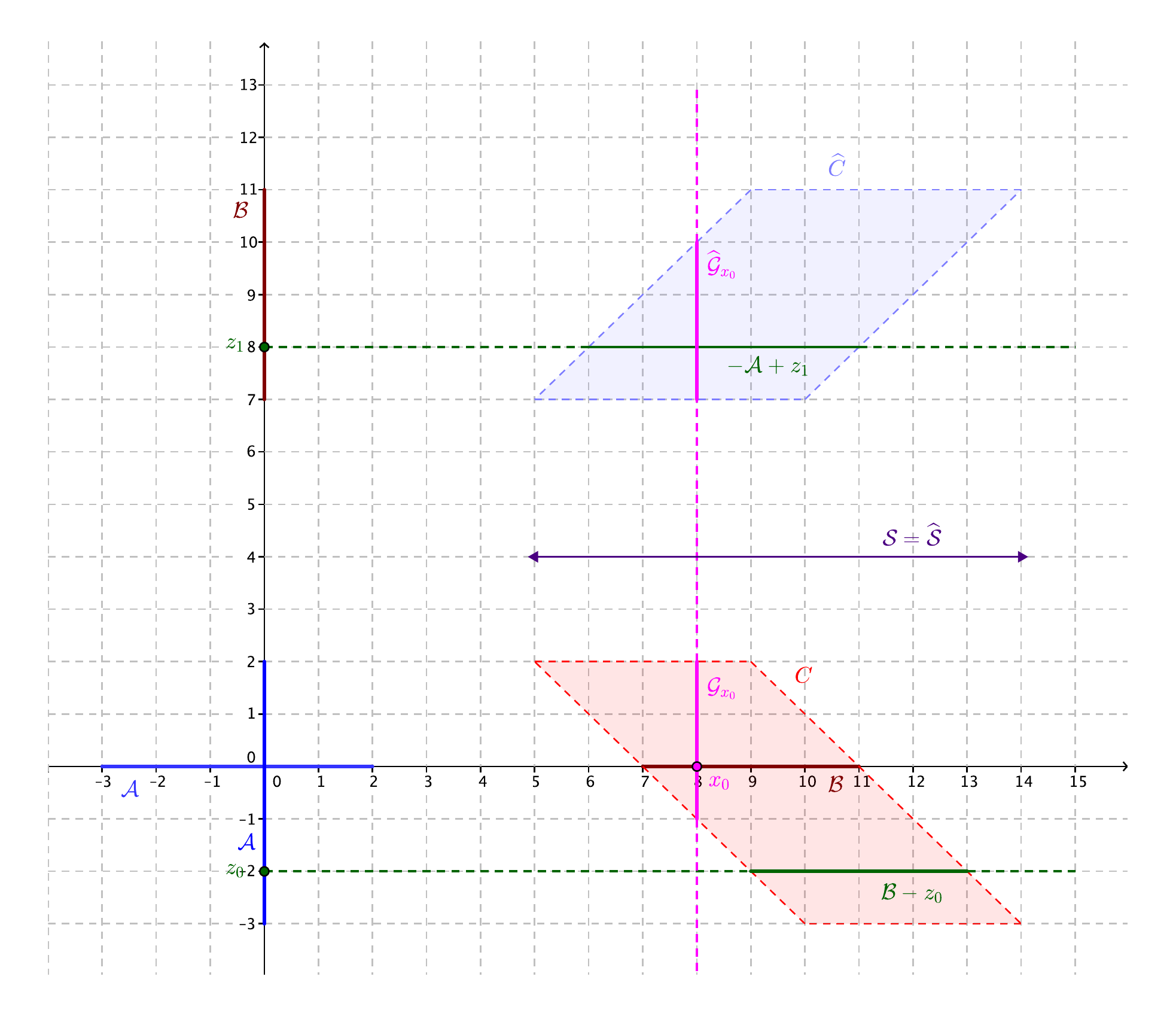}}
\caption{}
\label{explainingmeasure}
\end{figure}

Then
\begin{equation}\label{fubini}
\begin{split}
M_{1}(\As,\Bs)=&\int_{\R^N}\left( \int_{\R^N} \mathbf{1}_{\As}(z)\,\mathbf{1}_{\Bs-z}(x)\dd\nu(x)\right)\dd z=\int_{\R^N}\left( \int_{\R^N} \mathbf{1}_{\Cs}(x,z)\dd\nu(x)\right)\dd z\\
=&\int_{\R^N}\left( \int_{\R^N} \mathbf{1}_{\Cs}(x,z)\dd z\right)\dd\nu(x)=\int_{\R^N}\left( \int_{\R^N}\mathbf{1}_{\Ss}(x)\, \mathbf{1}_{\Gs_x}(z) \dd z\right)\dd\nu(x)\\
=&\int_{\Ss}\left( \int_{\Gs_x} \dd z\right)\dd\nu(x)=\int_{\Ss}|\Gs_x|\dd\nu(x),
\end{split}
\end{equation}
where the third equality follows by Tonelli's theorem (the tensor measure is a nonnegative Radon measure), and $|\cdot|$ denotes the Lebesgue measure on $\R^N$.

We can proceed in the same way to change the order of integration in the expression for $M_2(\As,\Bs)$, but first we make use of the symmetry of $\nu$ 
\[
M_2(\As,\Bs)=\int_{\Bs}\nu(\As-z)\dd z=\int_{\Bs}\nu(-\As+z)\dd z=\int_{\Bs}\left( \int_{z-\As}\dd\nu(x)\right)\dd z.
\]
Using the same technique we consider the sets,
\[
\widehat{\Cs}=\{(x,z)\in \R^{2N}: z\in \Bs, \ \ x\in -\As+z\},
\]
\[
\widehat{\Ss}=\{x=x_{\Bs}+x_{\As}: x_{\As}\in -\As, \ \ x_{\Bs} \in \Bs\}=\bigcup_{x_{\Bs}\in \Bs}(x_{\Bs}-\As),
\]
\[
\widehat{\Gs}_x=\{z\in \Bs: x\in -\As+z\}=\{z\in \Bs: z\in \As+x\}=\Bs\cap(\As+x).
\]
Second, we follow \eqref{fubini} to get
\begin{equation}\label{fubini2}
M_2(\As,\Bs)=\int_{\widehat{\Ss}}|\widehat{\Gs}_x|\dd\nu(x).
\end{equation}

Now, note that $\Ss=\widehat{\Ss}$. Moreover, $\Gs_x=\As\cap(\Bs-x)$ is just a translation of $\widehat{\Gs}_x=\Bs\cap(\As+x)$. Then $|\Gs_x|=|\widehat{\Gs}_x|$, since the Lebesgue measure is invariant under translations. (Consult Figure \ref{explainingmeasure} for a visual overview of the sets.)

Finally, we can conclude by \eqref{fubini} and \eqref{fubini2} that
\[
M_{1}(\As,\Bs)=\int_{\Ss}|\Gs_x|\dd\nu(x)=\int_{\widehat{\Ss}}|\widehat{\Gs}_x|\dd\nu(x)=M_2(\As,\Bs),
\]
which completes the proof.\end{proof}

\subsection*{Acknowledgements}

E. R. Jakobsen was supported by the grant Waves and Nonlinear
Phenomena (WaNP) from the Research Council of Norway. F. del Teso was
supported by the FPU grant AP2010-1843 and the grants MTM2011-24696 and MTM2014-52240-P from the Spanish Ministry of
Education, Culture and Sports, the BERC 2014-2017 program from the
Basque Government, and  BCAM Severo Ochoa excellence accreditation SEV-2013-0323 from Spanish Ministry of Economy and Competitiveness (MINECO). 
He would like to thank the Department of Mathematical
Sciences at NTNU for hosting him from September to
December 2014. We also thank Juan Luis V\'azquez for encouraging and
supporting the stay of the second author at NTNU where this paper was
started, and for useful comments and
discussions. Finally we thank the anonymous referee for
 his suggestions and remarks that helped us improve our paper.


\end{document}